\documentclass[twoside,11pt,preprint]{article}
\usepackage{jmlr2e}
\usepackage{mathtools, bm, bbm, mathrsfs, bigints}
\usepackage[capitalise]{cleveref}

\DeclareFontFamily{U}{mathx}{\hyphenchar\font45}
\DeclareFontShape{U}{mathx}{m}{n}{
      <5> <6> <7> <8> <9> <10>
      <10.95> <12> <14.4> <17.28> <20.74> <24.88>
      mathx10
      }{}
\DeclareSymbolFont{mathx}{U}{mathx}{m}{n}
\DeclareFontSubstitution{U}{mathx}{m}{n}
\DeclareMathSymbol{\bigtimes}{1}{mathx}{"91}
\usepackage{rotating} 
\usepackage{tikz, tikz-cd, tikz-dimline} 
\usetikzlibrary{arrows, decorations.pathreplacing, positioning, shapes,arrows.meta}
 
\usepackage{caption}
\usepackage{subcaption}
\usepackage{booktabs, multirow}
\usepackage{enumerate}
\usepackage{algpseudocode, subcaption}
\usepackage{enumitem}
\usepackage[ruled,linesnumbered]{algorithm2e}
\crefname{algocf}{Algorithm}{Algorithms}
\Crefname{algocf}{Algorithm}{Algorithms}
\SetKwInOut{Input}{input}
\SetKwInOut{Output}{output}
\SetKw{Throw}{throw}

\DeclareMathOperator{\PDAG}{PDAG}

\DeclareMathOperator{\DAG}{DAG}

\DeclareMathOperator{\PossDe}{PossDe}

\DeclareMathOperator{\De}{De}

\DeclareMathOperator{\An}{An}

\DeclareMathOperator{\Pa}{Pa}

\newcommand{\T}{\top}

\newcommand{\indep}{\mathrel{\text{\scalebox{1.07}{$\perp\mkern-10mu\perp$}}}}
\DeclareMathOperator{\var}{\mathrm{var}}
\DeclareMathOperator{\avar}{\mathrm{avar}}

\DeclareMathOperator{\cov}{\mathrm{cov}}
\DeclareMathOperator{\acov}{\mathrm{acov}}

\DeclareMathOperator{\diag}{\mathrm{diag}}
\DeclareMathOperator{\E}{\mathbb{E}}

\DeclareMathOperator{\Tr}{\mathrm{Tr}}
\newcommand*\dd{\mathop{}\!\mathrm{d}}

\newcommand{\distconvto}{\rightarrow_{d}}

\newcommand{\unif}{\mathrm{unif}}

\newcommand{\N}{\mathcal{N}}

\newcommand{\PD}[1]{\mathbb{R}_{\text{PD}}^{{#1} \times {#1}}}

\newcommand{\g}[1][G]{\mathcal{#1}}

\let \hat \widehat

\newcommand{\Do}{\mathrm{do}}
\DeclareMathOperator{\vect}{vec}

\usepackage{lastpage}
\ShortHeadings{Efficient Least Squares for Total Effects}{Guo and Perkovi\'c}
\firstpageno{1}
\jmlrheading{23}{2022}{1-\pageref{LastPage}}{3/21; Revised
9/21}{3/22}{21-0232}{F.~Richard Guo and Emilija Perkovi\'c}

\begin{document}

\title{Efficient Least Squares for Estimating Total Effects under Linearity and Causal Sufficiency}

\author{\name F.~Richard Guo \email ricguo@uw.edu \AND
		Emilija Perkovi\'c \email perkovic@uw.edu \\
		\addr Department of Statistics\\
       University of Washington\\
       Seattle, WA 98195-4322, USA}

\editor{Joris Mooij}

\maketitle

\begin{abstract}
Recursive linear structural equation models are widely used to postulate causal mechanisms underlying observational data. In these models, each variable equals a linear combination of a subset of the remaining variables plus an error term. When there is no unobserved confounding or selection bias, the error terms are assumed to be independent. We consider estimating a total causal effect in this setting. The causal structure is assumed to be known only up to a maximally oriented partially directed acyclic graph (MPDAG), a general class of graphs that can represent a Markov equivalence class of directed acyclic graphs (DAGs) with added background knowledge. We propose a simple estimator based on recursive least squares, which can consistently estimate any identified total causal effect, under point or joint intervention. We show that this estimator is the most efficient among all regular estimators that are based on the sample covariance, which includes covariate adjustment and the estimators employed by the joint-IDA algorithm. Notably, our result holds without assuming Gaussian errors.
\end{abstract}

\begin{keywords}
structural equation model, causal inference, semiparametric efficiency, partially directed acyclic graph, observational study
\end{keywords}

\section{Introduction} \label{sec:intro}

A linear structural equation model (SEM) specifies a causal mechanism underlying a set of variables \citep{bollen1989structural}. Each variable equals a linear combination of a subset of the remaining variables plus an error term. A SEM is \emph{associated} with a mixed graph, also known as a path diagram \citep{wright1921correlation,wright1934method}, which consists of both directed edges and bi-directed edges. A directed edge $i \rightarrow j$ represents that variable $i$ appears as a covariate in the structural equation defining variable~$j$. The equation for variable $j$ takes the form 
\begin{equation}
X_j = \sum_{i: i \rightarrow j} \gamma_{ij} X_i + \epsilon_j, \label{eqs:sem}
\end{equation}
where $\epsilon_j$ is an error term. Often, the errors are assumed to follow a multivariate normal distribution, but it need not be the case. A bi-directed edge $i \leftrightarrow j$ indicates that errors $\epsilon_i$ and $\epsilon_j$ are dependent, which is assumed when there exists an unobserved (i.e., latent) confounder between $i$ and $j$. The mixed graph is usually assumed to be \emph{acyclic}, i.e., the graph does not contain cycles made of directed edges. 

We focus on the setting when there is no unobserved confounder or selection bias, a condition also known as \emph{causal sufficiency}; see  \citet[Chap.~3]{spirtes2000causation} and \citet[Chap.~6]{pearl2009causality}. In this setting, all the error terms are assumed to be  \emph{mutually independent} and the mixed graph associated with the linear SEM is a \emph{directed acyclic graph} (DAG), often called a \emph{causal DAG}. Aside from being a statistical model for observational data, the linear SEM is also a causal model in the sense that it specifies the behavior of the system under interventions (see \cref{sec:sem-intervention}).  Therefore, the total causal effect of one treatment variable (point intervention) or several treatment variables (joint intervention) on some outcome varibles can be defined. 

The underlying causal DAG is usually unknown. In fact, linear SEMs associated with different DAGs may define the same \emph{observed} distribution \citep{drton2011global}. Without further assumptions on the error distributions, the underlying DAG can only be learned from observational data up to its Markov equivalence class, which can be uniquely represented by a completed partially directed acyclic graph (CPDAG) \citep{meek1995causal,andersson1997characterization}. Additional background knowledge, such as knowledge of certain causal relationships \citep{meek1995causal,fangida} or partial orderings \citep{tetrad1998}, restrictions on the error distributions \citep{shimizu2006linear,shimizu2011directlingam,shimizu2014lingam,hoyer08,peters2014identifiability}, and other assumptions \citep{hauserBuehlmann12,wang2017permutation,ernestroth2016,eigenmann17} can be used to further refine the Markov equivalence class of DAGs, resulting in representing the causal structure as a maximally oriented PDAG (MPDAG), which is a rather general class of graphs containing directed and undirected edges that subsumes DAGs and CPDAGs \citep{meek1995causal}. A given total causal effect is \emph{identified} given a graph, if it can be expressed as a functional of the observed distribution, which is the same for every DAG in the equivalence class. Recently, a necessary and sufficient graphical criterion for identification given an MPDAG has been shown by \citet{perkovic20}. In general, there may be more than one identifying functional.

Naturally, the next step is to develop estimators for an identified total effect with desirable properties. When the effect is unidentified, the reader is referred to IDA-type  \citep{maathuis2009estimating,nandy2017estimating} or enumerative \citep{guo2020minimal} approaches, which are beyond the scope of this paper. Among others, we consider the following desiderata. 
\begin{description}[itemsep=0mm,topsep=1ex,partopsep=1ex,parsep=1ex]
\item [{\bf Completeness.}] Can the estimator consistently estimate every identified effect, under either point or joint interventions? 
\item [{\bf Efficiency.}] Does the estimator achieve the smallest asymptotic (co)-variance compared to a reasonably large class of estimators? 
\end{description}
To the best of our knowledge, no estimator proposed in the literature fulfills both desiderata. Indeed, the commonly used covariate adjustment estimators \citep{pearl1993bayesian,shpitser2010validity,maathuis2015generalized,perkovic2015complete} do not exist for certain total effects under joint interventions \citep{nandy2017estimating,perkovic18,perkovic20}. Furthermore, when they exist, even with an optimal adjustment set chosen to maximize efficiency \citep{henckel2019graphical,rotnitzky2019efficient,witte2020efficient},  we will show in \cref{sec:numerical} that covariate adjustment can compare less favorably against a larger class of estimators. 

We propose an estimator that is based on simple recursive least squares, that affirmatively fulfills both desiderata. In particular, our proposed estimator achieves the efficiency bound among all regular estimators that only depend on the sample covariance; see \cref{sec:efficiency} for the precise definition of the class of estimators. Remarkably, our result holds regardless of the type of error distribution in the underlying linear SEM. Our method is implemented in the R \citep{R:2020} package \texttt{eff2} ({\small\url{https://cran.r-project.org/package=eff2}}), which stands for ``efficient effect'' (estimate).

Admittedly, our estimator can be less efficient when compared to an even larger class of estimators, such as the class of all regular estimators considered in standard semiparametric theory. A semiparametric efficient estimator, relative to all regular estimators, can in principle be constructed by computing the efficient influence function and employing estimation strategies such as one-step correction or targeted maximum likelihood estimation \citep{van2011targeted}. In fact, the semiparametric model we consider (see \cref{eqs:p-G}) is a generalized, multivariate location-shift regression model with additional conditional independence constraints; see also \citet[\S 5.1]{tsiatis2007semiparametric} and \citet[\S 4.3]{bickel1993efficient}. While it is theoretically possible to construct a semiparametric efficient estimator by firstly estimating the error score and then solving the associated estimating equations \citep[\S 7.8]{bickel1993efficient}, the resulting estimator tends to be too complicated and unstable for practical purposes unless the sample size is very large \citep[page 111]{tsiatis2007semiparametric}. On the other hand, despite the potential loss of efficiency, our least squares estimator is easily computed and numerically stable. Hence, our proposal can be viewed as a deliberate trade-off between optimality and practicality. 

The paper is organized as follows. In \cref{sec:lit-rev}, we review related work on efficient estimation of total effects in over-identified settings. In \cref{sec:prelim}, we introduce the preliminaries on linear structural equation models, causal graphs and the identification of total causal effects. The concept of bucket decomposition is introduced. In \cref{sec:block}, we introduce a block-recursive representation for the observational data and identify the total causal effect under such a representation. We first derive the proposed least squares estimator by finding the maximum likelihood estimator (MLE) under the assumption of Gaussian errors in \cref{sec:MLE}. We then prove the optimal efficiency of our proposed estimator under arbitrary error distributions in \cref{sec:efficiency}. Additional preliminaries, proofs and numerical results can be found in the Appendix.

 \section{Related work} \label{sec:lit-rev}
The statistical performance of an estimator of a total causal effect, in \emph{over-identified} settings, has recently received more attention; see, e.g, \citet{kuroki2003covariate,henckel2019graphical,witte2020efficient,gupta2020estimating,rotnitzky2019efficient,smucler2020efficient,kuroki2020variance}. Here, ``over-identified'' \citep{koopmans1950identification} refers to the fact that the total causal effect can be expressed as more than one functional of the (population) observed distribution, all of which coincide due to the additional conditional independence constraints obeyed by the observed distribution. For example, in the case where a total causal effect can be identified through covariate adjustment, usually there exists more than one valid adjustment set \citep{henckel2019graphical}. This is in contrast to the more traditional setting of causal inference, where the observed data distribution is nonparametric and is not expected to satisfy extra conditional independences.

Intuitively, the conditional independences in over-identified models can be exploited to maximize asymptotic efficiency; see, e.g., \citet{sargan1958estimation,hansen1982large} for early works in this direction. Under a linear SEM with independent errors, a total causal effect can be estimated via covariate adjustment as the least squares coefficient from the regression of the 
outcome on the treatment and adjustment variables. \citet{henckel2019graphical} recently showed that, under a linear SEM with independent errors, a valid adjustment set that minimizes asymptotic variance, also referred to as the \emph{optimal adjustment set}, can be  graphically characterized; see also \citet{witte2020efficient} for further properties of such an optimal set. This result was generalized by \citet{rotnitzky2019efficient} beyond linear SEMs: an optimal adjustment set is shown to always exist for point interventions, and a semiparametric efficient estimator is developed for this case.  Note that, while valid adjustment sets (called ``time-independent'' adjustment sets by \citet{rotnitzky2019efficient}) exist for point interventions \citep[Proposition 4.2]{perkovic20}, they may not exist for joint interventions \citep{nandy2017estimating,perkovic18,perkovic20}. 

Less is known about how to efficiently estimate the total causal effect of a joint intervention, at least in a generic fashion. For linear SEMs with independent errors, with the knowledge of the parents of the treatment variables in the underlying causal DAG, \citet{nandy2017estimating} considered two estimators for the joint-IDA algorithm, one based on recursive least squares and one based on a modified Cholesky decomposition. However, the efficiency properties of these estimators were not explored. In \cref{sec:numerical}, numerical comparisons will show that our proposed estimator significantly outperforms these estimators.

Other results on the linear SEM include explicit calculations and comparisons for typical examples with either a particular structure or only a few variables; see, e.g., \citet{kuroki2004selection,gupta2020estimating}. Gaussian errors are also assumed in these calculations.

 \section{Linear SEMs, causal graphs and effect identification} \label{sec:prelim}
\subsection{Linear SEMs under causal sufficiency}
A linear SEM postulates a causal mechanism that generates data. Let $X$ denote a vector of variables generated by a linear SEM, where $X$ is indexed by $V$ ($X = X_V$). Let $\mathcal{D}$ be the associated DAG on vertices $V$. For this $|V|$-dimensional random vector $X$, the model in \cref{eqs:sem} can be compactly rewritten as
\begin{equation} \label{eqs:sem-vector}
{X}= \Gamma^{\T} {X} + \epsilon, \quad \Gamma = (\gamma_{ij}), \quad i \rightarrow j \text{ not in } \mathcal{D} \,\Rightarrow\, \gamma_{ij} = 0.
\end{equation}
where $\Gamma \in \mathbb{R}^{|V| \times |V|}$ is a coefficient matrix, and $\epsilon = (\epsilon_i)$ is a $|V|$-dimensional random vector. DAG $\g[D]$ is associated with the linear SEM in \cref{eqs:sem} in the sense that the non-zero entries of $\Gamma$ correspond to the edges in $\mathcal{D}$. 

Under causal sufficiency (no latent variables), we assume
\begin{equation} \label{eqs:sem-error}
\{ \epsilon_i: i \in V\} \text{ are independent},\ \E \epsilon = 0,\  \E \epsilon \epsilon^{\T} \succ \bm{0},
\end{equation}
where for a real symmetric matrix $A$, $A \succ \bm{0}$ means $A$ is positive definite. The errors $\{\epsilon_i: i \in V\}$ are not necessarily Gaussian, nor identically distributed. 

The law $P(X)$ is called the \emph{observed distribution}. For a given $\mathcal{D}$, we will use $\mathcal{P}_{\mathcal{D}}$ to denote the set of possible laws of $X$, namely the collection of $P(X)$ as $\Gamma$ and the error distribution vary subject to \cref{eqs:sem-vector,eqs:sem-error}. The linear SEM poses certain restrictions on the set of laws $\mathcal{P}_{\mathcal{D}}$. Let $\Pa(i, \mathcal{D})$ denote the set of parents of vertex $i$, i.e., $\{j: j \rightarrow i \text{ is in } \mathcal{D}\}$. For any $P \in \mathcal{P}_{\mathcal{D}}$, among other constraints, (i) $P$ factorizes according to $\mathcal{D}$, (ii) $\E[X_i \mid X_{\Pa(i, \mathcal{D})}]$ is linear in $X_{\Pa(i, \mathcal{D})}$ and (iii) $\var[X_i \mid X_{\Pa(i, \mathcal{D})}]$ is constant in $X_{\Pa(i, \mathcal{D})}$. 

We observe $n$ iid samples generated by the model above, namely $X^{(i)} = (I - \Gamma)^{-\T} \epsilon^{(i)}$ for $i=1,\dots,n$. Note that $(I - \Gamma)$ is invertible because $\Gamma$ can be permuted into a lower-triangular matrix by a topological ordering (i.e., causal ordering) of vertices in $\mathcal{D}$. 

\subsection{Interventions and total causal effects} \label{sec:sem-intervention}
The assumed linear SEM also dictates the behavior of the system under interventions. Let $A \subseteq V$ be a set of vertices indexing treatment variables $X_A$. We use $\Do(X_A = x_A)$ to denote \emph{intervening} on variables $X_A$ and forcing them to take values $x_A$ \citep{pearl1995causal}. We call this a \emph{point} intervention if $A$ is a singleton, and a \emph{joint} intervention if $A$ consists of several vertices, which correspond to the case of multiple treatments. While $X_A$ is fixed to $x_A$, the remaining variables are generated by their corresponding structural equations \cref{eqs:sem}, with each $X_i$ for $i \in A$ appearing in the equations replaced by the corresponding enforced value $x_i$ \citep{strotz1960recursive}. This generating mechanism defines the \emph{interventional distribution}, denoted by $P(X | \Do(X_A = x_A))$, where the conditional probability notation is only conventional. More formally, the interventional distribution is expressed as
\begin{equation} \label{eqs:g-formula}
P(X|\Do(X_A = x_A)) = \prod_{j \in A} \delta_{x_j}(X_j) \prod_{i \notin A} P\left(X_i | X_{\Pa(i, \mathcal{D})} \right),
\end{equation}
where $\delta$ denotes a Dirac measure. Factor $P\left(X_i | X_{\Pa(i)} \right)$ is defined by the structural equation for $X_i$. \cref{eqs:g-formula} is known as the truncated factorization formula \citep{pearl2009causality}, manipulated density formula \citep{spirtes2000causation} or the g-formula \citep{robins1986new}.

\begin{definition}[Total causal effect, \citealp{pearl2009causality,nandy2017estimating}] \label{def:total-effect}
Let $X_A$ be a vector of treatment variables and $X_Y$ with $Y \in V \setminus A$ be an outcome variable. The total causal effect of $X_A$ on $X_Y$ is defined as the vector $\tau_{AY} \in \mathbb{R}^{|A|}$, where
\begin{equation*}
(\tau_{AY})_i = \frac{\partial}{\partial x_{A_i}} \E[X_Y \mid \Do(X_A = x_A)], \quad i=1,\dots,|A|.
\end{equation*}
\end{definition}

That is, $\tau_{AY}$ is the gradient of the linear map $x_A \mapsto \E[X_Y | \Do(X_A = x_A)]$. When multiple outcomes $Y= \{Y_1,\dots, Y_k\}$, $k > 1$, are considered, the total causal effect of $X_A$ on $X_{Y_1}, \dots, X_{Y_k}$ can be defined by concatenating $\tau_{A Y_1}, \dots, \tau_{A Y_k}$. Therefore, throughout, we assume the outcome variable is a singleton without loss of generality. Each coordinate of the total causal effect $\tau_{AY}$ can be expressed as a sum-product of the underlying linear SEM coefficients along certain causal paths from $A$ to $Y$ in $\mathcal{D}$, that is, certain paths of the form $A_1 \to \dots \to Y_i$ for $A_1 \in A$; see also \citet{wright1934method,sullivant2010trek}.

\subsection{Causal graphs}
Two different linear SEMs on the same set of variables can define the same observed distribution. For example, under Gaussian errors, linear SEMs associated with DAGs $A \rightarrow Y$ and $A \leftarrow Y$, define the same set of observed distributions, namely the set of centered bivariate Gaussian distributions.
Without making additional assumptions on the error distribution, such as non-Gaussianity \citep{shimizu2006linear}, partial non-Gaussianity \citep{hoyer08}, or equal variance of errors \citep{peters2014identifiability,chen2019causal}, the underlying causal DAG can only be learned from the observed distribution up to its Markov equivalence class \citep{pearl1995theory,chickering2002optimal}.

\paragraph*{CPDAGs} Two DAGs on the same set of vertices are Markov equivalent if they encode the same set of d-separation relations between the vertices. The d-separations between the vertices, prescribe conditional independences between the corresponding variables (known as the Markov condition \citep[\S3.2.2]{lauritzen1996graphical}); see \cref{apx:sec:graph} for the definition of d-separation and more background.  This equivalence relation defines a Markov equivalence class, which consists of DAGs as elements. A Markov equivalence class can be uniquely represented by a \emph{completed partially directed acyclic graph} (CPDAG), also known as an \emph{essential graph} \citep{meek1995causal,andersson1997characterization}. A CPDAG $\mathcal{C}$ is a graph on the same set of vertices, that can contain both directed and undirected edges. We use $[\mathcal{C}]$ to denote the Markov equivalence class represented by CPDAG $\mathcal{C}$. A directed edge $i \rightarrow j$ in $\mathcal{C}$ implies $i \rightarrow j$ is in every $\mathcal{D} \in [\mathcal{C}]$, whereas an undirected edge $i - j $ in $\mathcal{C}$ implies there exist $\mathcal{D}_1, \mathcal{D}_2 \in [\mathcal{C}]$ such that $i \rightarrow j$ in $ \mathcal{D}_1$ but $i \leftarrow j $ in  $\mathcal{D}_2$. Given a DAG $\mathcal{D}$, the CPDAG $\mathcal{C}$ representing the Markov equivalence class of $\mathcal{D}$ can be drawn by keeping the skeleton of $\mathcal{D}$, adding all the unshielded colliders from $\mathcal{D}$ and completing the orientation rules R1--R3 of \citet{meek1995causal}; see \cref{fig:orientationRules} in \cref{apx:sec:graph}. For example, DAGs $A \rightarrow Y$ and $A \leftarrow Y$ are represented by CPDAG $A - Y$.
To slightly abuse the notation, for a distribution $Q$, we write $Q \in [\mathcal{C}]$ if $Q$ factorizes according to some DAG $\mathcal{D} \in [\mathcal{C}]$; see \citet[\S 3.2.2]{lauritzen1996graphical}.

There are various \emph{structure learning} algorithms that can be used to uncover CPDAG $\mathcal{C}$ from observational data. Some well-known examples are the PC algorithm \citep{spirtes2000causation} and the greedy equivalence search \citep{chickering2002optimal}. Choosing an appropriate algorithm for the data set at hand is beyond the scope of this paper; the reader is referred to \citet[\S 4]{drton2017structure} for a recent overview. 

\paragraph*{MPDAGs} Certain background knowledge, if present, can be used to further  orient some  undirected edges in a CPDAG $\mathcal{C}$. Typically, knowledge of temporal orderings can inform the orientation of certain undirected edges; see \citet[\S 5.8.4]{spirtes2000causation} for an example. Adding these background-knowledge orientations and the additionally implied orientations based on the orientation rules of \citet{meek1995causal} to $\mathcal{C}$ results in a \emph{maximally oriented partially directed graph} (MPDAG) $\mathcal{G}$. See \cref{fig:orientationRules} and \cref{algo:MPDAG}  in \cref{apx:sec:graph}. MPDAGs are a rather general class of graphs that subsumes both DAGs and CPDAGs. An MPDAG $\g$ represents a \emph{restricted} Markov equivalence class of DAGs, which we also denote by $[\g]$.  Analogously to the case of a CPDAG,  $i \rightarrow j $ in $\mathcal{G}$ implies $i \rightarrow j$ is in every $\mathcal{D} \in [\mathcal{G}]$, and   $i - j$ in $\mathcal{G}$ implies there exist $\mathcal{D}_1, \mathcal{D}_2 \in [\mathcal{G}]$ such that $i \rightarrow j $ in $\mathcal{D}_1$ but $i \leftarrow j$ in $ \mathcal{D}_2$. 

For the rest of the paper, we will assume that we have access to an MPDAG $\mathcal{G}$ that represents our structural knowledge about the underlying DAG $\mathcal{D}$. That is,
\begin{equation}
\text{causal DAG } \mathcal{D} \in [\mathcal{G}], \quad \mathcal{G} \text{ is an MPDAG},
\end{equation}
where $[\mathcal{G}]$ represents a collection of DAGs that are Markov equivalent, but can be strictly smaller than the corresponding Markov equivalence class due to background knowledge. 

\subsection{Causal effect identification}
Throughout the paper we will use the following notations.
Given treatment variables $X_A$ and an outcome variable $X_Y$ such that $Y \notin A$, we are interested in learning the total causal effect $\tau_{AY}$.
We assume that we have access to an MPDAG $\g$, and to observational data that are generated as iid samples from a linear SEM defined by \cref{eqs:sem-vector,eqs:sem-error}, where the causal DAG $\mathcal{D}$ is in $[\g]$.
 Before estimation can be performed, we need to make sure that $\tau_{AY}$ can be identified from observational data. That is, we need to ensure that $\tau_{AY}$ can be expressed as a functional of the observed distribution that is the same for \emph{every} DAG in $[\mathcal{G}]$. We have the following graphical criterion.

\begin{theorem}[{\citealp{perkovic20}}] \label{thm:id-criterion}
The total causal effect $\tau_{AY}$ of $X_A$ on $X_Y$ is identified given an MPDAG $\g$ if and only if there is no proper, possibly causal path from $A$ to $Y$ in $\mathcal{G}$ that starts with an undirected edge. 
\end{theorem}
Theorem \ref{thm:id-criterion} is Proposition 3.2 of \citet{perkovic20}, which holds for nonparametric causal graphical models. It does not require that the data is generated by a linear SEM. However, \citet{perkovic20} proves that when the criterion fails, then two linear SEMs with Gaussian errors can be constructed such that their observed distributions coincide but their $\tau_{AY}$'s are different. Hence, even if we restrict ourselves to linear SEMs, \cref{thm:id-criterion} still holds. 

A few terms need some explanation. A \textit{path} from $A$ to $Y$ in $\mathcal{G}$ is a sequence of distinct vertices $\langle v_1, \dots, v_k \rangle$ for $k >1$ with $v_1 \in A$ and $v_k = Y$, such that every pair of successive vertices are adjacent in $\g$. The path is \emph{proper} when only its first vertex is in $A$. The path is \emph{possibly causal} if no edge $v_l \leftarrow v_r$ is in $\mathcal{G}$ for $1 \leq l < r \leq k$. The reader is referred to \cref{apx:sec:graph} for more graphical preliminaries. When $\mathcal{G}$ satisfies Theorem \ref{thm:id-criterion} relative to vertex sets $A$ and $Y$, the interventional distribution $P(X_Y | \Do(X_A = x_A))$, and hence the total effect, can be computed from the observed distribution $P(X)$. To express the identification formula, we require the following concepts.

\subsubsection{Buckets and bucket decomposition}  Let $\mathcal{G}=(V, E, U)$ be a partially directed graph, where $V$ is the set of vertices, and $E$ and $U$ are sets of directed and undirected edges respectively. Let $B_1, \dots, B_K$ be the \emph{maximal connected components} of the undirected graph $\mathcal{G}_U := (V, \emptyset, U)$. Then $V = B_1\,\dot{\cup} \dots \dot{\cup}\,B_K$, where symbol $\dot{\cup}$ denotes disjoint union. Note that all the directed edges within each $B_i$ are due to background knowledge. If we ignore the distinction between directed and undirected edges, then the subgraph induced by each $B_i$ is chordal \citep[\S4]{andersson1997characterization}.

Suppose the connected components are ordered such that 
\begin{equation} \label{eqs:bucket-ordering}
i \rightarrow j \in E, \  i \in B_i, \  j \in B_j \quad \Rightarrow \quad i < j.
\end{equation}
One can show that such a \emph{partial causal ordering} always exists, though it may not be unique; see \cref{alg:pto} in \cref{apx:sec:graph} to obtain such an ordering. Our result does not depend on the particular choice of partial causal ordering. We call $B_1, \dots, B_K$ the \emph{bucket decomposition} of $V$ and call each $B_k$ for $k = 1,\dots, K$ a \emph{bucket}; see \cref{fig:PDAG-and-saturated}(a) for an example.
If it is clear which graph $\g$ is being referred to, we will shorten $\Pa(j, \g)$ as $\Pa(j)$ to reduce clutter. For a set of vertices $C$ in $\g$, we use $\Pa(C) := \cup_{i \in C} \Pa(i) \setminus C$ to denote the set of their \emph{external parents}. Clearly, $\Pa(B_k) \subseteq B_{[k-1]}$, where $B_{[k-1]} := B_1 \cup \dots \cup B_{k-1}$. 

\begin{lemma}\label{lemma:anc-bucket21} 
Let $i$ and $j$ be two distinct vertices in MPDAG $\mathcal{G} = (V, E, U)$ such that $i \rightarrow j \in  E$. Suppose that there is no undirected path from $i$ to $j$ in $\mathcal{G}$. 
If there is a vertex $k$, and an undirected path  $j - \dots - k$ in $\mathcal{G}$, then $i \rightarrow k \in E$.
\end{lemma}

By definition of the parent set above we have that $\Pa(B_k) = \cup_{i \in B_k} \Pa(i) \setminus B_k$, $k = 1, \dots , K$. However, since a bucket $B_k$ is a maximal subset of $V$ that is connected by undirected edges in $\g$, \cref{lemma:anc-bucket21} implies the following important property.

\begin{corollary}[Restrictive property] \label{cor:restrictive}  
Let $B_1, \dots, B_K$ be the bucket decomposition of $V$ in MPDAG $\g = (V,E,U)$. Then, all vertices in the same bucket have the same set of external parents, namely
\begin{equation*}
\Pa(B_k) = \Pa(i) \setminus B_k, \quad \text{for any } i \in B_k,\, k = 1, \dots , K.
\end{equation*}
\end{corollary}

The causal identification formula for $P(X_Y |\Do(X_A = x_A))$ of \citet{perkovic20} relies on a decomposition of certain \emph{ancestors} of $Y$ in MPDAG $\g$ according to the buckets.
We call vertex $i$ an ancestor of vertex $j$ in $\g$ if there exists a directed path  $i \to \dots \to j$ in $\g$; we use the convention that $j$ is an ancestor of itself. We denote the set of ancestors of $j$ in $\g$ as $\An(j, \g)$, or shortened as $\An(j)$.

Let $\mathcal{G}_{V \setminus A} = (V \setminus A, E', U')$ denote the subgraph of $\mathcal{G}$ \emph{induced} by the vertices $V \setminus A$, where $E'$ includes those edges in $E$ that are between vertices in $V \setminus A$, and similarly for $U'$.  Consider the set of ancestors of $Y$ in $\mathcal{G}_{V \setminus A}$, denoted as 
\begin{equation}
D := \An(Y, \mathcal{G}_{V \setminus A}).
\end{equation}
The bucket decomposition $D_1, \dots, D_K$ of $D$, induced by the bucket decomposition of $V$, is simply
\begin{equation}
D = \dot{\bigcup}_{k=1}^K D_k, \quad D_k = D \cap B_k, \quad i=1,\dots,K.
\label{eq:small-bucket}
\end{equation}
\begin{lemma} \label{lem:pa-D}
When the criterion in \cref{thm:id-criterion} is satisfied, we have $\Pa(D_k, \g) = \Pa(B_k, \g)$ for every nonempty $D_k$. 
\end{lemma}
\noindent Proofs of \cref{lemma:anc-bucket21,lem:pa-D} are left to \cref{apx:sec:graph-proofs}.

\begin{theorem} [{\citealp{perkovic20}}] \label{thm:id-nonpara}
Suppose the criterion in \cref{thm:id-criterion} is satisfied for $A,Y$ in MPDAG $\g = (V,E,U)$ such that $Y \notin A$. Let $P(X)$ be the observed distribution. Let $D = \An(Y, \mathcal{G}_{V \setminus A})$ and $D_1, \dots D_K$ be the bucket decomposition of $D$ as in \cref{eq:small-bucket}. Then the interventional distribution $ P(X_Y | \Do(X_{A} = x_A))$  can be identified as 
\begin{equation} \label{eqs:truncated-fact}
P(X_Y | \Do(X_{A} = x_A)) = \bigintsss \left\{ \prod_{k=1}^{K} P \left(X_{D_k} | X_{\Pa(D_k)} \right) \right\} \dd X_{D \setminus Y}
\end{equation}
for values $X_{\Pa(D_k)}$ in agreement with $x_A$, where $P \left(X_{D_k} | X_{\Pa(D_k)}\right) \equiv 1$ if $D_k = \emptyset$.
\end{theorem}
The expression in \cref{eqs:truncated-fact} above is a generalization of the truncated factorization \cref{eqs:g-formula} from DAGs to MPDAGs. \cref{thm:id-nonpara} holds generally even when an underlying linear SEM is not assumed.

 \section{Block-recursive representation} \label{sec:block}
In this section, we express the observed distribution $P(X)$ induced by a linear SEM compatible with MPDAG $\g = (V,E,U)$ in a block-recursive form. Each block corresponds to a bucket in the bucket decomposition of $V$. Such a reparameterization is necessitated by the fact that the causal ordering of $\mathcal{D}$ is unknown, whereas the buckets can be arranged into a valid partial causal ordering as in \cref{eqs:bucket-ordering}. We will use this representation to compute an estimator for the total causal effect.

Recall that $\mathcal{P}_\mathcal{D}$ denotes the family of laws of $X$ arising from a linear SEM \cref{eqs:sem-vector,eqs:sem-error} compatible with DAG $\g[D]$. Let $\mathcal{P}_{\mathcal{G}} := \cup_{\mathcal{D} \in [\mathcal{G}]} \mathcal{P}_\mathcal{D}$, which denotes the family of laws of $X$ arising from a linear SEM compatible with a DAG in $[\g]$.
\begin{proposition}[Block-recursive form] \label{prop:block-recursive}
Let $\g[D]$ be the causal DAG associated with the linear SEM and $\g$ an MPDAG such that $\g[D] \in [\g]$. Further, let $B_1, \dots, B_K$ be the bucket decomposition of $V$ in $\g$.
Then the linear SEM \cref{eqs:sem-vector,eqs:sem-error} can be rewritten as
\begin{equation*}
X = \Lambda^{\T} X + \varepsilon, 
\end{equation*}
for some matrix of coefficients $\Lambda = (\lambda_{ij}) \in \mathbb{R}^{|V| \times |V|}$ and random vector $\varepsilon = (\varepsilon_i) \in \mathbb{R}^{|V|}$ such that 
\begin{align}
& j \in B_l, \ i \notin \Pa(B_l, \g) \quad \Rightarrow \quad \lambda_{ij} = 0,  \label{eqs:lambda-cons}  \\
& \E \varepsilon = 0,\quad  \E \varepsilon_{B_k} \varepsilon_{B_k}^{\T} \succ \bm{0}, \ (k=1,\dots,K), \quad \varepsilon_{B_1}, \dots, \varepsilon_{B_K} \text{ are mutually independent},  \label{eqs:err-dist-cons}
\end{align}
and 
\begin{equation} \label{eqs:block-error-markov}
\text{law of }\left(\varepsilon_{B_k}\right) \in \mathcal{P}_{\mathcal{G}_{B_k}}, \quad k=1,\dots,K,
\end{equation}
where $\mathcal{G}_{B_k}$ is the subgraph of $\mathcal{G}$ induced by $B_k$. 
\end{proposition}
\noindent Note that in contrast to symbol $\epsilon$ used in \cref{eqs:sem-vector,eqs:sem-error}, symbol $\varepsilon$ is used here to denote the errors in the block-recursive form. The coordinates within each $\varepsilon_{B_k}$ may be \emph{dependent}.
\begin{proof}
For $k=2,\dots,K$, by \cref{eqs:sem-vector} and the restrictive property (\cref{cor:restrictive}), we have
\[ X_{B_k} = \Gamma^{\T}_{\Pa(B_k), B_k} X_{\Pa(B_k)} + \Gamma_{B_k}^{\T} X_{B_k} + \epsilon_{B_k},  \]
where $\Pa(B_k) = \Pa(B_k, \g)$. The expression can be rewritten as 
\begin{equation*}
\begin{split}
X_{B_k} &= \left(I - \Gamma_{B_k}\right)^{-\T} \Gamma^{\T}_{\Pa(B_k),B_k} X_{\Pa(B_k)} + \left(I - \Gamma_{B_k}\right)^{-\T} \epsilon_{B_k} \\
&= \Lambda_{\Pa(B_k), B_k}^{\T} X_{\Pa(B_k)} + \varepsilon_{B_k},
\end{split}
\end{equation*}
where $\varepsilon_{B_k} := \left(I - \Gamma_{B_k}\right)^{-\T} \epsilon_{B_k}$ for $k=1,\dots,K$ (note that $X_{B_1} = \varepsilon_{B_1}$). Additionally, $\Lambda_{\Pa(B_k), B_k} = \Gamma_{\Pa(B_k),B_k} \left(I - \Gamma_{B_k} \right)^{-1}$ for $k=2,\dots,K$. 

Matrix $\Lambda \in \mathbb{R}^{|V| \times |V|}$ in the statement of the proposition is defined by blocks $\Lambda_{\Pa(B_k), B_k} $ for $k=2,\dots,K$ and zero entries otherwise. Therefore, $\lambda_{ij} = 0$ if $j \in B_l$ and $i \notin \Pa(B_l)$ for some $l = 1, \dots K$. Hence, by putting the blocks together, the model can be written as $X = \Lambda^{\T} X + \varepsilon$. 

The ``new'' errors $\varepsilon$ satisfy
\[ \varepsilon_{B_k} = \Gamma_{B_k}^{\T} \varepsilon_{B_k} + \epsilon_{B_k}, \quad k=1,\dots,K. \]
It then follows from \cref{eqs:sem-vector,eqs:sem-error} that for every $k$, 
\[ \text{law of } \varepsilon_{B_k} \in \mathcal{P}_{\mathcal{D}_{B_k}} \subset \mathcal{P}_{\mathcal{G}_{B_k}},\]
since $D \in [\mathcal{G}]$. Moreover, for every $k$,
\[ \E \varepsilon_{B_k} = 0, \quad \E \varepsilon_{B_k} \varepsilon_{B_k}^{\T} = (I - \Gamma_{B_k})^{-\T} \E \epsilon_{B_k}\epsilon_{B_k}^{\T} (I - \Gamma_{B_k})^{-1} \succ \bm{0}, \]
where both $(I - \Gamma_{B_k})$ and $\E \epsilon_{B_k}\epsilon_{B_k}^{\T}$ are full rank, because $\Gamma_{B_k}$ can be permuted into an upper-triangular matrix and $\E \epsilon \epsilon^{\T} \succ \bm{0}$ by \cref{eqs:sem-error}.
\end{proof}

\begin{corollary} \label{cor:block-recursive}
Under the same conditions as \cref{prop:block-recursive}, it holds that
\begin{equation} \label{eqs:block-rec}
\begin{split}
X_{B_1} &= \varepsilon_{B_1}, \\
X_{B_k} &= \Lambda^{\T}_{\Pa(B_k), B_k} X_{\Pa(B_k)} + \varepsilon_{B_k}, \quad \varepsilon_{B_k} \indep X_{\Pa(B_k)}, \quad k=2,\dots,K,
\end{split}
\end{equation}
where $\Pa(B_k) = \Pa(B_k, \g)$.
\end{corollary}

Next, we show that if the total causal effect $\tau_{AY}$ is identifiable from MPDAG $\g$ (Theorem \ref{thm:id-criterion}), then it can be calculated from $\Lambda$ in the block-recursive representation of Proposition \ref{prop:block-recursive}. Therefore, the distribution of $\varepsilon$ is a nuisance relative to estimating $\tau_{AY}$. 
\begin{proposition} \label{prop:id}
Suppose the criterion in \cref{thm:id-criterion} is satisfied for $A,Y$ in MPDAG $\g = (V,E,U)$ such that $Y \notin A$. Let $\Lambda$ be the block-recursive coefficient matrix given by \cref{prop:block-recursive}. The total causal effect of $X_A$ on $X_Y$ is identified as
\begin{equation} \label{eqs:effect-id}
\tau_{AY} = \Lambda_{A, D} \left[(I - \Lambda_{D,D})^{-1} \right]_{D,Y},
\end{equation}
where $D = \An(Y, \g_{V \setminus A})$ and the last subscript denotes the column corresponding to $Y \in D$. 
\end{proposition}
\begin{proof}
We derive this result using \cref{thm:id-nonpara}. Recall that $D_1, \dots, D_K$ is a partition of $D$ induced by the bucket decomposition $B_1, \dots, B_K$ of $V$ in the sense that $D_k = D \cap B_k$ for $k=1,\dots,K$. When $D_k = \emptyset$, we use the convention that $P(X_{D_k} | X_{\Pa(D_k)}) \equiv 1$. By definition of $D = \An(Y, \mathcal{G}_{V \setminus A})$ and \cref{eqs:bucket-ordering}, observe that a vertex in $\Pa(D_k) = \Pa(D_k, \g)$ is either in $D_1 \cup \dots \cup D_{k-1}$ or in $A$. Let $F_k := A \cap \Pa(D_k)$. In \cref{eqs:truncated-fact}, we note that the joint interventional distribution of $X_D$ is given by
\begin{equation*}
P(X_{D} | \Do(X_A =x_A)) = \prod_{k=1}^{K} P(X_{D_k} | X_{\Pa(D_k)}) = \prod_{k=1}^{K} P(X_{D_k} | X_{\Pa(D_k) \setminus F_k}, X_{F_k} = x_{F_k}),
\end{equation*}
where $x_{F_k}$ is fixed by the $\text{do}(X_A = x_A)$ operation. Further, fix a factor $i \in \{1,\dots,K\}$. By \cref{lem:pa-D}, $\Pa(D_i) = \Pa(B_i)$. By \cref{eqs:block-rec} and $\varepsilon_{D_i} \indep X_{\Pa(B_i)}$, we have 
\begin{equation*}
\begin{split}
X_{D_i} \mid \left\{ X_{\Pa(D_i) \setminus F_i}, X_{F_i} = x_{F_i}\right\} &=_{d} \Lambda_{\Pa(D_i) \setminus F_i, D_i}^{\T} X_{\Pa(D_i) \setminus F_i} + \Lambda_{F_i, D_i} x_{F_i} + \varepsilon_{D_i} \\
&= \Lambda_{\Pa(D_i) \cap D, D_i}^{\T} X_{\Pa(D_i) \cap D} + \Lambda_{\Pa(D_i) \cap A, D_i} x_{\Pa(D_i) \cap A} + \varepsilon_{D_i}
\end{split}
\end{equation*}
The fact that the display above holds for every $i=1,\dots,K$ implies that the joint interventional distribution $P(X_{D} | \Do(X_A =x_A)) $ satisfies
\begin{equation*}
X_{D} = \Lambda_{D,D}^{T} X_{D} + \Lambda_{A,D}^{\T} x_A + \varepsilon_{D}.
\end{equation*}
It follows that $X_D = (I - \Lambda_{D,D})^{-\T} (\Lambda_{A,D}^{\T} x_A + \varepsilon_{D})$ and hence 
\begin{equation*}
\E[X_D \mid \text{do}(X_{A} = x_A)] = (I - \Lambda_{D,D})^{-\T} \Lambda_{A,D}^{\T} x_{A}.
\end{equation*}
Since $Y \in D$, by \cref{def:total-effect} we have
\begin{equation*}
\tau_{AY} = \frac{\partial}{\partial x_A} \E[X_Y \mid \text{do}(X_{A} = x_A)] = \Lambda_{A,D} \left[(I - \Lambda_{D,D})^{-1} \right]_{D,Y}.
\end{equation*} 
\end{proof}

We say vertex $j$ is a possible descendant of $i$, denoted as $j \in \PossDe(i)$, if there exists a possibly causal path from $i$ to $j$. For a set of vertices $A$, define $\PossDe(A) := \cup_{i \in A} \PossDe(i)$. See \cref{apx:sec:graph} for more details.
\begin{corollary}
If $Y \notin \PossDe(A)$, then $\tau_{AY} = 0$.
\end{corollary}
\begin{proof}
Since $D = \An(Y, \mathcal{G}_{V \setminus A})$ and $Y \notin \PossDe(A)$, $\Lambda_{A,D} = \bm{0}$. 
\end{proof}

 \section{Recursive least squares} \label{sec:MLE}
Consider the \emph{special case} when the errors in the linear SEM \cref{eqs:sem} are jointly Gaussian. In this case, by the standard maximum likelihood theory, the Cram\'er--Rao bound is achieved by the maximum likelihood estimator (MLE) of the total causal effect, which can be obtained by plugging in the MLE for $\Lambda$ in the block-recursive form (\cref{prop:block-recursive}) into the formula \cref{eqs:effect-id}. We now compute the MLE for $\Lambda$ given an MPDAG $\g$.

When $\epsilon$ is multivariate Gaussian, the block-recursive form in \cref{prop:block-recursive} is a linear Gaussian model parameterized by $\{(\Lambda_k)_{k=2}^{K}, (\Omega_k)_{k=1}^K\}$, where $\Lambda_k := \Lambda_{\Pa(B_k), B_k}$ and $\Omega_k$ is the covariance for $\varepsilon_{B_k}$. Because $\varepsilon$ are independent between blocks (Proposition \ref{prop:block-recursive}), the likelihood factorizes as
\begin{equation} \label{eqs:gaussian-likelihood}
\mathcal{L}((\Lambda_k)_k,(\Omega_k)_k) = \prod_{k=1}^{K} \N\left(X_{B_k} - \Lambda^{\T}_{k} X_{\Pa(B_k)}; \bm{0}, \Omega_k \right).
\end{equation}

Denote the MLE of $\Lambda$ by $\hat{\Lambda}^{\mathcal{G}}$, which consists of blocks $(\hat{\Lambda}_{k}^{\mathcal{G}})_{k=2}^{K}$ and zero values elsewhere, and the MLE of $\Omega$ by $\hat{\Omega}^{\mathcal{G}} = (\hat{\Omega}_{k}^{\mathcal{G}})_{k=1}^{K}$. The superscripts highlight the dependence on MPDAG $\mathcal{G}$. The MLE maximizes $\mathcal{L}((\Lambda_k)_k,(\Omega_k)_k)$ subject to \cref{eqs:block-error-markov}, namely
 \begin{equation*}
\N(\bm{0}, \Omega_k) \in \mathcal{P}_{\mathcal{G}_{B_k}}, \quad k=1,\dots,K,
\end{equation*}
where $\mathcal{G}_{B_k}$ is the subgraph of $\mathcal{G}$ induced by $B_k$. 
This further translates to a set of algebraic constraints on $(\Omega_k)_{k=1}^{K}$, namely for $k=1,\dots,K$,
\begin{equation} \label{eqs:algebraic-cons}
\det \left[ (\Omega_k)_{\{i\} \cup C, \{j\} \cup C} \right] = 0, \ \text{if } \text{$i$ and $j$ are d-separated by $C$ in $\mathcal{G}_{B_k}$};
\end{equation}
see, e.g., \citet[\S 3.1]{drton2008lectures}.
Although the constraints \cref{eqs:algebraic-cons} may seem daunting, we will show that they do not affect the MLE for $\Lambda$.

Let the sample covariance matrix   be computed with respect to mean zero, i.e.,
\begin{equation} \label{eqs:sample-cov}
\hat{\Sigma}^{(n)} := \frac{1}{n} \sum_{i=1}^n X^{(i)} X^{(i) \T},
\end{equation}
where $n$ is the sample size, and the superscripts are reserved to index samples. To reduce clutter, for a set of indices $C$, we often abbreviate $\Sigma_{C,C}$ as $\Sigma_{C}$.

\begin{lemma} \label{lem:MLE-G-regression} 
Suppose $X^{(i)}: i=1,\dots,n$ is generated iid from a linear SEM \cref{eqs:sem-vector,eqs:sem-error} associated with an unknown causal DAG $\mathcal{D}$. Suppose the error $\epsilon$ is distributed as multivariate Gaussian. Suppose $\mathcal{D} \in [\mathcal{G}]$ for a known MPDAG $\mathcal{G}$. Let $\hat{\Sigma}^{(n)}$ be the sample covariance as defined in \cref{eqs:sample-cov}. The MLE for $\Lambda_k = \Lambda_{\Pa(B_k), B_k}$ in the block-recursive form is given by
\begin{equation} \label{eqs:G-reg}
\hat{\Lambda}_{k}^{\mathcal{G}} = \left(\hat{\Sigma}^{(n)}_{\Pa(B_k)} \right)^{-1} \hat{\Sigma}^{(n)}_{\Pa(B_k), B_k}, \quad k=2,\dots,K.
\end{equation}
\end{lemma}
\begin{proof}
By factorization in \cref{eqs:gaussian-likelihood}, MLE $(\hat{\Lambda}_{k}^{\mathcal{G}}, \hat{\Omega}_{k}^{\mathcal{G}})$ is the maximizer of log-likelihood
\begin{equation*}
\begin{split}
& \quad \ell_n(\Lambda_k, \Omega_k) \\
& =  -\frac{1}{2} \sum_{i=1}^{n} \left(X_{B_k}^{(i)} - \Lambda_k^{\T} X_{\Pa(B_k)}^{(i)}\right)^{\T} \Omega_{k}^{-1} \left(X_{B_k}^{(i)} - \Lambda_k^{\T} X_{\Pa(B_k)}^{(i)}\right) - \frac{n}{2} \log \det(\Omega_k) \\
&= -\frac{1}{2} \Tr\left(\sum_{i=1}^{n} \Omega_{k}^{-1} (X_{B_k}^{(i)} - \Lambda_k^{\T} X_{\Pa(B_k)}^{(i)}) (X_{B_k}^{(i)} - \Lambda_k^{\T} X_{\Pa(B_k)}^{(i)})^{\T} \right) - \frac{n}{2} \log \det(\Omega_k),
\end{split}
\end{equation*}
subject to \cref{eqs:algebraic-cons}. Taking a derivative with respect to $\Lambda_k \in \mathbb{R}^{|\Pa(B_k)| \times |B_k|}$, we have
\begin{equation*}
\frac{\partial \ell_n(\Lambda_k, \Omega_k)}{\partial \Lambda_k} = -2 \sum_{i=1}^{n} X_{\Pa(B_k)}^{(i)} X_{B_k}^{(i)\T} \Omega_k^{-1} + 2 \sum_{i=1}^{n} X_{\Pa(B_k)}^{(i)} X_{\Pa(B_k)}^{(i) \T} \Lambda_k \Omega_k^{-1}.
\end{equation*}
For any positive definite $\Omega_k$ satisfying \cref{eqs:algebraic-cons}, setting the derivative ${\ell_n(\Lambda_k, \Omega_k)}/{\partial \Lambda_k}$ to zero yields the estimate
\begin{equation*}
\hat{\Lambda}_{k}^{\mathcal{G}} = \left(\frac{1}{n}\sum_{i=1}^{n} X_{\Pa(B_k)}^{(i)} X_{\Pa(B_k)}^{(i) \T}\right)^{-1} \left(\frac{1}{n}\sum_{i=1}^{n} X_{\Pa(B_k)}^{(i)} X_{B_k}^{(i)\T}\right) =\left(\hat{\Sigma}_{\Pa(B_k)}^{(n)} \right)^{-1} \hat{\Sigma}_{\Pa(B_k), B_k}^{(n)}.
\end{equation*}
\end{proof}
\begin{remark}
Because of the restrictive property (\cref{cor:restrictive}), each $\hat{\Lambda}_{k}^{\mathcal{G}}$ is computed by optimizing over the space of $|\Pa(B_k)| \times |B_k|$ matrices and the resulting MLE takes the simple form as above; see also \citet[\S 5]{anderson1985maximum} and \citet[\S 6.4]{amemiya1985advanced} for earlier discussions of this phenomenon. 

However, such a simple form is unavailable in general, when the zero constraints on $\Lambda$ do not obey the restrictive property, even if we ignore the algebraic constraints \cref{eqs:algebraic-cons} on $\Omega$. In fact, the likelihood function can be multimodal; see also \citet{drton2004multimodality,drton2006computing,drton09a} on seemingly unrelated regressions.
\end{remark}

Since $\hat{\Lambda}^{\mathcal{G}}$ is obtained by simply regressing each $B_i$ onto $\Pa(B_i, \g)$ using ordinary least squares, we call this specific recursive least squares \emph{$\mathcal{G}$-regression}. The resulting MLE for an identified total causal effect is a plugin estimator using the formula in \cref{prop:id}.

\begin{definition}[$\mathcal{G}$-regression estimator] \label{def:G-reg-plugin}
Suppose $X^{(i)}: i=1,\dots,n$ is generated iid from a linear SEM \cref{eqs:sem-vector,eqs:sem-error} associated with an unknown causal DAG $\mathcal{D}$. Suppose $\mathcal{D} \in [\mathcal{G}]$ for a known MPDAG $\mathcal{G}$. Further, suppose for $A \subset V$, $Y \in V \setminus A$, $\tau_{AY}$ is identified under the criterion of \cref{thm:id-criterion}. 
The $\mathcal{G}$-regression estimator for the total causal effect $\tau_{AY}$ is defined as
\begin{equation}
\hat{\tau}_{AY}^{\mathcal{G}} = \hat{\Lambda}_{A, D}^{{\mathcal{G}}} \left[(I - \hat{\Lambda}_{D,D}^{{\mathcal{G}}})^{-1} \right]_{D,Y},
\end{equation}
where $\hat{\Lambda}^{\mathcal{G}}$ is given by \cref{eqs:G-reg}.
\end{definition}
 \section{Efficiency theory} \label{sec:efficiency}
In this section, we establish the asymptotic efficiency of our $\mathcal{G}$-regression estimator, when the errors in the generating linear SEM are \emph{not} necessarily Gaussian, among a reasonably large class of estimators---all regular estimators that only depend on the sample covariance. This class of estimators, despite not covering all the estimators considered in the standard semiparametric  efficiency theory, includes many in the literature:
\begin{enumerate}
\item Optimal covariate adjustment \citep{henckel2019graphical,witte2020efficient}: $\hat{\tau}^{\texttt{adj.O}}_{AY}$ is the coefficient of $A$ in the least squares regression of $Y \sim A + O$, where $O$ is the optimal adjustment set, which minimizes asymptotic variance among all valid adjustment sets.

\item Recursive least squares \citep{nandy2017estimating, gupta2020estimating}: Under point intervention ($|A|=1$), total effect $\hat{\tau}_{AY}$ is taken as the coefficient of $A$ in the least square regression of $Y \sim A + \Pa(A)$. \citet[\S 3.1]{nandy2017estimating} consider a recursive relation that expresses $\hat{\tau}_{AY}$ under joint intervention ($|A|>1$) as a polynomial function of a collection of (intermediate) point intervention total effects.

\item Modified Cholesky decomposition \citep{nandy2017estimating}: If the causal ordering is known, then $\Gamma$ can be recovered by regressing each variables on all the preceding variables, which is equivalent to Cholesky decomposition of the covariance $\Sigma = L D L^{\T}$. In fact, $L^{-1}$ is a lower-triangular matrix filled with negative values of these regression coefficients and ones on the diagonal. Since the causal ordering is unknown, \citet[\S 3.2]{nandy2017estimating} consider an iterative procedure that alters between Cholesky decomposition and its reverse. 
\end{enumerate}

\begin{definition} \label{def:acov}
Consider an estimator $\hat{\theta}_n$ of $\theta$, $\theta \in \mathbb{R}^k$. We say that the asymptotic covariance of $\hat{\theta}_n$ is $S$, and write $\acov \hat{\theta}_n = S$, if $\sqrt{n}(\hat{\theta}_n - \theta) \distconvto \N(\bm{0}, S)$. When $k=1$, we write $\avar \hat{\theta}_n$ for asymptotic variance.
\end{definition}

For real symmetric matrices $A$ and $B$, we say $A \succeq B$ if $A - B$ is positive semidefinite. We now state our main result.

\begin{theorem}[Asymptotic efficiency of the $\g$-regression estimator] \label{thm:optimal-efficiency}
Suppose data is generated iid from a linear SEM \cref{eqs:sem-vector,eqs:sem-error} associated with an unknown causal DAG $\mathcal{D}$. Suppose $\mathcal{D} \in [\mathcal{G}]$ for a known MPDAG $\mathcal{G}$. Further, suppose for $A \subset V$, $Y \in V \setminus A$, $\tau_{AY}$ is identified under the criterion of \cref{thm:id-criterion}. 
Let $\hat{\tau}_{AY}^{\mathcal{G}}$ be the $\mathcal{G}$-regression estimator of $\tau_{AY}$ (\cref{def:G-reg-plugin}). Consider any consistent estimator $\hat{\tau}_{AY} = \hat{\tau}_{AY}(\hat{\Sigma}^{(n)})$ that is a differentiable function of the sample covariance. It holds that 
\begin{equation*}
\acov \left( \hat{\tau}_{AY} \right) \succeq \acov \left(\hat{\tau}_{AY}^{\mathcal{G}} \right).
\end{equation*}
\end{theorem}

It is clear from definitions that both $\hat{\tau}_{AY}^{\mathcal{G}}$ and $\hat{\tau}_{AY}$ are asymptotically linear. Therefore, their asymptotic covariances are well-defined. To prove  \cref{thm:optimal-efficiency}, it suffices to show that for every $w \in \mathbb{R}^{|A|}$
\begin{equation*}
\avar \left(w^\T \hat{\tau}_{AY} \right) \geq \avar \left(w^{\T} \hat{\tau}_{AY}^{\mathcal{G}} \right).
\end{equation*}
To this end, for any fixed $w \in \mathbb{R}^{|A|} $ we define $\tau_{w}$ as
\begin{equation} \label{eqs:tau-w}
\tau_{w} := w^{\T} \tau_{AY} = \tau_{w}(\Lambda),
\end{equation}
which is a smooth function of $\Lambda$. The corresponding $\mathcal{G}$-regression estimator $\hat{\tau}_{w}^{\mathcal{G}} := w^{\T} \hat{\tau}_{AY}^{\mathcal{G}} =  \tau_{w}(\hat{\Lambda}^{\mathcal{G}})$ is still a plugin estimator (now of $\tau_{w}$). Additionally, for a consistent estimator $\hat{\tau}_{AY}$ of $\tau_{AY}$, the corresponding $\hat{\tau}_w := w^{\T} \hat{\tau}_{AY} = \hat{\tau}_w(\hat{\Sigma}^{(n)})$ is a consistent estimator of $\tau_{w}$, in the form of a differentiable function of the sample covariance. It suffices to show $\avar \hat{\tau}_{w} \geq \avar \hat{\tau}_{w}^{\mathcal{G}}$ for every  $w \in \mathbb{R}^{|A|} $.

The rest of this section is devoted to proving \cref{thm:optimal-efficiency}. First, we introduce graph $\bar{\mathcal{G}}$ as a  saturated version of $\g$ (\cref{prop:satmpdag}). In \cref{sec:G-bar-reg}, we show that $\g$-regression with $\mathcal{G}$ replaced by $\bar{\mathcal{G}}$, aptly named $\bar{\g}$-regression, is a diffeomorphism between the space of covariance matrices and the space of parameters. In \cref{sec:class-of-estimators}, we characterize the class of estimators relative to which $\mathcal{G}$-regression is optimal. To prove \cref{thm:optimal-efficiency}, we establish an efficiency bound for this class of estimators in \cref{sec:effic-bound} and verify that $\mathcal{G}$-regression achieves this bound in \cref{sec:g-reg-effic}. Some of the proofs are left to \cref{apx:sec:asymp}. See also \cref{apx:fig:proof-diagram-eff} in for an overview of the dependency structure of our results in this section.

\subsection{$\bar{\mathcal{G}}$-regression as a diffeomorphism} \label{sec:G-bar-reg}

\begin{proposition}[Saturated MPDAG $\bar{\mathcal{G}}$] \label{prop:satmpdag} 
For MPDAG $\mathcal{G} = (V, E, U)$, an associated saturated MPDAG is $\bar{\mathcal{G}} = (V, \bar{E}, U)$, such that $\Pa(B_k, \mathcal{\bar{G}}) = B_{[k-1]}$ for $k = 2,\dots,K$, where $(B_1, \dots, B_K)$ is a bucket decomposition of $V$ in both $\g$ and $\bar{\g}$. 
\end{proposition}
The proof can be found in \cref{apx:sec:graph-proofs}. In words, to create the saturated MPDAG $\bar{\mathcal{G}}$, we add all the possible directed edges between buckets $B_1, \dots, B_K$ subject to the ordering $B_1, \dots, B_K$. By construction, $\bar{\g}$ also satisfies the restrictive property in \cref{cor:restrictive}. See \cref{fig:PDAG-and-saturated} for an example. 

\begin{figure}[!htb]
\begin{subfigure}[t]{.45\linewidth}
  \centering
	\begin{tikzpicture}[->,>=triangle 45,shorten >=1pt,
        auto,thick,
        main node/.style={circle,inner
          sep=2pt,fill=gray!20,draw,font=\sffamily}]  
      
        \node[main node] (1) {1}; 
        \node[main node] (2) [right=1.5cm of 1] {2}; 
        \node[main node] (3) [right=1.5cm of 2] {3}; 
        \node[main node] (4) [below=1.5cm of 3] {4};
        \node[main node] (5) [left=1.5cm of 4] {5};
      	\node[main node] (6) [left=1.5cm of 5] {6};
      
        \path[color=black!20!blue,every
        node/.style={font=\sffamily\small}] 
        (1) edge node {} (2)
        (1) edge node {} (4)
        (4) edge node {} (5);
        
        \path[color=black!20!blue,every
        node/.style={font=\sffamily\small}, style={bend left}] 
        (1) edge node {} (3)
        (4) edge node {} (6);
        
        \path[color=black!20!blue,every
        node/.style={font=\sffamily\small}, style={-}] 
        (2) edge node {} (3)
        (3) edge node {} (4)
        (5) edge node {} (6);
      \end{tikzpicture}
\caption{}
\end{subfigure}
\begin{subfigure}[t]{.45\linewidth}
  \centering
      \begin{tikzpicture}[->,>=triangle 45,shorten >=1pt,
        auto,thick,
        main node/.style={circle,inner
          sep=2pt,fill=gray!20,draw,font=\sffamily}]  
      
        \node[main node] (1) {1}; 
        \node[main node] (2) [right=1.5cm of 1] {2}; 
        \node[main node] (3) [right=1.5cm of 2] {3}; 
        \node[main node] (4) [below=1.5cm of 3] {4};
        \node[main node] (5) [left=1.5cm of 4] {5};
      	\node[main node] (6) [left=1.5cm of 5] {6};
      
        \path[color=black!20!blue,every
        node/.style={font=\sffamily\small}] 
        (1) edge node {} (2)
        (4) edge node {} (5)
        (1) edge node {} (4);
        
        \path[color=black!20!blue,every
        node/.style={font=\sffamily\small}, style={dashed}] 
        (2) edge node {} (5)
        (3) edge node {} (5)
	    (1) edge node {} (5)
	    (1) edge node {} (6)
	    (2) edge node {} (6)
	    (3) edge node {} (6);
        
        \path[color=black!20!blue,every
        node/.style={font=\sffamily\small}, style={bend left}]
        (1) edge node {} (3)
        (4) edge node {} (6);
                
        \path[color=black!20!blue,every
        node/.style={font=\sffamily\small}, style={-}] 
        (2) edge node {} (3)
        (3) edge node {} (4)
        (5) edge node {} (6);
      \end{tikzpicture}
\caption{}
\end{subfigure}
\caption{(a) MPDAG $\mathcal{G}=(V, E, U)$ with buckets $B_1=\{1\}$, $B_2=\{2,3,4\}$ and $B_3=\{5,6\}$ and (b) its associated saturated MPDAG $\bar{\mathcal{G}}=(V, \bar{E}, U)$. The new edges in $\bar{E} \setminus E$ are drawn as dashed. Both $\mathcal{G}$ and $\bar{\mathcal{G}}$ satisfy the restrictive property in Corollary \ref{cor:restrictive}.}
\label{fig:PDAG-and-saturated}
\end{figure}
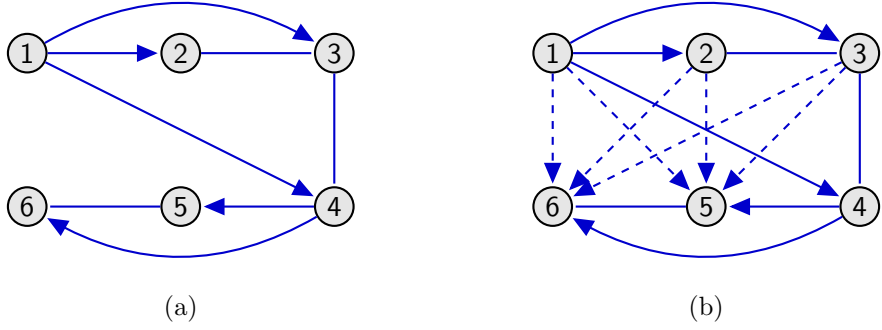

In the following, we introduce $\bar{\mathcal{G}}$-regression as a technical tool for establishing a diffeomorphism between the space of sample covariance matrices and the space of parameters in our semiparametric model. This link is the key to analyzing the efficiency of the estimators under consideration. 

Recall that $\mathcal{P}_{\mathcal{G}}$ is the set of observed distributions generated by some linear SEM associated with a causal DAG $\mathcal{D} \in [\mathcal{G}]$, which is characterized by \cref{prop:block-recursive}. 
More explicitly, let $Q_k$ be the law of $\varepsilon_{B_k}$ for $k=1,\dots,K$.
The set of laws is explicitly prescribed as 
\begin{equation} \label{eqs:p-G}
\mathcal{P}_{\mathcal{G}} = \left \{Q_{1}(X_{B_1}) \prod_{k=2}^{K} Q_{k} \left(X_{B_k} - \Lambda_{B_{[k-1]},B_{k}}^{\T} X_{B_{[k-1]}} \right): Q_k \in \mathcal{P}_{\mathcal{G}_{B_k}},\ i \rightarrow j \text{ not in } \mathcal{G} \Rightarrow \lambda_{ij} = 0 \right \}
\end{equation}
where the law is indexed by $\Lambda = (\lambda_{ij})$ and $(Q_k)_{k=1}^{K}$. This is a \emph{semiparametric} model and $(Q_k)_{k}$ is an infinite-dimensional nuisance \citep[Chap.~25]{van2000asymptotic}.

Consider the set of laws $\mathcal{P}_{\bar{\mathcal{G}}}$ associated with the saturated graph. Let $\Omega_k := \E_{Q_k} \varepsilon \varepsilon^{\T}$ be the covariance of $Q_k$ for $k=1,\dots,K$. Let $\PD{n}$ denote the set of $n \times n$ symmetric, positive definite matrices. By our assumption, $\Omega_k \in \PD{|B_k|}$. Also, consider the coefficients $\Lambda = (\lambda_{ij})$ such that $\lambda_{ij} \neq 0$ only if $i \rightarrow j$ in $\bar{\mathcal{G}}$, or equivalently, $i \in B_l$ and $j \in B_m$ for $l < m$. 
Then, the covariance of $X$, denoted as $\Sigma$, under any $P \in \mathcal{P}_{\bar{\mathcal{G}}}$ is determined from $(\Omega_k)_{k}$ and $\Lambda$. Let us write this \emph{covariance map} as 
\begin{equation*}
\Sigma = \phi_{\bar{\mathcal{G}}}\left((\Lambda_k)_{k=2}^{K}, (\Omega_k)_{k=1}^K \right),
\end{equation*}
where $\Lambda_k = \Lambda_{B_{[k-1]}, B_k}$ is of dimension $(|B_1| + \dots + |B_{k-1}|) \times |B_k|$. It follows from \cref{cor:block-recursive} that the covariance map $\phi_{\bar{\mathcal{G}}}$ is explicitly given by 
\begin{equation} \label{eqs:cov-map}
\Sigma_{B_1} = \Omega_1, \quad \Sigma_{B_k} = \Lambda_k^{\T} \Sigma_{B_{[k-1]}} \Lambda_k + \Omega_k, \quad \Sigma_{B_{[k-1]}, B_k} = \Sigma_{B_{[k-1]}} \Lambda_k, \quad k=2,\dots,K.
\end{equation}
Further, the covariance map $\phi_{\bar{\mathcal{G}}}$ is a \emph{diffeomorphism} between its domain and the set of $|V| \times |V|$ positive definite matrices.
\begin{lemma} \label{lem:diffeo}  
Covariance map $\phi_{\bar{\g}}$ given by \cref{eqs:cov-map} is invertible. Further, $\left( (\Lambda_k)_{k=2}^{K}, (\Omega_k)_{k=1}^{K} \right) \leftrightarrow \Sigma$ given by $\phi_{\bar{\g}}$ and its inverse $\phi^{-1}_{\bar{\g}}$ is a diffeomorphism between $\left(\bigtimes_{k=2}^{K} \mathbb{R}^{(|B_1| + \dots + |B_{k-1}|) \times |B_k|}\right) \times \left(\bigtimes_{k=1}^{K} \PD{|B_k|}\right)$ and $\PD{|V|}$.
\end{lemma}
\begin{proof}
By definition, covariance map $\phi_{\bar{\g}}$ is differentiable. To show diffeomorphism, we need to show that $\phi_{\bar{\mathcal{G}}}^{-1}(\Sigma)$ exists for every $\Sigma \in \PD{|V|}$ and that $\phi_{\bar{\mathcal{G}}}^{-1}$ is differentiable.
For any positive definite $\Sigma$, the inverse covariance map $\phi_{\bar{\mathcal{G}}}^{-1}(\Sigma)$ is explicitly given by
\begin{equation} \label{eqs:full-regression}
\Lambda_k = \left(\Sigma_{B_{[k-1]}} \right)^{-1} \Sigma_{B_{[k-1]},B_k}, \quad k=2, \dots, K,
\end{equation}
and 
\begin{equation} \label{eqs:inverse-cov-map-omega}
\Omega_k = \Sigma_{B_k \cdot B_{[k-1]}} = \Sigma_{B_k} - \Sigma_{B_{[k-1]},B_k}^{\T} \Sigma_{B_{[k-1]}}^{-1} \Sigma_{B_{[k-1]},B_k}, \quad k=1, \dots, K,
\end{equation}
where $\Sigma_{B_k \cdot B_{[k-1]}}$ is the Schur complement of block $B_k$ with respect to block $B_{[k-1]}$. Because $\Sigma$ is positive definite, Schur complement $\Omega_k$ is also positive definite \citep[page 495]{horn2012matrix}. Clearly, the map $\phi_{\bar{\mathcal{G}}}^{-1}(\cdot)$ is differentiable.
\end{proof}
By \cref{eqs:full-regression,eqs:inverse-cov-map-omega}, $\Lambda_k$ is the matrix of population least squares coefficients in a regression of $X_{B_k}$ onto $X_{B_1 \cup \dots \cup B_{k-1}}$ according to $\bar{\mathcal{G}}$, and $\Omega_k$ is the corresponding covariance of regression residuals. Hence, $\phi^{-1}_{\bar{\g}}(\Sigma)$ is called ``$\bar{\g}$-regression''. 

\begin{remark}
In the special case when $\g$ is a DAG such that every bucket $B_i$ is a singleton, \cref{lem:diffeo} reduces to $(\Lambda, \omega) \leftrightarrow \Sigma$ given by $(\phi_{\bar{\g}}, \phi_{\bar{\g}}^{-1})$ being a diffeomorphism between
\begin{equation*}
\left\{\Lambda \in \mathbb{R}^{|V| \times |V|}: \text{$\Lambda$ is upper-triangular} \right\} \times \left\{\omega \in \mathbb{R}^{|V|}: \omega_i > 0, \,i=1,\dots,|V| \right\} \,\longleftrightarrow\, \PD{|V|}.
\end{equation*}
The covariance map is $\Sigma =\phi_{\bar{\g}}(\Lambda, \omega) = (I - \Lambda)^{-\T}\diag(\omega)(I-\Lambda)^{-1}$, and the inverse covariance map $\phi_{\bar{\g}}^{-1}$ is given by the unique LDL decomposition of $\Sigma^{-1}$. \cref{lem:diffeo} is a generalization of \citet[Theorem 7.2]{drton2018algebraic}.
\end{remark}

\subsection{Covariance-based, consistent estimators} \label{sec:class-of-estimators}
We now characterize the class of estimators relative to which the optimality of our estimator is established. Recall that under $P \in \mathcal{P}_{\mathcal{G}}$, $\hat{\Sigma} = \hat{\Sigma}^{(n)}$ is the sample covariance, $\Sigma$ is the population covariance and $\tau_w = w^{\T} \tau_{AY}$. We assume that $n > \max_{k} \{|B_k| + |\Pa(B_k, \g)|\}$ such that $\hat{\Sigma}^{(n)}$ is positive definite almost surely \citep[Sec. 3.1]{drton2006maximum}. For simplicity, the superscript $(n)$ is often omitted. 

\begin{definition} \label{def:class-of-estimators}
The class of estimators for $\tau_w$ under consideration is 
\begin{multline} \label{eqs:class-of-estimators}
\mathcal{T}_{w} := \bigg\{\hat{\tau}_w\left(\hat{\Sigma}^{(n)}\right): \PD{|V|} \rightarrow \mathbb{R}: \\
 \hat{\tau}_w \text{ differentiable},\  \hat{\tau}_w(\hat{\Sigma}^{(n)}) \rightarrow_{p} \tau_{w}(P) \text{ as $n \rightarrow \infty$ under every $P \in \mathcal{P}_{\mathcal{G}}$}\bigg\}.
\end{multline}
\end{definition}
By definition, in particular, $\mathcal{T}_{w}$ includes all regular estimators computable with least squares operations.

\subsubsection{Characterizing $\mathcal{T}_w$} 
Let $(\hat{\Lambda}_k^{\bar{\g}})_{k=2}^{K}, (\hat{\Omega}_k^{\bar{\g}})_{k=1}^{K}$ be the image of $\hat{\Sigma}$ under $\phi_{\bar{\mathcal{G}}}^{-1}$. 
Recall that $(\Lambda_k)_{k=2}^{K}, (\Omega_k)_{k=1}^K$ is the image of $\Sigma$ under $\phi_{\bar{\mathcal{G}}}^{-1}$.
For a matrix $C$, let $\vect C$ denote vectorizing $C$ by concatenating its columns.
Each $\vect \hat{\Lambda}_k^{\bar{\g}}$ can be split by \emph{coordinates} into vectors
\begin{equation} \label{eqs:split-subscript}
\hat{\Lambda}_{k,\g}^{\bar{\g}} = \left(\hat{\lambda}^{\bar{\g}}_{ij}: j \in B_k, i \in \Pa(B_k, \g) \right), \quad \hat{\Lambda}_{k,\g^c}^{\bar{\g}} = \left(\hat{\lambda}^{\bar{\g}}_{ij}: j \in B_k, i \in \Pa(B_k, \bar{\g}) \setminus \Pa(B_k, \g) \right),
\end{equation}
where $\left(\hat{\Lambda}_{k,\g}^{\bar{\g}}\right)_k$ corresponds to between-bucket edges in $\g$ and $\left(\hat{\Lambda}_{k,\g^c}^{\bar{\g}}\right)_k$ corresponds to between-bucket edges in $\bar{\g}$ but not in $\g$.
In the example of \cref{fig:PDAG-and-saturated}, we have $\hat{\Lambda}_{2,\g}^{\bar{\g}} = (\hat{\lambda}^{\bar{\g}}_{12}, \hat{\lambda}^{\bar{\g}}_{13}, \hat{\lambda}^{\bar{\g}}_{14})^{\T}$, $\hat{\Lambda}_{3,\g}^{\bar{\g}}=(\hat{\lambda}^{\bar{\g}}_{45}, \hat{\lambda}^{\bar{\g}}_{46})^{\T}$ and $\hat{\Lambda}_{2,\g^c}^{\bar{\g}}=\texttt{NULL}$, $\hat{\Lambda}_{3,\g^c}^{\bar{\g}} = (\hat{\lambda}^{\bar{\g}}_{15}, \hat{\lambda}^{\bar{\g}}_{16}, \hat{\lambda}^{\bar{\g}}_{25}, \hat{\lambda}^{\bar{\g}}_{26}, \hat{\lambda}^{\bar{\g}}_{35}, \hat{\lambda}^{\bar{\g}}_{36})^{\T}$. Similarly, $\vect \Lambda_k$ can be split into $\Lambda_{k,\g}$ and $\Lambda_{k,\g^c}$ for $k=2,\dots,K$. 

The following lemma directly follows from \cref{def:class-of-estimators} and \cref{lem:diffeo}.

\begin{lemma} \label{lem:estimator-repr}
An estimator $\hat{\tau}_w \in \mathcal{T}_w$ can be written as 
\begin{equation*}
\hat{\tau}_w\left(\hat{\Sigma}^{(n)} \right) = \hat{\tau}_w\left( (\hat{\Lambda}^{\bar{\g}}_{k,\mathcal{G}})_{k=2}^{K},\, (\hat{\Lambda}^{\bar{\g}}_{k,\mathcal{G}^c})_{k=2}^{K},\,(\hat{\Omega}^{\bar{\g}}_k)_{k=1}^{K} \right)
\end{equation*}
for function $\hat{\tau}_w\left( (\hat{\Lambda}^{\bar{\g}}_{k,\mathcal{G}})_{k=2}^{K},\, (\hat{\Lambda}^{\bar{\g}}_{k,\mathcal{G}^c})_{k=2}^{K},\, (\hat{\Omega}^{\bar{\g}}_k)_{k=1}^{K} \right)$ that is differentiable in its arguments.
\end{lemma}

The consistency of $\hat{\tau}_w$ implies the following two results.
\begin{lemma} \label{lem:consistency}
For any $\hat{\tau}_w \in \mathcal{T}_w$, it holds that
\begin{equation}
\hat{\tau}_w \left((\Lambda_{k,\mathcal{G}})_{k=2}^K,(\bm{0})_{k=2}^{K},(\Omega_k)_{k=1}^{K} \right) \equiv \tau_{w} \left((\Lambda_{k,\mathcal{G}})_{k=2}^{K} \right) \label{eqs:consistency-cond}
\end{equation}
for all $(\Lambda_{k,\mathcal{G}})_k$ and all positive definite $(\Omega_k)_k$. 
\end{lemma}
\begin{proof}
Under any $P \in \mathcal{P}_{\mathcal{G}}$, since $\hat{\Sigma} \rightarrow_p \Sigma$ as $n \rightarrow \infty$ by the law of large numbers, by \cref{lem:diffeo} and the continuous mapping theorem \citep[page 11]{van2000asymptotic}, we have $\hat{\Lambda}_{k,\mathcal{G}}^{\bar{\g}} \rightarrow_{p} \Lambda_{k,\mathcal{G}}$, $\hat{\Lambda}^{\bar{\g}}_{k,\mathcal{G}^c} \rightarrow_{p} \bm{0}$ and $\hat{\Omega}^{\bar{\g}}_k \rightarrow_{p} \Omega_k$ for $k=2,\dots,K$. By \cref{lem:estimator-repr} and continuous mapping again, $\hat{\tau}_w \rightarrow_{p} \hat{\tau}_w \left((\Lambda_{k,\mathcal{G}})_{k=2}^K,(\bm{0})_{k=2}^{K},(\Omega_k)_{k=1}^{K} \right)$. The result then follows from the consistency of $\hat{\tau}_w$ under every $P \in \mathcal{P}_{\mathcal{G}}$. 
\end{proof}

\begin{corollary} \label{cor:gradient-conditions}
For $\hat{\tau}_w \in \mathcal{T}_{w}$, at any $((\Lambda_{k,\mathcal{G}})_{k=2}^K,(\bm{0})_{k=2}^{K},(\Omega_k)_{k=1}^{K})$, it holds that 
\begin{equation} \label{eqs:cond-gradients}
\frac{\partial \hat{\tau}_w}{\partial \Lambda_{k,\mathcal{G}}} = \frac{\partial \tau_{w}}{\partial \Lambda_{k,\mathcal{G}}} \ (k=2,\dots,K), \quad \frac{\partial \hat{\tau}_w}{\partial \Omega_{k}} = \bm{0} \ (k=1,\dots,K).
\end{equation}
\end{corollary}
\begin{proof}
Let symbol $\langle \cdot, \cdot \rangle$ denote inner product. Since $\hat{\tau}_w$ is differentiable (\cref{lem:estimator-repr}), by a Taylor expansion at $((\Lambda_{k,\mathcal{G}})_{k=2}^K,(\bm{0})_{k=2}^K,(\Omega_k)_{k=1}^K)$, we have
\begin{equation*}
\begin{split}
&\quad \hat{\tau}_w \left((\Lambda_{k,\mathcal{G}} + \Delta \Lambda_{k,\mathcal{G}})_{k=2}^K,(\bm{0})_{k=2}^K,(\Omega_k + \Delta \Omega_k)_{k=1}^K \right) - \hat{\tau}_w\left((\Lambda_{k,\mathcal{G}})_{k=2}^K,(\bm{0})_{k=2}^K,(\Omega_k)_{k=1}^K \right)\\
&= \sum_{k=2}^{K} \left( \left\langle \frac{\partial \hat{\tau}_w}{\partial \Lambda_{k,\mathcal{G}}}, \Delta \Lambda_{k,\mathcal{G}} \right\rangle + o(\|\Delta \Lambda_{k,\mathcal{G}}\|) \right )+ \sum_{k=1}^{K} \left(\left\langle \frac{\partial \hat{\tau}_w}{\partial \Omega_{k}}, \Delta \Omega_k \right \rangle  + o(\|\Delta \Omega_{k}\|) \right),
\end{split}
\end{equation*}
which by \cref{lem:consistency} must equal $\tau_{w}((\Lambda_{k,\mathcal{G}} + \Delta \Lambda_{k,\mathcal{G}})_{k=2}^K) - \tau_{w}((\Lambda_{k,\mathcal{G}})_{k=2}^K)$. The result then follows from the differentiability of $\tau_w(\cdot)$ and the definition of derivatives. 
\end{proof}

Note that \cref{cor:gradient-conditions} is similar to the conditions imposed on influence functions in standard semiparametric efficiency theory; see, e.g., \citet[Corollary 1, \S 3.1]{tsiatis2007semiparametric}. However, the gradients $ \partial \hat{\tau}_w/ \partial \hat{\Lambda}^{\bar{\g}}_{k,\mathcal{G}^{c}}$ for $k=2,\dots,K$ are \emph{free to vary} because $\hat{\Lambda}^{\bar{\g}}_{k,\mathcal{G}^c} \rightarrow_{p} \bm{0}$. That is, an estimator $\hat{\tau}_w \in \mathcal{T}_w$ can take arbitrary values as its second argument varies in the vicinity of zero, as long as differentiability is maintained.

\subsection{Asymptotic covariance of least squares coefficients}
We use this section to derive some asymptotic results that will be used to prove \cref{thm:optimal-efficiency}.

Consider a vertex $j \in B_k$ for $k \in \{2, \dots , K\}$ and a set of vertices $C$ such that $ \Pa(B_k, \mathcal{G}) \subseteq C \subseteq \Pa(B_k, \bar{\mathcal{G}})$. Let $\hat{\lambda}_{C,j}^{(n)} \in \mathbb{R}^{|C|}$ be the least squares coefficients from regressing $X_j$ onto $X_C$ under sample size $n$. Let $\lambda_{C,j}$ be the corresponding true edge coefficient vector from $\Lambda$ in \cref{prop:block-recursive}. Then $\lambda_{C,j}$ has non-zero coordinates only for those indices in   $\Pa(B_k, \mathcal{G})$. Because $X_j = \lambda_{C,j}^{\T} X_C + \varepsilon_j$ with $\varepsilon_j \indep X_C$ by \cref{cor:block-recursive}, we have $\hat{\lambda}_{C,j}^{(n)} \rightarrow_{p} \lambda_{C,j}$ under every $P \in \mathcal{P}_{\mathcal{G}}$. Moreover, we have the following asymptotic linear expansion. 

\begin{lemma} \label{lem:RAL}
Let $j$ be a vertex in bucket $B_k$ for $k \in \{2,\dots,K\}$. Let $C$ be a set of vertices such that $\Pa(B_k, \g) \subseteq C \subseteq \Pa(B_k, \bar{\g})$. Under any $P \in \mathcal{P}_{\g}$, it holds that
\begin{equation*}
\hat{\lambda}_{C,j}^{(n)} - \lambda_{C, j} = \frac{1}{n} \sum_{i=1}^{n} (\Sigma_{C})^{-1} X_{C}^{(i)} \varepsilon_{j}^{(i)} + O_{p}(n^{-1}),
\end{equation*}
where $\Sigma = \E_{P} X X^{\T}$, $\hat{\lambda}_{C,j}^{(n)}$ is the vector of least squares coefficients from regressing $X_j$ onto $X_C$ under sample size $n$, and $\lambda_{C,j}$ is the vector of true coefficients in \cref{prop:block-recursive}.
\end{lemma}

We now use \cref{lem:RAL} to obtain the covariance structure of $\bar{\mathcal{G}}$-regression coefficients $(\hat{\Lambda}_{k}^{\bar{\g}})_{k=2}^{K}$. Recall that $\hat{\Lambda}_{k}^{\bar{\g}} \in \mathbb{R}^{|B_{[k-1]}| \times |B_k|}$ with $B_{[k-1]} = B_1 \cup \dots \cup B_{k-1}$ and 
\begin{equation*}
\left( (\hat{\Lambda}_{k}^{\bar{\g}})_{k=2}^{K}, (\hat{\Omega}_{k}^{\bar{\g}})_{k=1}^{K} \right) = \phi_{\bar{\mathcal{G}}}^{-1} \left( \hat{\Sigma}^{(n)} \right),
\end{equation*}
as given by \cref{eqs:full-regression,eqs:inverse-cov-map-omega}. For matrices $A \in \mathbb{R}^{m \times n}, B \in \mathbb{R}^{p \times q}$, the Kronecker product $A \otimes B$ is an $mp \times nq$ matrix given by
\[A \otimes B = \begin{pmatrix} a_{11}B &\cdots &a_{1n}B \\
 \vdots & \ddots & \vdots \\
a_{m1}B &\cdots &a_{mn}B
\end{pmatrix}. \]

\begin{lemma} \label{lem:cov-G-bar-regression}
Let $(\hat{\Lambda}_{k}^{\bar{\g}})_{k=2}^{K}$ be the $\bar{\mathcal{G}}$-regression coefficients under sample size $n$. Under any $P \in \mathcal{P}_{\mathcal{G}}$, it holds that 
\begin{equation*}
\sqrt{n} \begin{pmatrix} \vect (\hat{\Lambda}_{2}^{\bar{\g}} - \Lambda_{2}) \\
\vdots \\
\vect (\hat{\Lambda}_{K}^{\bar{\g}} - \Lambda_{K})
\end{pmatrix} \distconvto \N\left(\bm{0},\, \diag\left\{ \Omega_2 \otimes \left(\Sigma_{B_{[1]}}\right)^{-1}, \dots, \Omega_{K} \otimes \left(\Sigma_{B_{[K-1]}} \right)^{-1} \right\} \right).
\end{equation*}
\end{lemma}

\begin{remark}
$\sqrt{n} \vect (\hat{\Lambda}_{k}^{(n)} - \Lambda_{k}) \distconvto \N\left(\bm{0},\, \Omega_{k} \otimes \left(\Sigma_{B_{[k-1]}} \right)^{-1}\right)$ is equivalent to
\begin{equation*}
\sqrt{n} (\hat{\Lambda}_{k}^{(n)} - \Lambda_{k}) \distconvto \mathcal{MN}\left(\bm{0}, \left(\Sigma_{B_{[k-1]}} \right)^{-1}, \Omega_{k} \right),
\end{equation*}
where the RHS is a centered matrix normal distribution with row covariance $(\Sigma_{B_{[k-1]}})^{-1}$ and column covariance $\Omega_{k}$; see \citet{dawid1981some}.
\end{remark}

Similarly, we can compute the asymptotic covariance of the $\mathcal{G}$-regression coefficients. To obtain the result below, we rely on the restrictive property of $\mathcal{G}$ (\cref{cor:restrictive}). 

\begin{lemma} \label{lem:cov-G-regression}
Let $(\hat{\Lambda}_{k}^{\mathcal{G}})_{k=2}^{K}$ be the $\mathcal{G}$-regression coefficients as defined in \cref{lem:MLE-G-regression} under sample size $n$. Under any $P \in \mathcal{P}_{\mathcal{G}}$, it holds that 
\begin{equation*} \label{eqs:G-CLT}
\sqrt{n} \begin{pmatrix} \vect (\hat{\Lambda}_{2}^{\mathcal{G}} - \Lambda_{2}) \\
\vdots \\
\vect (\hat{\Lambda}_{K}^{\mathcal{G}} - \Lambda_{K})
\end{pmatrix} \distconvto \N\left(\bm{0},\, \diag\left\{ \Omega_2 \otimes \left(\Sigma_{\Pa(B_2,\mathcal{G})}\right)^{-1}, \dots, \Omega_{K} \otimes \left(\Sigma_{\Pa(B_K,\mathcal{G})}\right)^{-1} \right\} \right).
\end{equation*}
\end{lemma}

\subsection{Efficiency bound} \label{sec:effic-bound}
We first notice a simple fact of the quadratic form and a property of the Kronecker product.
\begin{lemma} \label{lem:quadratic-form}
Let $S  \in \PD{n},  x \in \mathbb{R}^n$ and suppose that $(A, B)$ is a partition of the set $\{1,\dots,n\}$. For any fixed $x_{A}$, it holds that
\begin{equation*}
x^{\T} S x \ge x_{A}^{\T} (S_{A \cdot B}) x_{A},
\end{equation*}
where $S_{A \cdot B} = S_{A,A} - S_{A,B} S_{B,B}^{-1} S_{B,A}$. The equality holds if and only if $x_{B} = -S_{B,B}^{-1} S_{B,A} x_{A}$. 
\end{lemma}

\begin{lemma}[{\citet[Theorem 1]{liu1999some}}] \label{lem:kronecker-partial} 
Let $A \in \mathbb{R}^{m \times m}$ and $C \in \mathbb{R}^{n \times n}$ be non-singular. Suppose $\alpha \subset [m]$, $\beta \subset [n]$. Let $\alpha^{c}$, $\beta^{c}$ denote their respective complements. Let $\gamma^{c} = \{n(i-1) + j: i \in \alpha^c, j \in \beta^c\}$ and $\gamma = [mn] \setminus \gamma^{c}$. We have
\begin{equation*}
A_{\alpha^{c} \cdot \alpha} \otimes C_{\beta^{c} \cdot \beta} = (A \otimes C)_{\gamma^{c} \cdot \gamma}.
\end{equation*}
\end{lemma}

\begin{lemma} \label{lem:efficiency-bound}
Suppose the assumptions of \cref{thm:optimal-efficiency} hold. Fix $w \in \mathbb{R}^{|A|}$ and let $\tau_{w} = w^{\T} \tau_{AY} = \tau_w((\Lambda_{k,\mathcal{G}})_{k=2}^K)$ as in \cref{eqs:tau-w}. 
Consider any estimator $\hat{\tau}_w \in \mathcal{T}_{w}$ given by \cref{def:class-of-estimators}. Then under any $P \in \mathcal{P}_{\mathcal{G}}$, it holds that
\begin{equation} \label{eqs:efficiency-bound}
\avar(\hat{\tau}_w) \geq \sum_{k=2}^{K} h_k^{\T} \Omega_{k} \otimes (\Sigma_{\Pa(B_k, \mathcal{G})})^{-1} h_k,
\end{equation}
where $(\Omega_k)_{k=2}^{K}$ and $\Sigma$ are determined by $P$, and the gradient vectors $h_k = {\partial \tau_w((\Lambda_{k,\mathcal{G}})_k)} / {\partial \Lambda_{k,\mathcal{G}}}$ for $k=2,\dots,K$ evaluated at $(\Lambda_{k,\mathcal{G}})_k$ are determined by $\tau_w(\cdot)$ and $P$.
\end{lemma}

\begin{proof}
By \cref{lem:estimator-repr}, estimator $\hat{\tau}_w \in \mathcal{T}_{w}$ can be written as
\begin{equation*}
\hat{\tau}_w = \hat{\tau}_w\left( (\hat{\Lambda}^{\bar{\g}}_{k,\mathcal{G}})_{k=2}^{K}, \ (\hat{\Lambda}^{\bar{\g}}_{k,\mathcal{G}^c})_{k=2}^{K}, \ (\hat{\Omega}_k^{\bar{\g}})_{k=1}^{K} \right),
\end{equation*}
where the arguments correspond to the image of $\hat{\Sigma}$ under $\phi_{\bar{\mathcal{G}}}^{-1}$; see \cref{eqs:split-subscript}. Estimator $\hat{\tau}_w \in \mathcal{T}_{w}$ is asymptotically normal. By the delta method \citep[Sec 11.2]{shorack2000probability}, we have
\begin{equation*}
\avar (\hat{\tau}_w) = \begin{pmatrix}
\partial \hat{\tau}_w / \partial ({\Lambda}_{k,\mathcal{G}})_{k=2}^K \\
\partial \hat{\tau}_w / \partial ({\Lambda}_{k,\mathcal{G}^{c}})_{k=2}^K \\
\partial \hat{\tau}_w / \partial ({\Omega}_k)_{k=1}^K
\end{pmatrix}^{\T} \acov 
\begin{Bmatrix}
\vect\, (\hat{\Lambda}^{\bar{\g}}_{k,\mathcal{G}})_{k=2}^{K} \\
\vect\, (\hat{\Lambda}^{\bar{\g}}_{k,\mathcal{G}^c})_{k=2}^{K}\\
\vect\, (\hat{\Omega}^{\bar{\g}}_k)_{k=1}^{K}
\end{Bmatrix}
\begin{pmatrix}
\partial \hat{\tau}_w / \partial ({\Lambda}_{k,\mathcal{G}})_{k=2}^K \\
\partial \hat{\tau}_w / \partial ({\Lambda}_{k,\mathcal{G}^{c}})_{k=2}^K \\
\partial \hat{\tau}_w / \partial ({\Omega}_k)_{k=1}^K
\end{pmatrix},
\end{equation*}
where the partial derivatives of $\hat{\tau}_w(\cdot)$ are evaluated at $\left( ({\Lambda}_{k,\mathcal{G}})_{k=2}^{K},\, (\bm{0})_{k=2}^{K},\, ({\Omega}_k)_{k=1}^{K})\right)$, the image of $\Sigma$ under $\phi_{\bar{\g}}^{-1}$.

Using $\partial \hat{\tau}_w / \partial \Omega_{k} = \bm{0}$ for $k=1,\dots,K$ from \cref{cor:gradient-conditions}, it follows that
\begin{equation*}
\begin{split}
\avar (\hat{\tau}_w) &= \begin{pmatrix}
\partial \hat{\tau}_w / \partial ({\Lambda}_{k,\mathcal{G}})_{k=2}^K \\
\partial \hat{\tau}_w / \partial ({\Lambda}_{k,\mathcal{G}^{c}})_{k=2}^K
\end{pmatrix}^{\T} \acov 
\begin{Bmatrix}
\vect\, (\hat{\Lambda}^{\bar{\g}}_{k,\mathcal{G}})_{k=2}^{K} \\
\vect\, (\hat{\Lambda}^{\bar{\g}}_{k,\mathcal{G}^c})_{k=2}^{K}
\end{Bmatrix}
\begin{pmatrix}
\partial \hat{\tau}_w / \partial ({\Lambda}_{k,\mathcal{G}})_{k=2}^K \\
\partial \hat{\tau}_w / \partial ({\Lambda}_{k,\mathcal{G}^{c}})_{k=2}^K
\end{pmatrix} \\
&= \sum_{k=2}^{K} \begin{pmatrix}
\partial \hat{\tau}_w / \partial {\Lambda}_{k,\mathcal{G}} \\
\partial \hat{\tau}_w / \partial {\Lambda}_{k,\mathcal{G}^{c}}
\end{pmatrix}^{\T} \acov 
\begin{Bmatrix}
\hat{\Lambda}^{\bar{\g}}_{k,\mathcal{G}} \\
\hat{\Lambda}^{\bar{\g}}_{k,\mathcal{G}^c}
\end{Bmatrix}
\begin{pmatrix}
\partial \hat{\tau}_w / \partial {\Lambda}_{k,\mathcal{G}} \\
\partial \hat{\tau}_w / \partial {\Lambda}_{k,\mathcal{G}^{c}}
\end{pmatrix},
\end{split}
\end{equation*}
where we have used the block-diagonal structure of the asymptotic covariance from \cref{lem:cov-G-bar-regression}. Let
\[ S^{(k)} := \acov \begin{Bmatrix}
\hat{\Lambda}^{\bar{\g}}_{k,\mathcal{G}} \\
\hat{\Lambda}^{\bar{\g}}_{k,\mathcal{G}^c}
\end{Bmatrix}, \quad k=2,\dots,K,  \]
which equals 
\[ S^{(k)} = \Omega_{k} \otimes \left(\Sigma_{B_{[k-1]}} \right)^{-1}, \quad k =2,\dots,K, \]
by \cref{lem:cov-G-bar-regression}. From \cref{cor:gradient-conditions}, note that $\partial \hat{\tau}_w / \partial ({\Lambda}_{k,\mathcal{G}})_k \equiv h_k$ is \emph{fixed} for $k=2,\dots,K$. Then, \cref{lem:quadratic-form} yields the lower bound 
\begin{equation*}
\avar(\hat{\tau}_w) \geq \sum_{k=2}^{K} h_k^{\T} S^{(k)}_{\mathcal{G} \cdot \mathcal{G}^c} h_k,
\end{equation*}
where indices $\g$ and $\g^c$ correspond to the coordinates in $\hat{\Lambda}^{\bar{\g}}_{k,\mathcal{G}}$ and $\hat{\Lambda}^{\bar{\g}}_{k,\mathcal{G}^c}$ respectively. 
Indices $\g$ correspond to $\{(i,j): j \in B_k, i \in \Pa(B_k, \g)\}$; by construction of $\bar{\g}$, indices $\g^c$ correspond to $\{(i,j): j \in B_k, i \in \Pa(B_k, \bar{\g}) \setminus \Pa(B_k, \g)\}$. 
Now, to abuse the notation slightly, we apply \cref{lem:kronecker-partial} with
\[ A = \Omega_k,\quad C = (\Sigma_{B_{[k-1]}})^{-1}, \quad \alpha = \emptyset, \quad \beta = \Pa(B_k, \bar{\g}) \setminus \Pa(B_k, \g), \]
such that 
\[ \alpha^c = \{1,\dots,|B_k|\}, \quad \beta^c = \Pa(B_k, \g), \quad \gamma = \g^c, \quad \gamma^c = \g. \]
We obtain
\begin{equation*}
S^{(k)}_{\mathcal{G} \cdot \mathcal{G}^c} = \Omega_k \otimes \left[(\Sigma_{B_{[k-1]}})^{-1}\right]_{\beta^{c} \cdot \beta} = \Omega_{k} \otimes \left(\Sigma_{\Pa(B_k, \g)}\right)^{-1},
\end{equation*}
where the last step follows from $(H^{-1})_{\beta^{c} \cdot \beta} = (H_{\beta^c, \beta^c})^{-1}$ \citep[\S 0.8]{horn2012matrix}.
\end{proof}

\subsection{Efficiency of $\mathcal{G}$-regression estimator} \label{sec:g-reg-effic}
In \cref{sec:MLE}, we have seen that when the errors are Gaussian, the $\mathcal{G}$-regression plugin is the MLE and hence achieves the efficiency bound. Here, we show that this is still true relative to the class of estimators we consider, even though the errors are not necessarily Gaussian. We verify that $\hat{\tau}_w^{\mathcal{G}} = w^{\T} \hat{\tau}_{AY}^{\mathcal{G}}$ achieves the efficiency bound above. 

\begin{lemma} \label{lem:efficiency-G-plugin}
Let $\hat{\tau}_w^{\mathcal{G}} := w^{\T} \hat{\tau}_{AY}^{\mathcal{G}}$, where $\hat{\tau}_{AY}^{\g}$ is the $\g$-regression estimator (\cref{def:G-reg-plugin}). Under the same assumptions as \cref{lem:efficiency-bound}, it holds that $\hat{\tau}_{w}^{\mathcal{G}} \in \mathcal{T}_{w}$ and $\hat{\tau}_w^{\mathcal{G}}$ achieves the efficiency bound in \cref{eqs:efficiency-bound} under every $P \in \mathcal{P}_{\mathcal{G}}$. 
\end{lemma}
\begin{proof}
By \cref{def:G-reg-plugin}, $\hat{\tau}_{w}^{\mathcal{G}} \in \mathcal{T}_{w}$. Further, note that
\begin{equation*}
\hat{\tau}_{w}^{\mathcal{G}} = \tau_w\left((\hat{\Lambda}_{k}^{\mathcal{G}})_{k=2}^{K} \right),
\end{equation*}
where $(\hat{\Lambda}_{k}^{\mathcal{G}})_{k=2}^{K}$ are the $\mathcal{G}$-regression coefficients in \cref{eqs:G-reg}. Under any $P \in \mathcal{P}_{\mathcal{G}}$, we now verify that $\avar \hat{\tau}_{\mathcal{G}}$ matches the RHS of \cref{eqs:efficiency-bound}. By the delta method \citep[Sec 11.2]{shorack2000probability}, we have
\begin{equation*}
\begin{split}
\avar \hat{\tau}_{w}^{\mathcal{G}} &= \left (\partial \tau_{w} / \partial \vect\, (\Lambda_k)_{k=2}^K \right )^{\T} \acov \left\{\vect\, (\hat{\Lambda}_{k}^{\mathcal{G}})_{k=2}^{K} \right\} \left (\partial \tau_{w} / \partial \vect\, (\Lambda_k)_{k=2}^K \right ) \\
&\stackrel{\text{(i)}}{=} \sum_{k=2}^{K} \left (\partial \tau_{w} / \partial \vect \Lambda_{k} \right)^{\T} \acov \left\{\vect \hat{\Lambda}_{k}^{\g} \right\} \left (\partial \tau_{w} / \partial \vect \Lambda_{k} \right) \\
&\stackrel{\text{(ii)}}{=} \sum_{k=2}^{K} \left (\partial \tau_{w} / \partial \vect \Lambda_{k} \right)^{\T} \Omega_k \otimes \left(\Sigma_{\Pa(B_k,\mathcal{G})}\right)^{-1} \left (\partial \tau_{w} / \partial \vect \Lambda_{k} \right),
\end{split}
\end{equation*}
which equals the RHS of \cref{eqs:efficiency-bound}. The partial derivatives of $\tau_w(\cdot)$ are evaluated at $(\Lambda_k)_{k=2}^{K}$. Step (i) follows from the block-diagonal structure of the asymptotic covariance of $\hat{\Lambda}_{\mathcal{G}}$ given by \cref{lem:cov-G-regression}, and (ii) follows from the same lemma.
\end{proof}

Finally, we complete the proof of our main result.
\begin{proof}{of \cref{thm:optimal-efficiency}.}
Fix any $P \in \mathcal{P}_{\mathcal{G}}$. It suffices to show that for every $w \in \mathbb{R}^{|A|}$,
\begin{equation*}
w^{\T} \acov(\hat{\tau}_{AY}) w \geq w^{\T} \acov(\hat{\tau}_{AY}^{\mathcal{G}}) w,
\end{equation*}
or equivalently
\begin{equation*}
\avar \left(w^\T \hat{\tau}_{AY} \right) \geq \avar \left(w^{\T} \hat{\tau}_{AY}^{\mathcal{G}} \right).
\end{equation*}
This is true because for every $\hat{\tau}_{AY}$ in consideration, $\hat{\tau}_w := w^{\T} \hat{\tau}_{AY} \in \mathcal{T}_w$ and hence $\hat{\tau}_w$ is subject to the lower bound in \cref{lem:efficiency-bound}. Meanwhile, by \cref{lem:efficiency-G-plugin}, such a lower bound is achieved by $\hat{\tau}_w^{\mathcal{G}} = w^{\T} \hat{\tau}_{AY}^{\mathcal{G}}$. The proof is complete because the choice of $w$ is arbitrary.
\end{proof}

\begin{remark}
For \cref{thm:optimal-efficiency} to hold, the independence error assumption \cref{eqs:sem-error} of the underlying linear SEM \emph{cannot} be relaxed to uncorrelated errors. This comes from inspecting the proof of \cref{lem:cov-G-bar-regression} in \cref{apx:sec:asymp}. To show that the $\bar{\g}$-regression coefficients are asymptotically independent across buckets, the independence of errors is used to establish that for $2 \leq k < k' \leq K$, $j \in B_k$, $j' \in B_{k'}$, $\cov(\varepsilon_j X_{B_{[k-1]}}, \varepsilon_{j'} X_{B_{[k'-1]}}) = \bm{0}$. 

Suppose for now $\{\epsilon_i: i \in V\}$ are only uncorrelated and hence $\{\varepsilon_{B_k}: k=1,\dots,K\}$ are only uncorrelated across buckets. Further, suppose $B_1 = \{1\}, B_2=\{2\}, B_3=\{3\}$ with $j=k=2$ and $j'=k'=3$. Then, we have
\begin{equation*}
\begin{split}
\cov\left(\varepsilon_j X_{B_{[k-1]}}, \varepsilon_{j'} X_{B_{[k'-1]}}\right) &= \cov\left(\varepsilon_2 \varepsilon_1, \varepsilon_3(\varepsilon_1, \gamma_{12} \varepsilon_1 + \varepsilon_2)^\T\right) \\
&=\E[\varepsilon_1 \varepsilon_2 (\varepsilon_1 \varepsilon_3, \gamma_{12} \varepsilon_1 \varepsilon_3 + \varepsilon_2 \varepsilon_3)^{\T} ],
\end{split}
\end{equation*}
which may be non-zero.
\end{remark}
 \section{Numerical Results} \label{sec:numerical}
In this section, the finite-sample performance of $\mathcal{G}$-regression is evaluated against contending estimators. We use simulations and an \emph{in silico} data set for predicting expression levels in gene knockout experiments. All the numerical experiments were conducted with R v3.6, package \texttt{pcalg} v2.6 \citep{Kalisch:2012aa} and our package \texttt{eff2} v0.1.

\subsection{Simulations} \label{sec:simu}
We compare the performance of $\mathcal{G}$-regression to several contending estimators under finite samples. We roughly follow the simulation setup of \citet{henckel2019graphical,witte2020efficient}. First, we draw a random undirected graph from the Erd{\H{o}}s-R{\'e}nyi model with average degree $k$, where $k$ is drawn from $\{2,3,4,5\}$ uniformly at random. The graph is converted to a DAG $\mathcal{D}$ with a random causal ordering and the corresponding CPDAG $\g$ is recorded. Then we fix a linear SEM by drawing $\gamma_{ij}$ uniformly from $[-2, -0.1] \cup [0.1 ,2]$ and choosing the error distribution uniformly at random from the following:
\begin{enumerate}[itemsep=0mm,topsep=1ex,partopsep=.5ex,parsep=.5ex]
\item $\epsilon_i \sim \N(0, v_i)$ with $v_i \sim \unif(0.5, 6)$,
\item $\epsilon_i / \sqrt{v_i} \sim t_5$ with $v_i \sim \unif(0.5, 1.5)$,
\item $\epsilon_i \sim \text{logistic}(0, s_i)$ with $s_i \sim \unif(0.4, 0.7)$,
\item $\epsilon_i \sim \unif(-a_i, a_i)$ with $a_i \sim \unif(1.2, 2.1)$.
\end{enumerate}
We generate $n$ iid samples from the model. Treatments $A$ of a fixed size are randomly selected from the set of vertices with non-empty descendants, and $Y$ is selected randomly from their descendants; the drawing is repeated until $\tau_{AY}$ is identified from $\g$ according to the criterion of \cref{thm:id-criterion}. Finally, the data and graph $\g$ are provided to each estimator of $\tau_{AY}$. 

We compare to the following three estimators:
\begin{itemize}[itemsep=0mm,topsep=1ex,partopsep=.5ex,parsep=.5ex]
\item \texttt{adj.O}: optimal adjustment estimator \citep{henckel2019graphical},
\item \texttt{IDA.M}: joint-IDA estimator based on modifying Cholesky decompositions \citep{nandy2017estimating},
\item \texttt{IDA.R}: joint-IDA estimator based on recursive regressions \citep{nandy2017estimating}.
\end{itemize}
They are implemented in R package \texttt{pcalg}. The two joint-IDA estimators use the parents of treatment variables to estimate a causal effect. Both of them reduce to the IDA estimator of \citet{maathuis2009estimating} when $|A| = 1$. Admittedly, compared to $\g$-regression and \texttt{adj.O}, the joint-IDA estimators require less knowledge about the graph, namely only $\Pa(i)$ for each $i \in A$.

For each estimator $\hat{\tau}_{AY}$, we compute its squared error $\|\hat{\tau}_{AY} - \tau_{AY}\|_2^2$. Dividing $\|\hat{\tau}_{AY} - \tau_{AY}\|_2^2$ by the squared error of $\mathcal{G}$-regression, we obtain the \emph{relative squared error} of each contending estimator. We consider $|A| \in \{1,2,3,4\}$, $|V| \in \{20, 50, 100\}$ and $n \in \{100, 1000\}$; each configuration of $(|A|, |V|, n)$ is replicated 1,000 times. 

\cref{fig:simu-true-graph} shows the distributions of relative squared errors. In \cref{tab:simu-true-graph}, we summarize the relative errors with their geometric mean and median. Our estimator dominates all the contending estimators in all cases, and the improvement gets larger as $|A|$ gets bigger. Even though \texttt{adj.O} achieves the minimal asymptotic variance among all adjustment estimators, it can compare less favorably to our estimator by several folds. In general, the IDA estimators have very poor performances. Moreover, the results in \cref{tab:simu-true-graph} are computed only from the replications where a contending estimator exists. As mentioned in the Introduction, unlike $\mathcal{G}$-regression, none of the contending estimators is guaranteed to exist for every identified effect under joint intervention (\texttt{adj.O} always exists for point interventions); see \cref{tab:unidentified-true-graph} for the percentages of instances that are not estimable by contending estimators, even though the effect is identified by \cref{thm:id-criterion} and hence estimable by $\g$-regression.

\begin{table}[!htb]
\begin{center}
\caption{Geometric average (brackets: median) of relative squared errors compared to $\g$-regression} \label{tab:simu-true-graph}
\scriptsize
\begin{tabular}{lrlrlcrlrlcrlrl}
\toprule[1.2pt]
&\multicolumn{4}{c}{$|V|=20$}&&\multicolumn{4}{c}{$|V|=50$}&&\multicolumn{4}{c}{$|V|=100$}\\ 
\cline{2-5}\cline{7-10}\cline{12-15}
$|A|$ & \multicolumn{2}{c}{$n=100$} & \multicolumn{2}{c}{$n=1000$} && \multicolumn{2}{c}{$n=100$} & \multicolumn{2}{c}{$n=1000$} && \multicolumn{2}{c}{$n=100$} & \multicolumn{2}{c}{$n=1000$} \\ 
\midrule
\texttt{adj.O}&&&&&&&&\\ 
1&1.3 &(1.0)&1.3 &(1.0)&&1.4 &(1.0)&1.3 &(1.0)&&1.5 &(1.0)&1.5 &(1.0)\\ 
2&3.4 &(1.7)&4.2 &(2.0)&&4.7 &(2.6)&4.9 &(2.6)&&4.2 &(2.6)&4.5 &(2.7)\\ 
3&6.3 &(3.8)&5.9 &(3.4)&&7.4 &(4.9)&7.2 &(4.4)&&7.8 &(5.7)&8.0 &(5.2)\\ 
4&9.3 &(5.0)&9.3 &(5.8)&&12&(8.0)&14&(8.7)&&12&(8.6)&12&(8.9)\\ 
\texttt{IDA.M}&&&&&&&&\\ 
1&20 &(17)&19 &(16)&&61 &(57)&48 &(45)&&103 &(78)&108 &(90)\\ 
2&62 &(48)&65 &(51)&&220 &(153)&182 &(120)&&293 &(205)&356 &(272)\\ 
3&93 &(72)&119 &(108)&&354 &(249)&396 &(205)&&749 &(547)&771 &(604)\\ 
4&154 &(111)&222 &(147)&&533 &(448)&895 &(440)&&1188 &(752)&1604 &(1508)\\ 
\texttt{IDA.R}&&&&&&&&\\ 
1&20 &(17)&19 &(16)&&61 &(57)&48 &(45)&&103 &(78)&108 &(90)\\ 
2&33 &(29)&38 &(29)&&121 &(96)&113 &(89)&&176 &(140)&199 &(168)\\ 
3&30 &(22)&39 &(30)&&171 &(141)&135 &(125)&&342 &(281)&312 &(292)\\ 
4&48 &(34)&50 &(41)&&187 &(132)&214 &(143)&&405 &(391)&432 &(342)\\ 
\bottomrule[1.2pt]
\end{tabular}
\end{center}
\end{table}

\begin{figure}[!htb]
\centering
\includegraphics[width=1\textwidth]{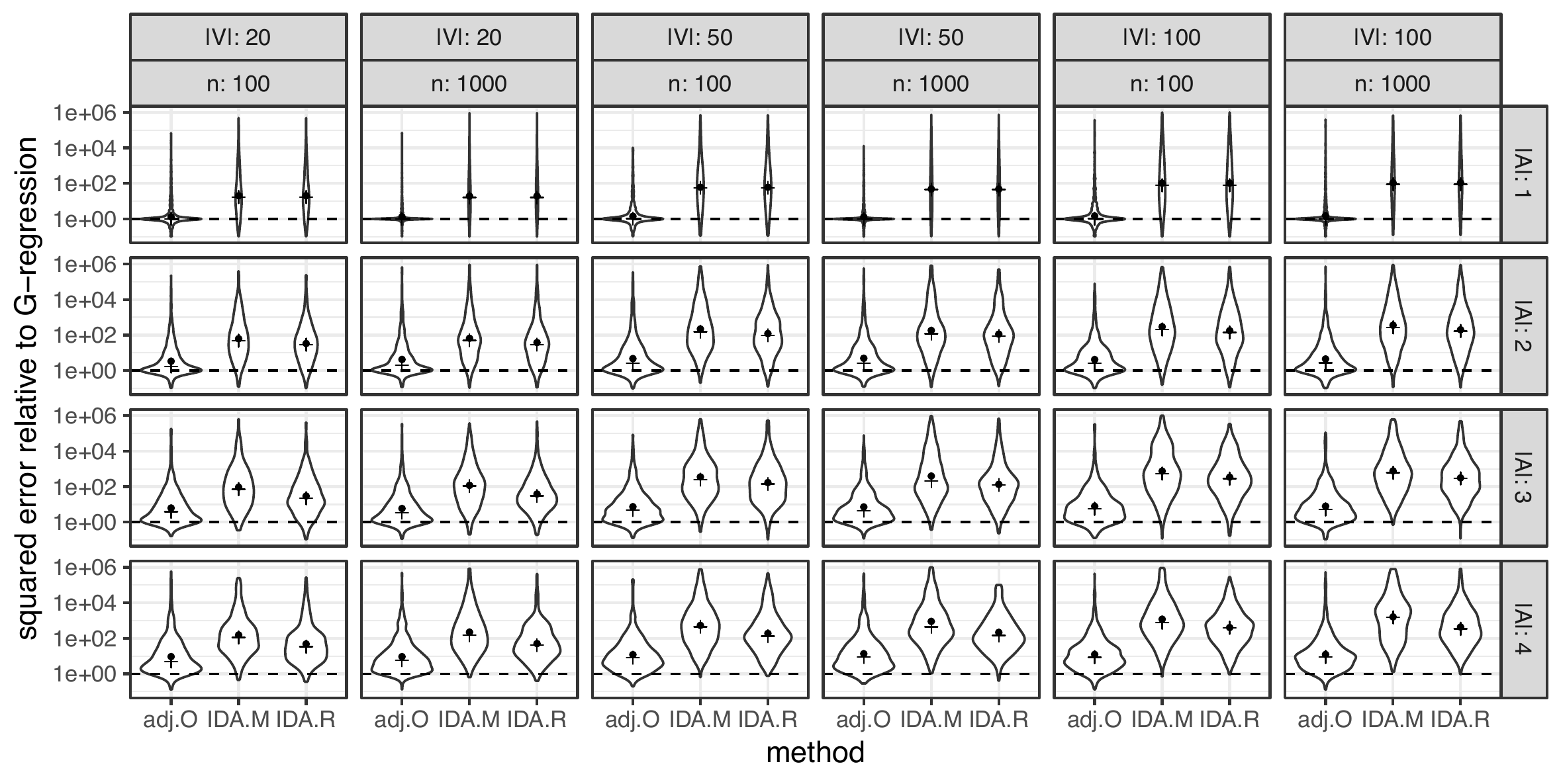}
\caption{Violin plots for the relative squared errors of contending estimators (`$\cdot$': geometric mean, `$+$': median).}
\label{fig:simu-true-graph}
\end{figure}

\begin{table}[!htb]
\begin{center}
\caption{Percentage of identified instances not estimable using contending estimators\\(all estimable with $\g$-regression)} \label{tab:unidentified-true-graph}
\small
\begin{tabular}{ccrrr}
\toprule[1.2pt]
Estimator & $|A|$ & $|V| = 20$ & $|V| = 50$ & $ |V| = 100$\\ 
\midrule
\multirow{4}{*}{\texttt{adj.O}} & 1&0\%&0\%&0\%\\ 
&2&17\%&10\%&5\%\\ 
&3&30\%&18\%&15\%\\ 
&4&36\%&29\%&22\%\\ 
\midrule
\multirow{4}{*}{\texttt{IDA.M}} & 1&29\%&32\%&32\%\\ 
&2&47\%&51\%&50\%\\ 
&3&61\%&59\%&63\%\\ 
&4&72\%&69\%&71\%\\  
\midrule
\multirow{4}{*}{\texttt{IDA.R}} & 1&29\%&32\%&32\%\\ 
&2&47\%&51\%&50\%\\ 
&3&61\%&59\%&63\%\\ 
&4&72\%&69\%&71\%\\ 
\bottomrule[1.2pt]
\end{tabular}
\end{center}
\end{table}

In \cref{apx:sec:simu}, we report two sets of additional simulation results: (1) The CPDAG is estimated with the greedy equivalence search algorithm \citep{chickering2002optimal} and provided to the estimators. The improvements are more modest but are still typically by several folds. (2) The data generating mechanism in this section is modified by rescaling the coefficients such that the variance of a vertex does not grow with the topological ordering. 

\subsection{Predicting double knockouts in DREAM4 data}
The DREAM4 \emph{in silico} network challenge data set \citep{marbach2009generating} provides a benchmark for evaluating the reverse engineering of gene regulation networks. Here we use the 5th \emph{Size10} data set \citep{marbach2009dream4} as our example, which is a small network of 10 genes. \cref{fig:DREAM4} shows the true gene regulation network, which is constructed based on the networks of living organisms. A stochastic differential equation model was used to generate the data under wild type (steady state), perturbed steady state and knockout interventions. A task in the challenge is to use data under wild type and perturbed steady state (both are observational data) to predict the steady state expression levels under 5 different joint interventions, each of which knocks out a pair of genes. For our purpose, we also use the true network as input. However, the true network contains one cycle (other networks in DREAM4  contain more than one cycles). In the following, we remove one edge in the cycle and provide the resulting DAG to the estimators. Necessarily, the causal DAG is misspecified. Results are reported under 4 different edge removals. 

\begin{figure}[!h]
\centering
\includegraphics[width=0.5\textwidth]{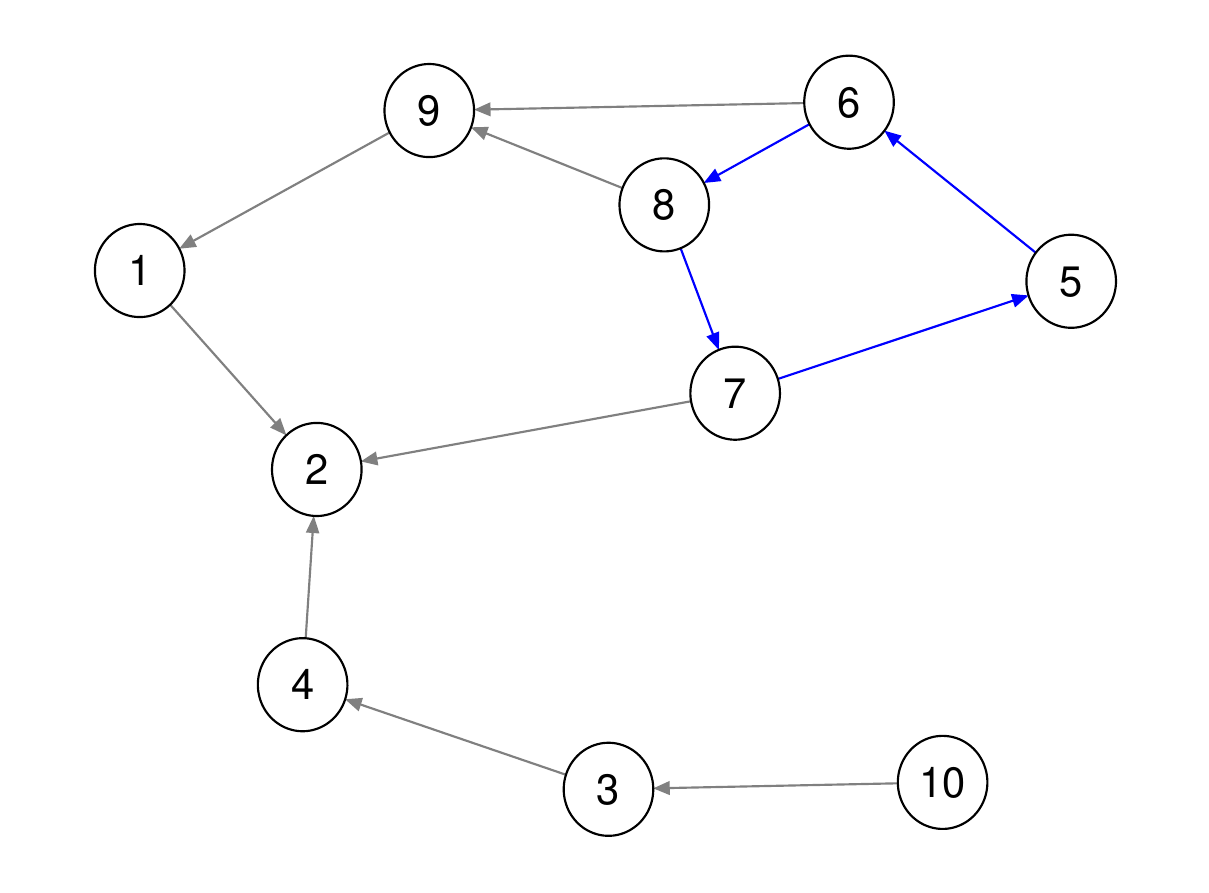}
\caption{Gene regulation network from DREAM4 data set, which contains a cycle (blue).} 
\label{fig:DREAM4}
\end{figure}

Unfortunately, the wild type data only consists of one sample. To estimate the observational covariance, we use the perturbed steady state data, which consists of 5 segments of time series. A sample covariance is computed from each segment, and the final estimate is taken as their average. For a double knockout of genes $(i, j)$, we use $\mathcal{G}$-regression to estimate the joint-intervention effect of $A = (i, j)$ on every other gene. The effect is identified because the DAG is given. For gene $k$, let $s_k$ and $s_k^{(ij)}$ respectively denote its expression level under wild type and double knockout of genes $(i,j)$. The expression level under double knockout is predicted as
\begin{equation*}
\hat{s}_k^{(ij)} = \begin{cases} s_k - (s_i, s_j)^{\T} \hat{\tau}_{ij, k}, \quad& k \notin \{i,j\} \\
0, \quad& k \in \{i,j\} \end{cases}.
\end{equation*}
The performance is evaluated with normalized squared error
\begin{equation*}
\mathcal{E} = \frac{\sum_{(i,j) \in \mathcal{A}} \sum_{k=1}^{10} (\hat{s}_k^{(ij)} - s_k^{(ij)})^2}{\sum_{(i,j) \in \mathcal{A}} \sum_{k=1}^{10} (s_k^{(ij)})^2},
\end{equation*}
where $\mathcal{A} = \{(6,8), (7,8), (8,10), (8,5), (8,9)\}$ consists of 5 double knockouts available in the data set. For comparison, we also evaluate the performance of \texttt{adj.O} (optimal adjustment, \citet{henckel2019graphical}) and \texttt{IDA.R} (joint-IDA based on recursive regressions, \citet{nandy2017estimating}); \texttt{IDA.R} is chosen because it outperforms \texttt{IDA.M} according to \cref{sec:simu}. Unfortunately, \texttt{adj.O} is not able to estimate the effect on every $k$ and a modified metric $\mathcal{E}^{\ast}$ is computed by only summing over those estimable $k$'s; the same metric $\mathcal{E}^{\ast}$ of $\mathcal{G}$-regression is also computed for comparison. As a baseline, we also compute $\mathcal{E}$ from naively estimating $s_k^{(ij)}$ with just $s_k$.

\begin{table}[!ht]
\begin{center}
\caption{Normalized squared errors of predicting gene double knockouts} \label{tab:DREAM4}
\begin{tabular}{@{}cccccccc@{}}
\toprule
edge removed & & \multicolumn{2}{c}{$\mathcal{E}^{\ast}$} & & \multicolumn{3}{c}{$\mathcal{E}$} \\
\cline{3-4}\cline{6-8}
from cycle      & $\nexists$ \texttt{adj.O} & \texttt{adj.O}  & $\mathcal{G}$-\texttt{reg} &  & \texttt{IDA.R} & $\mathcal{G}$-\texttt{reg} & \texttt{baseline}  \\ \midrule
$5 \rightarrow 6$ & 36\%      & 43\% & \textbf{35\%}  &  & 46\% & \textbf{30\%} & 81\% \\
$6 \rightarrow 8$ & 42\%      & \textbf{29\%} & 32\%  &  & 33\% & \textbf{26\%} & 81\% \\
$8 \rightarrow 7$ & 60\%      & 39\% & \textbf{35\%}  &  & 45\% & \textbf{44\%} & 81\% \\
$7 \rightarrow 5$ & 46\%      & 40\% & \textbf{33\%}  &  & 45\% & \textbf{34\%} & 81\% \\ \bottomrule
\end{tabular}
\end{center}
\end{table}

\cref{tab:DREAM4} reports the results, where the column `$\nexists$ \texttt{adj.O}' lists the percentage of effects not estimable by the adjustment estimator. In almost all the cases, $\mathcal{G}$-regression dominates all the contending estimators. In this example, even though both the causal graph and the linear SEM are misspecified, one can still witness some usefulness of our estimator. 
 \section{Discussion}
We have proposed $\mathcal{G}$-regression based on recursive least squares to estimate a total causal effect from observational data, under linearity and causal sufficiency assumuptions. $\mathcal{G}$-regression is applicable to estimating every identified total effect, under either point intervention or joint intervention. Further, via a new semiparametric efficiency theory, we have shown that the estimator achieves the efficiency bound within a restricted, yet reasonably large, class of estimators, including covariate adjustment and other regular estimators based on the sample covariance. Note that the restriction on the class of estimators is motivated by computational simplicity and numerical stability as mentioned in the Introduction. To construct confidence intervals and conduct hypothesis tests, bootstrap can be easily applied to estimate the asymptotic covariance of our estimator. This is implemented in R package \texttt{eff2}.  

\begin{figure}[!htb]
\centering
\begin{tikzpicture}[->,>=triangle 45,shorten >=1pt,
auto,thick,
main node/.style={circle,inner
    sep=1.5pt,fill=gray!20,draw,font=\sffamily}]
\begin{scope} \node[main node] (A) {$A$}; 
\node[main node] (O1) [left=0.5cm of A, yshift=1.2cm] {$O_1$}; 
\node[main node] (O2) [left=0.5cm of A, yshift=-1.2cm] {$O_2$}; 
\node[main node] (Y)  [right=1.2cm of A] {$Y$}; 
\node[below=1.5cm of A] (cap) {(a)}; 
\draw[color=black!20!blue,->] (A) to (Y);
\draw[color=black!20!blue,->] (O1) to (A);
\draw[color=black!20!blue,->] (O1) to (Y);
\draw[color=black!20!blue,->] (O2) to (A);
\draw[color=black!20!blue,->] (O2) to (Y);
\end{scope}
\begin{scope}[xshift=6cm] \node[main node] (A) {$A$}; 
\node[main node] (O1) [left=0.5cm of A, yshift=1.2cm] {$O_1$}; 
\node[main node] (O2) [left=0.5cm of A, yshift=-1.2cm] {$O_2$}; 
\node[main node] (Y)  [right=1.2cm of A] {$Y$}; 
\node[below=1.5cm of A] (cap) {(b)}; 
\draw[color=black!20!blue,->] (A) to (Y);
\draw[color=black!20!blue,->] (O1) to (A);
\draw[color=black!20!blue,->] (O1) to (Y);
\draw[color=black!20!blue,->] (O2) to (A);
\draw[color=black!20!blue,->] (O2) to (Y);
\draw[color=black!20!blue,-] (O1) to (O2);
\end{scope}
\end{tikzpicture}
\caption{(a) and (b) lead to the same $\g$-regression estimator of $\tau_{AY}$. The independence between $O_1$ and $O_2$ in (a) is dropped in (b).}
\label{fig:ignore-indep}
\end{figure}
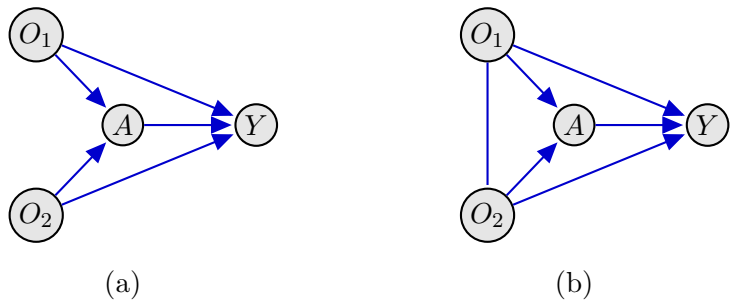

We conclude with a remark. We have seen that the conditional independence constraints in \cref{eqs:block-error-markov} play no role for the restricted class of estimators considered --- a feature that holds under the restrictive property (\cref{cor:restrictive}). For example, the marginal independence between $O_1$ and $O_2$ in \cref{fig:ignore-indep}(a) can be ignored without changing the $\g$-regression estimator of $\tau_{AY}$. However, this is no longer true when linearity is dropped. In terms of nonparametric causal graphical models, the asymptotic relative efficiency resulting from ignoring $O_1 \indep O_2$ can be arbitrarily large; see \citet[Lemma 23]{rotnitzky2019efficient}.
 \acks{Authors thank Thomas Richardson for valuable comments and discussions. The first author was supported by ONR Grant N00014-19-1-2446.}
\newpage
\appendix 
\crefalias{section}{appendix}
\section{Proofs for asymptotic efficiency} \label{apx:sec:asymp}
\begin{figure}[!htb]
\begin{center}
\begin{tikzpicture}[baseline= (a).base]
\node[scale=0.8] (a) at (0,0){
\begin{tikzcd}
                                      & \mathrm{\cref{lem:RAL}} \arrow[r] \arrow[rd] & \mathrm{\cref{lem:cov-G-regression}} \arrow[r]  & \mathrm{~~~~\cref{lem:efficiency-G-plugin}} \arrow[rd] &                \\
                                      &                          & \mathrm{\cref{lem:cov-G-bar-regression}} \arrow[r]  & \mathrm{~~~~\cref{lem:efficiency-bound}} \arrow[r]  & \mathrm{~~~~\cref{thm:optimal-efficiency}}            \\
\mathrm{\cref{lem:diffeo}} \arrow[ruu] \arrow[r] \arrow[rr, bend right] & \mathrm{~~~~\cref{lem:estimator-repr}} \arrow[r] \arrow[rr, bend right]    \arrow[rru]        & \mathrm{~~~~~\cref{lem:consistency}} \arrow[r] & \mathrm{~~~~~~~\cref{cor:gradient-conditions}} \arrow[u]  & \mathrm{\cref{lem:quadratic-form}} \arrow[lu]
\end{tikzcd}
};
\end{tikzpicture}
\caption{Dependency structure of proofs in \cref{sec:efficiency} of main text.}
\label{apx:fig:proof-diagram-eff}
\end{center}
\end{figure}
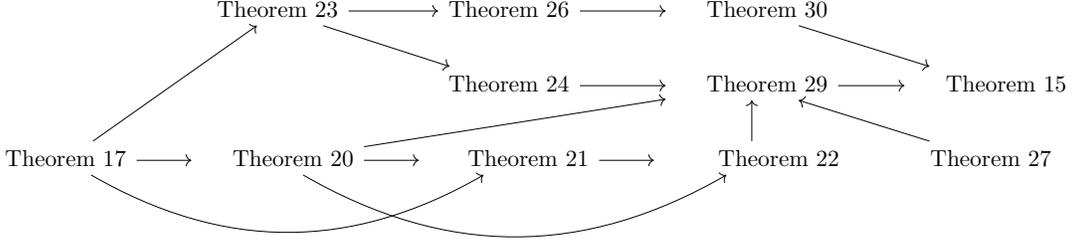

\subsection{Proof of \cref{lem:RAL}}
\begin{proof}
For simplicity, we drop the superscripts in $\hat{\Sigma}^{(n)}$ and $\hat{\lambda}^{(n)}$. Since $j \in B_k$ and $\Pa(B_k, \g) \subseteq C \subseteq \Pa(B_k, \bar{\g})$, we have
\begin{equation*}
\begin{split}
\hat{\lambda}_{C, j} - \lambda_{C, j} &= (\hat{\Sigma}_{C})^{-1} \hat{\Sigma}_{C, j} - (\Sigma_{C})^{-1} \Sigma_{C, j} \\
&= \left(\Sigma_{C} + \hat{\Sigma}_{C} - \Sigma_{C} \right)^{-1} \left(\hat{\Sigma}_{C,j} - \Sigma_{C,j} \right) + \left((\hat{\Sigma}_{C})^{-1} - (\Sigma_{C})^{-1} \right) \Sigma_{C, j}.
\end{split}
\end{equation*}
We compute the two terms separately. The first term becomes
\begin{equation*}
\begin{split}
\left(\Sigma_{C} + \hat{\Sigma}_{C} - \Sigma_{C} \right)^{-1} \left(\hat{\Sigma}_{C,j} - \Sigma_{C,j} \right) &= \left(\Sigma_{C} + O_p(n^{-1/2})\right)^{-1}\left (\hat{\Sigma}_{C,j} - \Sigma_{C,j} \right)\\
&= (\Sigma_{C})^{-1} \left(\hat{\Sigma}_{C,j} - \Sigma_{C,j} \right) + O_p(n^{-1}),
\end{split} 
\end{equation*}
where we used the fact that $\Sigma_{C}$ is positive definite (\cref{lem:diffeo}) and $\|\hat{\Sigma}_{C,j} - \Sigma_{C,j}\| = O_{p}(n^{-1/2})$, $\|\hat{\Sigma}_C - \Sigma_C\|_2 = O_p(n^{-1/2})$ by the central limit theorem.

In the second term, 
\begin{equation*}
\begin{split}
(\hat{\Sigma}_{C})^{-1} - (\Sigma_{C})^{-1} &= \left ( \Sigma_{C} + \hat{\Sigma}_{C} - \Sigma_{C} \right )^{-1} - (\Sigma_{C})^{-1} \\
&= \left[ I - \left(I - (\Sigma_{C})^{-1} \hat{\Sigma}_{C} \right)\right]^{-1} (\Sigma_{C})^{-1} - (\Sigma_{C})^{-1}.
\end{split}
\end{equation*}
Since $\|I - (\Sigma_{C})^{-1}\hat{\Sigma}_{C}\|_2 = O_{p}(n^{-1/2})$, using Neumann series $(I - H)^{-1} = I + H + H^2 + \dots$ for $H = I - (\Sigma_{C})^{-1}\hat{\Sigma}_{C}$ with $\|H\|_2 \rightarrow_{p} 0 < 1$, we have
\begin{equation*}
\begin{split}
(\hat{\Sigma}_{C})^{-1} - (\Sigma_{C})^{-1} &= \left[I + H + O_p(n^{-1})\right] (\Sigma_C)^{-1} - (\Sigma_C)^{-1} \\
&= H (\Sigma_C)^{-1} + O_p(n^{-1}) \\
&= \left [ I - (\Sigma_{C})^{-1} \hat{\Sigma}_{C}\right] (\Sigma_{C})^{-1} + O_{p}(n^{-1}).
\end{split}
\end{equation*}
Combining the two terms, we obtain
\begin{equation*}
\begin{split}
\hat{\lambda}_{C, j} - \lambda_{C, j} &= (\Sigma_{C})^{-1} \left(\hat{\Sigma}_{C,j} - \Sigma_{C,j} \right) +  \left [I - (\Sigma_{C})^{-1} \hat{\Sigma}_{C}\right] (\Sigma_{C})^{-1} \Sigma_{C,j} + O_{p}(n^{-1}) \\
&= (\Sigma_C)^{-1} \hat{\Sigma}_{C,j} - (\Sigma_C)^{-1} \Sigma_{C,j} + (\Sigma_C)^{-1} \Sigma_{C,j} - (\Sigma_C)^{-1} \hat{\Sigma}_C (\Sigma_C)^{-1} \Sigma_{C,j} + O_p(n^{-1}) \\
&\stackrel{\text{(i)}}{=} (\Sigma_{C})^{-1} \left(\hat{\Sigma}_{C,j} - \hat{\Sigma}_{C} \lambda_{C,j} \right) + O_{p}(n^{-1}) \\
&= \frac{1}{n} \sum_{i=1}^{n} (\Sigma_{C})^{-1} \left [ X_j^{(i)}X_{C}^{(i)} - X_{C}^{(i)} X_{C}^{(i) \T} \lambda_{C,j} \right] + O_{p}(n^{-1}) \\
&= \frac{1}{n} \sum_{i=1}^{n} (\Sigma_{C})^{-1} X_{C}^{(i)} \left(X_j^{(i)} - \lambda_{C,j}^{\T} X_{C}^{(i)} \right)  + O_{p}(n^{-1}) \\
&\stackrel{\text{(ii)}}{=} \frac{1}{n} \sum_{i=1}^{n} (\Sigma_{C})^{-1} X_{C}^{(i)} \left(X_j^{(i)} - \lambda_{\Pa(B_k,\g),j}^{\T} X_{\Pa(B_k,\g)}^{(i)} \right)  + O_{p}(n^{-1}) \\
&= \frac{1}{n} \sum_{i=1}^{n} (\Sigma_{C})^{-1} X_{C}^{(i)} \varepsilon_{j}^{(i)} + O_{p}(n^{-1}),
\end{split}
\end{equation*}
where (i) uses $\lambda_{C,j} = (\Sigma_C)^{-1} \Sigma_{C,j}$ and (ii) follows from \cref{prop:block-recursive} and $\Pa(B_k, \g) \subseteq C \subseteq \Pa(B_k, \bar{\g})$.
\end{proof}

\subsection{Proof of \cref{lem:cov-G-bar-regression}}
\begin{proof}
For each $k=2,\dots,K$, note that for $C = \Pa(B_k, \bar{\mathcal{G}}) = B_{[k-1]}$, $\vect \hat{\Lambda}_k^{\bar{\g}} = (\hat{\lambda}^{(n)}_{C,j})_{j \in B_k}$ by concatenation. By \cref{lem:RAL}, we have the following asymptotic linear expansion 
\begin{equation} \label{eqs:expansion-G-bar}
\hat{\lambda}_{B_{[k-1]},j}^{(n)} - \lambda_{B_{[k-1]}, j} = \frac{1}{n} \sum_{i=1}^{n} \left(\Sigma_{B_{[k-1]}} \right)^{-1} X_{B_{[k-1]}}^{(i)} \varepsilon_{j}^{(i)} + O_{p}(n^{-1}).
\end{equation}
By the central limit theorem, 
\begin{equation*}
\sqrt{n} \begin{pmatrix} \vect (\hat{\Lambda}_{2}^{\bar{\g}} - \Lambda_{2}) \\
\vdots \\
\vect (\hat{\Lambda}_{K}^{\bar{\g}} - \Lambda_{K})
\end{pmatrix}
\end{equation*}
converges to a centered multivariate normal distribution. Further, we claim that the asymptotic covariance must be block-diagonal according to $k=2,\dots,K$. To see this, take $k < k'$, $j \in B_k$, $j' \in B_{k'}$ and let $C = B_{[k-1]}$, $C' = B_{[k'-1]}$. Using \cref{eqs:expansion-G-bar}, we have
\begin{equation*}
\begin{split}
& \quad \lim_{n \rightarrow \infty} n^{-1} \E  \left(\hat{\lambda}_{C,j}^{(n)} - \lambda_{C,j}\right) \left(\hat{\lambda}_{C',j}^{(n)} - \lambda_{C',j}\right)^{\T}\\
&= (\Sigma_{C})^{-1} \cov(\varepsilon_{j} X_{C}, \varepsilon_{j'} X_{C'}) (\Sigma_{C'})^{-1} \\
&= (\Sigma_{C})^{-1} \left\{ \E \left[\varepsilon_{j} \varepsilon_{j'} X_{C} X_{C'}^{\T} \right] - \E \left[\varepsilon_j X_{C}\right] \E\left[\varepsilon_{j'} X_{C'}^{\T} \right] \right \} (\Sigma_{C'})^{-1}.
\end{split}
\end{equation*}
In the expression above, because $\varepsilon_{B_k} \indep X_{B_{[k-1]}}$ and $\varepsilon_{B_{k'}} \indep X_{B_{[k'-1]}}$ by \cref{cor:block-recursive} and $j \in B_k$, $j' \in B_{k'}$ for $k < k'$, we have $ \E \left[\varepsilon_{j} \varepsilon_{j'} X_{C} X_{C'}^{\T} \right] = \E \varepsilon_{j'} \E\left[\varepsilon_{j}  X_{C} X_{C'}^{\T} \right] = \bm{0}$, $\E\varepsilon_j X_{C} = \bm{0}$ and $ \E\varepsilon_{j'} X_{C'} = \bm{0}$. It follows that the display above evaluates to $\bm{0}$ and hence the asymptotic covariance matrix is block-diagonal. 

It remains to be shown that $\acov \vect (\hat{\Lambda}_{k}^{\bar{\g}} - \Lambda_{k}) = \Omega_k \otimes (\Sigma_{B_{[k-1]}})^{-1}$ for $k=2,\dots,K$. Fix $k$, take any two distinct $j, j' \in B_{k}$ and let $C=B_{[k-1]}$. Again using \cref{eqs:expansion-G-bar}, we have
\begin{equation*}
\acov \begin{pmatrix} \hat{\lambda}^{(n)}_{C,j}\\ 
\hat{\lambda}^{(n)}_{C,j'} \end{pmatrix} = \begin{pmatrix} H & F \\
F^{\T} & D
\end{pmatrix},
\end{equation*}
where
\begin{equation*}
\begin{split}
H &= (\Sigma_{C})^{-1} \cov(\varepsilon_j X_C, \varepsilon_j X_C) (\Sigma_{C})^{-1} = \var(\varepsilon_j) (\Sigma_{B_{[k-1]}})^{-1}, \\
F &= (\Sigma_{C})^{-1} \cov(\varepsilon_j X_C, \varepsilon_{j'} X_C) (\Sigma_{C})^{-1} = \cov(\varepsilon_j, \varepsilon_{j'}) (\Sigma_{B_{[k-1]}})^{-1}, \\
D &= (\Sigma_{C})^{-1} \cov(\varepsilon_{j'} X_C, \varepsilon_{j'} X_C) (\Sigma_{C})^{-1} = \var(\varepsilon_j') (\Sigma_{B_{[k-1]}})^{-1}.
\end{split}
\end{equation*}
Noting that $\Omega_k = \cov(\varepsilon_{B_k})$ and $\vect \hat{\Lambda}_k^{\bar{\g}} = (\hat{\lambda}^{(n)}_{C,j})_{j \in B_k}$, the result then follows from comparing the expressions above to the definition of Kronecker product for every pair $j, j' \in B_{k}$.
\end{proof}

\subsection{Proof of \cref{lem:cov-G-regression}}
\begin{proof}
Note that by the restrictive property of $\mathcal{G}$ (\cref{cor:restrictive}), we have $\vect \hat{\Lambda}_{k}^{\mathcal{G}} = \left(\hat{\lambda}^{(n)}_{\Pa(B_k, \g),j} \right)_{j \in B_k}$ for $k=2,\dots,K$. Using \cref{lem:RAL} with $C=\Pa(B_k, \g)$, we have the following asymptotic linear expansion 
\begin{equation} \label{eqs:expansion-G-reg}
\hat{\lambda}_{\Pa(B_k, \g),j}^{(n)} - \lambda_{\Pa(B_k, \g), j} = \frac{1}{n} \sum_{i=1}^{n} \left(\Sigma_{\Pa(B_k, \g)} \right)^{-1} X_{\Pa(B_k, \g)}^{(i)} \varepsilon_{j}^{(i)} + O_{p}(n^{-1}).
\end{equation}
The rest of computation follows similarly to the proof of \cref{lem:cov-G-bar-regression}.
\end{proof}

\subsection{Proof of \cref{lem:quadratic-form}}
\begin{proof}
Since $S \in \PD{n}$, by completing the square, we have
\begin{equation*}
\begin{split}
x^{\T} S x &= x_{A}^{\T} S_{A,A} x_A + x_B^{\T} S_{B,A} x_A + x_A^{\T} S_{A,B} x_B + x_B^{\T} S_{B,B} x_B \\
& \qquad - x_A^{\T} S_{A,B} S_{B,B}^{-1} S_{B,A} x_A  + x_A^{\T} S_{A,B} S_{B,B}^{-1} S_{B,A} x_A \\
&=  x_{A}^{\T} (S_{A,A} - S_{A,B} S_{B,B}^{-1} S_{B,A}) x_{A} + (x_B + S_{B,B}^{-1} S_{B,A} x_{A})^{\T} S_{B,B} (x_B + S_{B,B}^{-1} S_{B,A} x_{A}) \\
& \geq x_{A}^{\T} (S_{A,A} - S_{A,B} S_{B,B}^{-1} S_{B,A}) x_{A} = x_{A}^{\T} S_{A \cdot B} x_{A},
\end{split}
\end{equation*}
where the equality holds if and only if  $x_{B} = -S_{B,B}^{-1} S_{B,A} x_{A}$.
\end{proof}
\section{Proofs for graphical results} \label{apx:sec:graph-proofs}
\subsection{Proof of \cref{lemma:anc-bucket21}}
\begin{proof} 
Let the undirected path between $j$ and $k$ be $p = \langle j = V_1, \dots, V_l = k \rangle$ with $l > 1$. First note that $i$ is not on $p$ because there is no undirected path between $i$ and $j$ in $\mathcal{G}$. 

Further, since $i \rightarrow j - V_2$ is in $\mathcal{G}$, by Meek rules R1 and R2 (\cref{fig:orientationRules} in \cref{apx:sec:graph}), $i - V_2$ or $i \rightarrow V_2$ is in $\mathcal{G}$. Since, by assumption, there is no undirected path from $i$ to $j$ in $\mathcal{G}$, $i - V_2 \notin U$. Hence, $i \rightarrow V_2 \in E$  and if $l=2$, the statement of the lemma holds. If $l>2$, we can apply the above reasoning iteratively until we obtain $i \rightarrow V_l  \in E$.
\end{proof}

\subsection{Proof of \cref{lem:pa-D}}
\begin{proof} 
Let $l \in D_k$. Since $D_k \subseteq B_k$, $l \in B_k$. Then by \cref{cor:restrictive}, we have that $\Pa(B_k) = \Pa(l) \setminus B_k$. Therefore, $\Pa(B_k) \subseteq \cup_{j \in D_k} \Pa(j) \setminus B_k$ and furthermore, $\Pa(B_k) \subseteq \cup_{j \in D_k} \Pa(j) \setminus D_k = \Pa(D_k)$.
Hence, it suffices to show $\Pa(D_k) \subseteq \Pa(B_k)$.

We prove $\Pa(D_k) \subseteq \Pa(B_k)$ by contradiction. Suppose there exists $j \in \Pa(D_k) \setminus \Pa(B_k)$. We know that $\Pa(D_k) =  \cup_{l \in D_k} \Pa(l) \setminus D_k$ and since $D_k \subseteq B_k$, for any $l \in D_k$, $l \in B_k$. Using \cref{cor:restrictive}, we also know that $\Pa(B_k) =\Pa(l) \setminus B_k$ for every $l \in D_k$. Therefore, if $j \in \Pa(D_k) \setminus \Pa(B_k)$, it follows that $j \in B_k$.

Additionally, by definition $D = \An(Y, \g_{V \setminus A})$ and 
 $D = \cup_{r=1}^K D_r$. Therefore, if $k=1$, then $j \in A$; if $k >1$, $j$ must be contained in $\cup_{r=1}^{k-1} D_{r}$ or in $A$. If $j \in A$, then because condition of \cref{thm:id-criterion} holds, $j \in A \cap B_{k}$  leads to a contradiction with \cref{lemma:big2} in \cref{apx:sec:graph}. 
Suppose $k >1$ and $j \in  \cup_{r=1}^{k-1} D_{r}$. Because $\cup_{r=1}^{k-1} D_{r} \subseteq \cup_{r=1}^{k-1} B_r $ and buckets $\{B_1, \dots, B_K\}$ are disjoint, we have $(\cup_{r=1}^{k-1} D_{r} ) \cap B_k = \emptyset$. However, this contradicts that $j \in B_k$.
\end{proof}

\subsection{Proof of \cref{prop:satmpdag}}
\begin{proof}
By construction, the undirected component of $\bar{\mathcal{G}}$ remains the same as that of $\mathcal{G}$. Hence, $\bar{\mathcal{G}}$ has the same bucket decomposition as $\mathcal{G}$. We only need to show that $\bar{\g}$ is an MPDAG. It is enough to show that the edge orientations in $\bar{\g}$ are closed under rules R1--R4 of \citet{meek1995causal} that are displayed in  \cref{fig:orientationRules} of \cref{apx:sec:graph}.  
Note that since $\g$ is an MPDAG it is closed under R1--R4.   So if any of the left-hand-side graphs in Figure \ref{fig:orientationRules} are induced subgraphs of $\bar{\g}$, then at least one of the directed edges in these induced subgraphs must have been added in the construction of $\bar{\g}$.

Since the construction of $\bar{\g}$ does not involve adding directed edges within a bucket, the left-hand-side of rules R3 and R4 in  Figure \ref{fig:orientationRules} cannot appear as induced subgraphs of $\bar{\g}$. Hence, edge orientations in  $\bar{\g}$ are complete under rules R3 and R4.

Consider the left-hand-side of rule R1 in Figure \ref{fig:orientationRules}, $A \rightarrow B - C$, for some $A,B,C \in V$. For $A \rightarrow B-C$ to be an induced subgraph of $\bar{\g}$, $A \rightarrow B$ must have been added in the construction of $\bar{\g}$ from $\g$. Hence, $A $ and $B$ would need to be in different buckets in $V$ in $\g$. Since $B$ and $C$ are in the same bucket because of edge $B-C$, $A \rightarrow C$ would also be added to $\g$ in the construction of $\bar{\g}$. Hence, $A\rightarrow B - C$  will also not appear as an induced subgraph of $\bar{\g}$ and edge orientations in  $\bar{\g}$ are also closed under R1.

Consider the left-hand-side of R2 in Figure \ref{fig:orientationRules}, and suppose for a contradiction that $A \rightarrow B \rightarrow C $ and $A - C$ is an induced subgraph of $\bar{\g}$ for some $A, B, C \in V$.  Then $A \rightarrow B$, $B \rightarrow C$, or both $A \rightarrow B$ and $B \rightarrow C$, were added to $\g$ in the construction of $\bar{\g}$.
Because of $A- C$, suppose $A$ and $C$ are in the same bucket $B_i$ for some $i \in \{1, \dots , K\}$ in $\g$. Also, suppose $B \in B_j$. Because only directed edges between buckets are added, $i \neq j$. 

Now, $A \rightarrow B$ and $B \rightarrow C$ cannot be both added to $\g$ to construct $\bar{\g}$, because that would imply that $i < j$ and $j < i$. By R1, $B \rightarrow C - A$ cannot be an induced subgraph of MPDAG $\g$, so $A \rightarrow B$ alone also could not be added to $\g$. Therefore, $B \rightarrow C$ alone was added to $\g$.
But $C - A \rightarrow B$ is an induced subgraph of $\g$, so $ i < j$, which contradicts the direction of $B \rightarrow C$. 
\end{proof}
\section{Graphical preliminaries} \label{apx:sec:graph}

\paragraph*{Graphs, vertices, edges} A graph $\g= (V,F) $ consists of a set of vertices (variables) $ V$ and a set of edges $ F $. The graphs we consider are allowed to contain directed ($\rightarrow$) and  undirected ($-$) edges and at most one edge between any two vertices. 
We can thus partition the set of edges $F$ into a set of directed edges $E$ and undirected edges $U$ and denote graph $\g= (V,F) $ as $\g= (V,E,U)$. The corresponding undirected graph is simply $\g_U = (V, \emptyset, U)$.

\paragraph*{Subgraphs and skeleton} An \textit{induced subgraph} $\g_{ {V'}} =(V',F')$ of $\g= (V,F) $ consists of  $V' \subseteq V$ and  $F' \subseteq F$ where $F'$ are all edges in $F$ between vertices in $V'$.  A \emph{skeleton} of a graph $\g= (V,F) $ is an undirected graph $\g = (V, F')$, such that $F'$ are undirected versions of all edges in $F$.

\paragraph*{Paths. Directed, undirected, causal, non-causal, proper paths} A \textit{path} $p$ from $i$ to $j$ in $\g$ is a sequence of distinct vertices $p= \langle i, \dots,j \rangle$ in which every pair of successive vertices are adjacent.  A path consisting of undirected edges is an \textit{undirected} path.
A  \textit{directed path} from $i$ to $j$ is a path from $i$ to $j$ in which all edges are directed towards $j$, that is, $i \to \dots\to j$. We will use \textit{causal path} instead of \textit{directed  path} when talking about causal graphs. Let $p = \langle v_1, \dots , v_k \rangle$, $k > 1$ be a path in $\g$, $p$ is a \textit{possibly directed path} (\textit{possibly causal path}) if no edge $v_i \leftarrow v_j, 1 \le i < j \le k$ is in $\g$. Otherwise, $p$ is a \textit{non-causal path} in $\g$ (see Definition 3.1 and Lemma 3.2 of \citealp{perkovic17}).  A path from ${A}$ to ${Y}$ is \textit{proper} (w.r.t. ${A}$) if only its first vertex is in ${A}$.

\paragraph*{Directed cycles} 
A directed  path from $i$ to $j$ and the  edge $j\to i$ form a \textit{directed cycle}.

\paragraph*{Colliders, shields and definite status paths} 
If a path $p$ contains $i \rightarrow j \leftarrow k$ as a subpath, then $j$ is a \textit{collider} on $p$. A path $\langle i,j,k \rangle$ is an \emph{(un)shielded triple} if $ i $ and $ k$ are (not) adjacent. A path is \textit{unshielded} if all successive triples on the path are unshielded. 
A node $v_{j}$ is a \textit{definite non-collider} on a path $p$ if there is at least one edge out of $v_{j}$ on $p$, or if $v_{j-1} - v_j - v_{j+1}$ is a subpath of $p$ and $\langle v_{j-1} , v_j , v_{j+1}\rangle$ is an unshielded triple. A node is of \textit{definite status} on a path if it is a collider, a definite non-collider or an endpoint on the path. A path $p$ is of definite status if every node on $p$ is of definite status.

\paragraph*{Subsequences and subpaths} 
A \textit{subsequence} of a path $p$ is obtained by deleting some nodes from $p$ without changing the order of the remaining nodes. A subsequence of a path is not necessarily a path. For a path $p = \langle v_1,v_2,\dots,v_m \rangle$, the \textit{subpath} from $v_i$ to $v_k$ ($1\le i\le k\le m)$ is the path $p(v_i,v_k) = \langle v_i,v_{i+1},\dots,a_{k}\rangle$.

\paragraph*{Ancestral relations} 
If $i\to j$, then $i$ is a \textit{parent} of $j$, and $j$ is a \textit{child} of $i$. If there is a causal path from $k$ to $l$, then $k$ is an \textit{ancestor} of $l$, and $l$ is a \textit{descendant} of $k$.  If there is a possibly causal path from $k$ to $l$, then $k$ is a \textit{possible ancestor} of $l$, and $l$ is a \textit{possible descendant} of $k$. We use the convention that every vertex is a descendant, ancestor, possible ancestor and possible descendant of itself.
The sets of parents,  ancestors,  descendants and possible descendants of $i$ in~$\g$ are denoted by $\Pa(i,\g)$, $\An(i,\g)$, $\De(i,\g)$ and $\PossDe(i,\g)$ respectively. For a set of vertices ${A}$, we let $\Pa({A},\g) =( \cup_{i \in {A}}  \Pa(i,\g)) \setminus {A}$, whereas, $\An({A},\g) = \cup_{i \in {A}} \An(i,\g)$,  $\De({A},\g) = \cup_{i \in {A}} \De(i,\g)$ and $\PossDe({A},\g) = \cup_{i \in {A}} \PossDe(i,\g)$

\paragraph*{DAGs, PDAGs}
A \textit{directed graph} contains only directed edges.
A \textit{partially directed graph} may contain both directed and undirected edges.
A directed graph without directed cycles is a \textit{directed acyclic graph $(\DAG)$}.  A \textit{partially directed acyclic graph $(\PDAG)$} is a partially directed graph without directed cycles.

\paragraph*{Blocking and d-separation} (See Definition 1.2.3 of \cite{pearl2009causality} and Lemma C.1  of \cite{henckel2019graphical}). Let $Z$ be a set of vertices in an PDAG $\g = (V,E,U)$. A definite status path $p$ is  \textit{blocked} by ${Z}$ if (i) $p$ contains a non-collider that is in ${Z}$, or (ii) $p$ contains a collider $C$ such that no descendant of $C$ is in ${Z}$.
A definite status path that is not blocked by ${Z}$ is \textit{open} given ${Z}$. If ${A},{B}$ and ${Z}$ are three pairwise disjoint sets of nodes in a PDAG $\g= (V,E,U)$, then ${Z}$  \textit{d-separates} ${A}$ from ${B}$ in $\g$ if ${Z}$ blocks every definite status path between any node in ${A}$ and any node in ${B}$ in $\g$.

\paragraph*{CPDAGs, MPDAGs}
 Several DAGs can encode the same d-separation relationships. Such DAGs form a \emph{Markov equivalence class} which is uniquely represented  by a \emph{completed partially directed acyclic graph} (CPDAG) \citep{meek1995causal,andersson1997characterization}.  
A PDAG $\g =(V,E,U)$ is a \textit{maximally oriented} PDAG (MPDAG) if it is closed under orientation rules R1-R4 of \citep{meek1995causal}, presented in Figure \ref{fig:orientationRules}. The MPDAG can then be alternatively defined as any PDAG that does not contain graphs on the left-hand side of each orientation rule as induced subgraphs. Both  DAGs  and CPDAGs are types of MPDAGs \citep{meek1995causal}.

 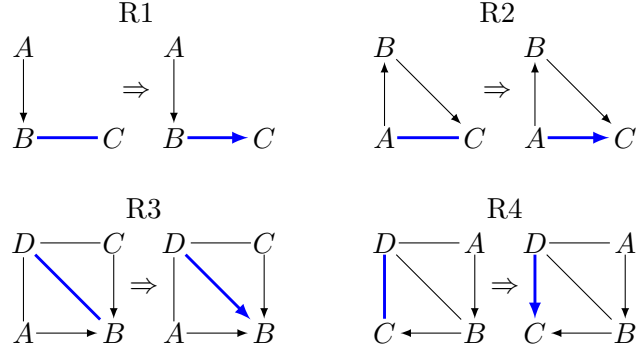
\begin{figure}[!tb]
\vspace{-.2cm}
\centering
\begin{tikzpicture}[->,>=latex,shorten >=1pt,auto,node distance=1.2cm,scale=1,transform shape]
  \tikzstyle{state}=[inner sep=0.5pt, minimum size=5pt]

\node[state] (Xia) at (0,0) {$B$};
  \node[state] (Xka) at (0,1.2) {$A$};
  \node[state] (Xja) at (1.2,0) {$C$};

  \path (Xka) edge (Xia);
 \draw[-,line width=1.1pt,blue]
          (Xia) edge (Xja);

  \coordinate [label=above:R1] (L1) at (1.5,1.4);
  \coordinate [label=above:$\Rightarrow$] (L2) at (1.5,0.4);

  \node[state] (Xib) at (2,0) {$B$};
  \node[state] (Xkb) at (2,1.2) {$A$};
  \node[state] (Xjb) at (3.2,0) {$C$};

  \path (Xkb) edge (Xib);
   \draw[->,line width=1.1pt,blue]   (Xib) edge (Xjb);

\node[state] (Xic) at (4.8,0) {$A$};
  \node[state] (Xkc) at (4.8,1.2) {$B$};
  \node[state] (Xjc) at (6,0) {$C$};

  \path (Xic) edge (Xkc)
          (Xkc) edge (Xjc);
   \draw[-,line width=1.1pt,blue]
          (Xic) edge (Xjc);

  \coordinate [label=above:R2] (L2) at (6.3,1.4);
  \coordinate [label=above:$\Rightarrow$] (L2) at (6.3,0.4);

  \node[state] (Xid)  at (6.8,0) {$A$};
  \node[state] (Xkd)  at (6.8,1.2) {$B$};
  \node[state] (Xjd)  at (8,0) {$C$};

  \path (Xid) edge (Xkd)
          (Xkd) edge (Xjd);
   \draw[->,line width=1.1pt,blue]  (Xid) edge (Xjd);

\node[state] (Xie) at (0,-1.4) {$D$};
  \node[state] (Xke) at (1.2,-1.4) {$C$};
  \node[state] (Xle) at (0,-2.6) {$A$};
  \node[state] (Xje) at (1.2,-2.6) {$B$};

  \path (Xke) edge (Xje)
          (Xle) edge (Xje);

  \draw[-,line width=1.1pt,blue]  (Xie) edge (Xje);
  \path[-]
          (Xke) edge (Xie)
          (Xle) edge (Xie);

  \coordinate [label=above:R3] (L3) at (1.6,-1.2);
    \coordinate [label=above:$\Rightarrow$] (L3) at (1.6,-2.2);

  \node[state] (Xif) at (2,-1.4) {$D$};
  \node[state] (Xkf) at (3.2,-1.4) {$C$};
  \node[state] (Xlf) at (2,-2.6) {$A$};
  \node[state] (Xjf) at (3.2,-2.6) {$B$};

  \draw[->,line width=1.1pt,blue]  (Xif) edge (Xjf);

  \path (Xkf) edge (Xjf)
          (Xlf) edge (Xjf);
  \path[-]
          (Xkf) edge (Xif)
          (Xlf) edge (Xif);

\node[state] (Xig) at (4.8,-1.4) {$D$};
  \node[state] (Xjg) at (6,-1.4) {$A$};
  \node[state] (Xkg) at (4.8,-2.6) {$C$};
  \node[state] (Xlg) at (6,-2.6) {$B$};

  \path (Xlg) edge (Xkg)
          (Xjg) edge (Xlg);
  \draw[line width=1.1pt,blue,-]
           (Xig) edge (Xkg);
  \path[-]
          (Xig) edge (Xjg)
          (Xig) edge (Xlg);

  \coordinate [label=above:R4] (L4) at (6.4,-1.2);
  \coordinate [label=above:$\Rightarrow$] (L4) at (6.4,-2.2);

  \node[state] (Xih) at (6.8,-1.4) {$D$};
  \node[state] (Xjh) at (8,-1.4) {$A$};
  \node[state] (Xkh) at (6.8,-2.6) {$C$};
  \node[state] (Xlh) at (8,-2.6) {$B$};

  \draw[blue,line width=1.1pt,->]  (Xih) edge (Xkh);
  \path (Xlh) edge (Xkh)
        (Xjh) edge (Xlh);
  \path[-]
          (Xih) edge (Xjh)
           (Xih) edge (Xlh);

\end{tikzpicture}
\caption{The orientation rules from~\cite{meek1995causal}. If the graph on the left-hand side of a rule is an induced subgraph of a PDAG $\g$, then \textit{orient} the blue undirected edge ({\color{blue} {\bfseries $-$}}) as shown on the right-hand side of the rule. Hence,  the graphs on the left-hand side of each rule are not allowed to be induced subgraphs of an MPDAG. }
\label{fig:orientationRules}
\end{figure}

\paragraph*{Background knowledge and constructing MPDAGs}
A PDAG $\g'$ is \textit{represented} by another PDAG $\g$ (equivalently $\g$ represents $\g'$) if $\g'$ and $\g$ have the same adjacencies and unshielded colliders and every directed edge $i \rightarrow j$ in $\g$ is also in $\g'$.
Let ${R}$ be a set of directed edges representing background knowledge. \cref{algo:MPDAG} of \cite{meek1995causal} describes how to incorporate background knowledge ${R}$ in an MPDAG $\g$. If \cref{algo:MPDAG} does not return a FAIL, then it returns a new MPDAG $\g'$ that is represented by $\g$. Background knowledge ${R}$ is \textit{consistent} with MPDAG $\g$ if and only if \cref{algo:MPDAG} does not return a FAIL \citep{meek1995causal}.
\unskip
\begin{algorithm}[!htb]
\KwData{MPDAG $\g$, background knowledge ${R}$}
 \KwResult{MPDAG $\g'$ or FAIL}
 Let $\g' = \g$\;
 \While{${R}\neq \emptyset$}{
  Choose an edge $\{X \rightarrow Y\}$ in ${R}$ \;
  ${R} = {R} \setminus \{X \rightarrow Y\}$ \;
  \eIf{$\{X - Y\}$ or $\{X\rightarrow Y\}$ is in $\g'$}{
   Orient $\{X \rightarrow Y\}$ in $\g'$\;
   Close the edge orientations under the rules in Figure~\ref{fig:orientationRules} in $\g'$\;
   }{
   FAIL\;
  }
 }
\caption{\texttt{ConstructMPDAG}, \citep{meek1995causal,perkovic17}}
\label{algo:MPDAG}
\end{algorithm}
\begin{remark}
 The MPDAG output by \texttt{ConstructMPDAG}($\g$, ${R}$) is the same independent of the ordering of edges in ${R}$. This stems from the fact that  the orientation rules of  \cite{meek1995causal}  are necessary and sufficient for the construction of an MPDAG given a set of adjacencies and unshielded colliders.
\end{remark}

\paragraph*{$\g$ and $[\g]$} 
If $\g$ is a MPDAG, then $[\g]$ denotes every DAG represented by $\g$.

\paragraph*{Causal and partial causal ordering of vertices}
  A total ordering, $<$, of vertices $ {V'} \subseteq  {V}$ is \textit{consistent} with a DAG $\g[D] = ( {V,E, \emptyset})$ and called a \textit{causal ordering} of $ {V'}$ if for every $i, j \in  {V'}$, such that $i < j$ and such that $i$ and $j$ are adjacent in $\g[D]$, $i \rightarrow j$ is in $\g[D]$. There can be more than one causal ordering of $ {V'}$ in a DAG $\g[D] = ( {V,E, \emptyset})$. For example, in DAG $i \leftarrow j \rightarrow k$ both orderings $j < i < k$ and $j <k < i$ are consistent.

 Since an MPDAG may contain undirected edges, there is generally no unique causal ordering of vertices in an MPDAG. Instead, we define a \emph{partial causal ordering,} $<$, of a vertex set $ V'$, $V' \subset V$ in an MPDAG $\g = ( {V,E,U})$ as a total ordering of pairwise disjoint vertex sets $ {A_1}, \dots ,  {A_k},$ $k \ge 1$, $\cup_{i=1}^{k}  {A_i} =  {V'}$, that satisfy the following: if $ {A_i} <  {A_j}$ and there is an edge between $i \in  {A_i}$ and $j \in  {A_j}$ in $\g$, then $i \rightarrow j$ is in $\g$.

\paragraph*{Buckets and bucket decomposition} 
\cref{alg:pto} describes how to obtain an ordered bucket decomposition for a set of vertices $ {V}$ in an MPDAG $\g = (V,E,U)$. By \citet[Lemma 1]{perkovic20}, the ordered list of buckets output by Algorithm \ref{alg:pto} is a partial causal ordering of $V$ in $\g$. 
 
\begin{algorithm}[!htb]
 \Input{vertex set $ {V}$ in  MPDAG $\mathcal{G}{=}({V,E,U})$ and MPDAG $\g$.}
 \Output{An ordered list $ {B} {=} ( {B_1}, \dots ,  {B_k}), k \ge 1,$ of the bucket decomposition of $ {V}$ in $\g$.}
   \SetAlgoLined
  Let $\g_{U}$ denote the undirected subgraph of $\g$;\\
   Let $ {ConComp}$ be the bucket decomposition (i.e., maximal connected components) of $ {V}$ in $\g_{U}$;\\
   Let $ {B}$ be an empty list;\\
\While{$ {ConComp}\neq \emptyset$}{
Let  $ {C} $ be any element from ${ConComp}$;\\
Let $\overline{ {C}}$ be the set of vertices in $ {ConComp}$ that are not in $ {C}$;\\
\If{all edges between  $ {C}$ and  $\overline{ {C}}$   are into $ {C}$ in $\g$}
{Remove $ {C}$ from $ {ConComp}$;\\
 { Add $C$ to the beginning of $ {B}$;}
}
}
\Return  $ {B}$;\\
\caption{Partial causal ordering \citep{perkovic20}}
\label{alg:pto}
\end{algorithm}

\begin{lemma}\label{lemma:big2} (see Lemma D.1 (i)  of \citealp{perkovic20})
 Let $A$  and $Y$ be disjoint node sets in MPDAG $\g =(V,E,U)$. Suppose that there is no proper possibly causal path from $A$ to $Y$ that starts with an undirected edge in $\g$, that is, suppose that the criterion in Theorem \ref{thm:id-criterion} is satisfied. Further, let $D = \An(Y, \g_{V \setminus A})$ and $D= \dot{\bigcup}_{i=1}^K D_i$ for $D_i = D \cap B_i$, $i=1,\dots,K$, where $B_1, \dots B_K$ is the bucket decomposition of $V$.
Then for all $i \in \{1, \dots ,K\}$,  there is no proper possibly causal path from $A$ to $B_i$ that starts with an undirected edge in $\g$.
\end{lemma}
\section{Additional simulation results} \label{apx:sec:simu}
In this section, we report additional simulation results. 
\subsection{With estimated graphs}

The setup is the same as \cref{sec:simu} of main text, but we replace the true CPDAG with the CPDAG estimated with the greedy equivalence search algorithm \citep{chickering2002optimal} based on the same sample. The relative squared errors of the contending estimators are shown in \cref{fig:simu-graph-GES} and are summarized in \cref{tab:simu-graph-GES}. Compared to the results with the true CPDAG, the performance improvement of $\mathcal{G}$-regression is more modest but still matters in practice. The reduced improvement is due to the error in estimating the graph, which diminishes as $n$ increases. 

\begin{table}[!t]
\begin{center}
\caption{Geometric average (brackets: median) of relative squared errors compared to $\g$-regression when CPDAGs are estimated} \label{tab:simu-graph-GES}
\scriptsize
\begin{tabular}{lrlrlcrlrlcrlrl}
\toprule[1.2pt]
&\multicolumn{4}{c}{$|V|=20$}&&\multicolumn{4}{c}{$|V|=50$}&&\multicolumn{4}{c}{$|V|=100$}\\ 
\cline{2-5}\cline{7-10}\cline{12-15}
$|A|$ & \multicolumn{2}{c}{$n=100$} & \multicolumn{2}{c}{$n=1000$} && \multicolumn{2}{c}{$n=100$} & \multicolumn{2}{c}{$n=1000$} && \multicolumn{2}{c}{$n=100$} & \multicolumn{2}{c}{$n=1000$} \\ 
\midrule
\texttt{adj.O}&&&&&&&&\\ 
1&1.0 &(1.0)&1.0 &(1.0)&&1.2 &(1.0)&1.3 &(1.0)&&1.8 &(1.1)&1.6 &(1.0)\\ 
2&2.0 &(1.1)&3.1 &(1.2)&&2.4 &(1.3)&3.1 &(1.4)&&3.2 &(1.9)&3.7 &(2.0)\\ 
3&3.3 &(1.7)&5.2 &(2.7)&&4.0 &(2.4)&5.9 &(2.8)&&4.7 &(2.5)&5.5 &(2.8)\\ 
4&4.6 &(2.2)&7.9 &(4.2)&&5.0 &(2.1)&9.0 &(5.7)&&10 &(5.9)&8.9 &(5.6)\\ 
\texttt{IDA.M}&&&&&&&&\\ 
1&2.9 &(1.4)&4.1 &(1.4)&&4.5 &(2.7)&10 &(5.7)&&7.3 &(4.5)&18 &(11)\\ 
2&4.2 &(2.0)&6.6 &(2.1)&&7.3 &(4.8)&14 &(7.2)&&13 &(7.9)&22 &(14)\\ 
3&6.2 &(3.1)&6.8 &(2.5)&&12 &(7.1)&16 &(8.3)&&15 &(10)&28 &(18)\\ 
4&9.5 &(5.6)&9.0 &(3.1)&&13 &(10)&20 &(12)&&19 &(14)&37 &(26)\\ 
\texttt{IDA.R}&&&&&&&&\\ 
1&2.9 &(1.4)&4.1 &(1.4)&&4.5 &(2.7)&10 &(5.7)&&7.3 &(4.5)&18 &(11)\\ 
2&2.7 &(1.3)&4.6 &(1.2)&&4.5 &(2.3)&9.6 &(4.0)&&8.5 &(5.9)&15 &(9.5)\\ 
3&3.1 &(1.5)&4.1 &(1.2)&&5.8 &(3.0)&7.8 &(2.5)&&7.6 &(5.2)&14 &(8.9)\\ 
4&3.6 &(1.6)&4.2 &(1.3)&&4.9 &(2.8)&8.2 &(3.6)&&8.1 &(5.4)&15 &(10)\\
\bottomrule[1.2pt]
\end{tabular}
\end{center}
\end{table}

\begin{figure}[!htb]
\centering
\includegraphics[width=1\textwidth]{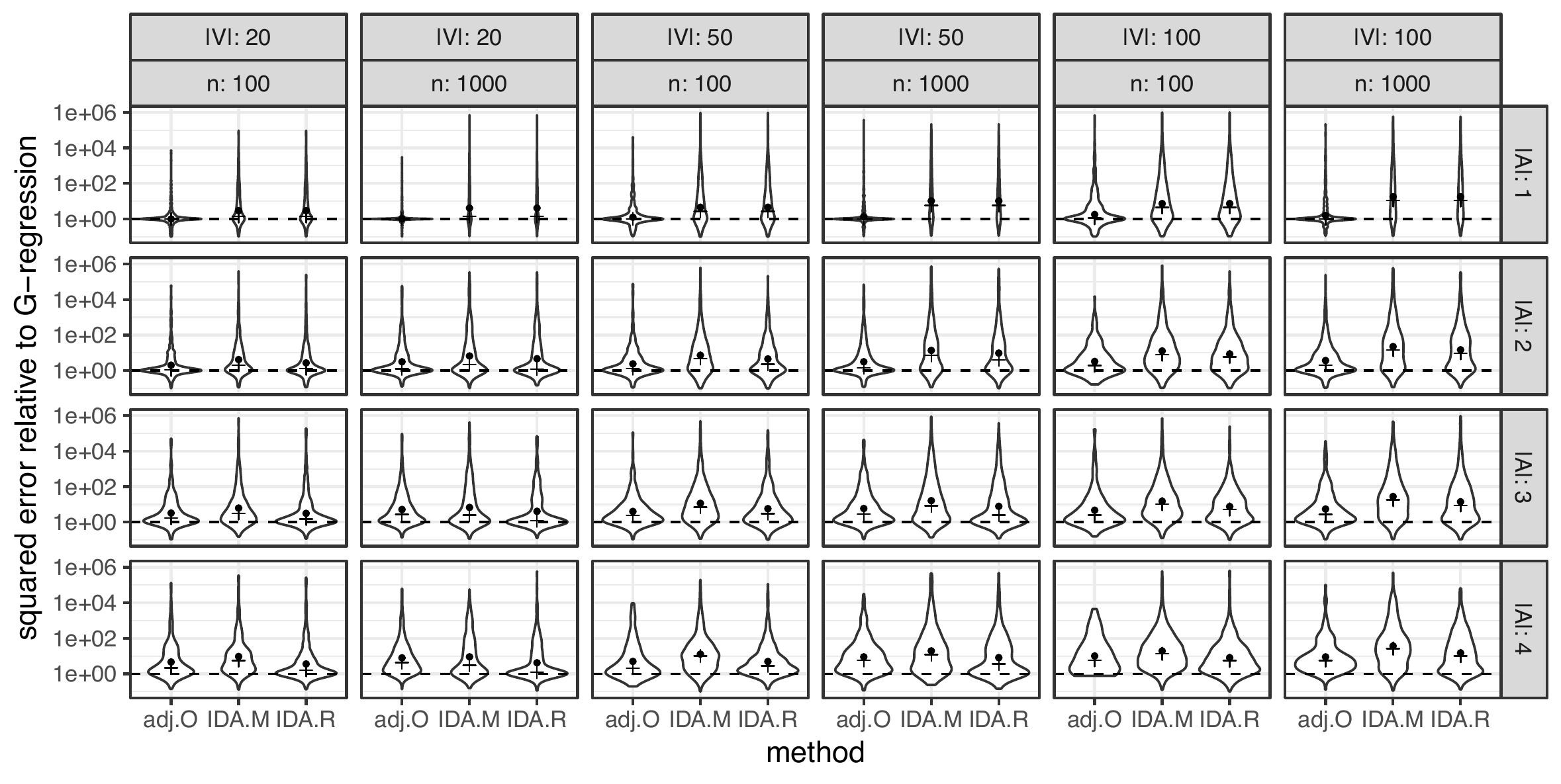}
\caption{Violin plots for the relative squared errors of contending estimators (`$\cdot$': geometric mean, `$+$': median). The estimated CPDAGs are provided to the estimators.}
\label{fig:simu-graph-GES}
\end{figure}

\subsection{With rescaled coefficients}
It can be argued that the data generating mechanism described in \cref{sec:simu} is less realistic because the variance of a vertex grows with its topological ordering. To remedy this, we rescale the coefficients $\{\gamma_{ij}\}$ properly such that the variances are roughly equalized along the topological order. The results (based on true graphs) are reported in \cref{tab:simu-rescaled}. 

\begin{table}[!t]
\begin{center}
\caption{Geometric average (brackets: median) of relative squared errors compared to $\g$-regression (coefficients rescaled)} \label{tab:simu-rescaled}
\scriptsize
\begin{tabular}{lrlrlcrlrlcrlrl}
\toprule[1.2pt]
&\multicolumn{4}{c}{$|V|=20$}&&\multicolumn{4}{c}{$|V|=50$}&&\multicolumn{4}{c}{$|V|=100$}\\ 
\cline{2-5}\cline{7-10}\cline{12-15}
$|A|$ & \multicolumn{2}{c}{$n=100$} & \multicolumn{2}{c}{$n=1000$} && \multicolumn{2}{c}{$n=100$} & \multicolumn{2}{c}{$n=1000$} && \multicolumn{2}{c}{$n=100$} & \multicolumn{2}{c}{$n=1000$} \\ 
\midrule
\texttt{adj.O}&&&&&&&&\\ 
1&1.0 &(1.0)&1.0 &(1.0)&&1.0 &(1.0)&1.0 &(1.0)&&1.0 &(1.0)&1.0 &(1.0)\\ 
2&1.1 &(1.0)&1.1 &(1.0)&&1.1 &(1.0)&1.0 &(1.0)&&1.1 &(1.0)&1.0 &(1.0)\\ 
3&1.2 &(1.0)&1.1 &(1.0)&&1.1 &(1.0)&1.0 &(1.0)&&1.1 &(1.0)&1.0 &(1.0)\\ 
4&1.2 &(1.0)&1.1 &(1.0)&&1.1 &(1.0)&1.1 &(1.0)&&1.1 &(1.0)&1.0 &(1.0)\\  
\texttt{IDA.M}&&&&&&&&\\ 
1&1.0 &(1.0)&1.0 &(1.0)&&1.0 &(1.0)&1.0 &(1.0)&&1.0 &(1.0)&1.0 &(1.0)\\ 
2&1.1 &(1.0)&1.0 &(1.0)&&1.1 &(1.0)&1.0 &(1.0)&&1.2 &(1.0)&1.0 &(1.0)\\ 
3&1.2 &(1.0)&1.1 &(1.0)&&1.2 &(1.0)&1.0 &(1.0)&&1.2 &(1.0)&1.0 &(1.0)\\ 
4&1.2 &(1.0)&1.0 &(1.0)&&1.2 &(1.0)&1.0 &(1.0)&&1.2 &(1.0)&1.0 &(1.0)\\
\texttt{IDA.R}&&&&&&&&\\ 
1&1.0 &(1.0)&1.0 &(1.0)&&1.0 &(1.0)&1.0 &(1.0)&&1.0 &(1.0)&1.0 &(1.0)\\ 
2&1.1 &(1.0)&1.0 &(1.0)&&1.1 &(1.0)&1.0 &(1.0)&&1.2 &(1.0)&1.0 &(1.0)\\ 
3&1.1 &(1.0)&1.0 &(1.0)&&1.1 &(1.0)&1.0 &(1.0)&&1.1 &(1.0)&1.0 &(1.0)\\ 
4&1.2 &(1.0)&1.0 &(1.0)&&1.1 &(1.0)&1.0 &(1.0)&&1.1 &(1.0)&1.0 &(1.0)\\
\bottomrule[1.2pt]
\end{tabular}
\end{center}
\end{table}


\vskip 0.2in
\begin{thebibliography}{69}
\providecommand{\natexlab}[1]{#1}
\providecommand{\url}[1]{\texttt{#1}}
\expandafter\ifx\csname urlstyle\endcsname\relax
  \providecommand{\doi}[1]{doi: #1}\else
  \providecommand{\doi}{doi: \begingroup \urlstyle{rm}\Url}\fi

\bibitem[Amemiya(1985)]{amemiya1985advanced}
Takeshi Amemiya.
\newblock \emph{Advanced Econometrics}.
\newblock Harvard University Press, 1985.

\bibitem[Anderson and Olkin(1985)]{anderson1985maximum}
Theodore~Wilbur Anderson and Ingram Olkin.
\newblock Maximum-likelihood estimation of the parameters of a multivariate
  normal distribution.
\newblock \emph{Linear algebra and its applications}, 70:\penalty0 147--171,
  1985.

\bibitem[Andersson et~al.(1997)Andersson, Madigan, and
  Perlman]{andersson1997characterization}
Steen~A. Andersson, David Madigan, and Michael~D. Perlman.
\newblock A characterization of {M}arkov equivalence classes for acyclic
  digraphs.
\newblock \emph{The Annals of Statistics}, 25:\penalty0 505--541, 1997.

\bibitem[Bickel et~al.(1993)Bickel, Klaassen, Ritov, and
  Wellner]{bickel1993efficient}
Peter~J. Bickel, Chris A.~J. Klaassen, Ya'acov Ritov, and Jon~A. Wellner.
\newblock \emph{Efficient and Adaptive Estimation for Semiparametric Models},
  volume~4.
\newblock Johns Hopkins University Press, Baltimore, 1993.

\bibitem[Bollen(1989)]{bollen1989structural}
Kenneth~A. Bollen.
\newblock \emph{Structural Equations with Latent Variables}.
\newblock Wiley, New York, 1989.

\bibitem[Chen et~al.(2019)Chen, Drton, and Wang]{chen2019causal}
Wenyu Chen, Mathias Drton, and Y.~Samuel Wang.
\newblock On causal discovery with an equal-variance assumption.
\newblock \emph{Biometrika}, 106\penalty0 (4):\penalty0 973--980, 2019.

\bibitem[Chickering(2002)]{chickering2002optimal}
David~Maxwell Chickering.
\newblock Optimal structure identification with greedy search.
\newblock \emph{Journal of Machine Learning Research}, 3\penalty0
  (Nov):\penalty0 507--554, 2002.

\bibitem[Dawid(1981)]{dawid1981some}
A.~Philip Dawid.
\newblock Some matrix-variate distribution theory: notational considerations
  and a {Bayesian} application.
\newblock \emph{Biometrika}, 68\penalty0 (1):\penalty0 265--274, 1981.

\bibitem[Drton(2006)]{drton2006computing}
Mathias Drton.
\newblock Computing all roots of the likelihood equations of seemingly
  unrelated regressions.
\newblock \emph{Journal of Symbolic Computation}, 41\penalty0 (2):\penalty0
  245--254, 2006.

\bibitem[Drton(2018)]{drton2018algebraic}
Mathias Drton.
\newblock Algebraic problems in structural equation modeling.
\newblock In \emph{The 50th Anniversary of Gr{\"o}bner Bases}, pages 35--86.
  Mathematical Society of Japan, 2018.

\bibitem[Drton and Eichler(2006)]{drton2006maximum}
Mathias Drton and Michael Eichler.
\newblock Maximum likelihood estimation in {Gaussian} chain graph models under
  the alternative {Markov} property.
\newblock \emph{Scandinavian Journal of Statistics}, 33\penalty0 (2):\penalty0
  247--257, 2006.

\bibitem[Drton and Maathuis(2017)]{drton2017structure}
Mathias Drton and Marloes~H. Maathuis.
\newblock Structure learning in graphical modeling.
\newblock \emph{Annual Review of Statistics and Its Application}, 4:\penalty0
  365--393, 2017.

\bibitem[Drton and Richardson(2004)]{drton2004multimodality}
Mathias Drton and Thomas~S. Richardson.
\newblock Multimodality of the likelihood in the bivariate seemingly unrelated
  regressions model.
\newblock \emph{Biometrika}, 91\penalty0 (2):\penalty0 383--392, 2004.

\bibitem[Drton et~al.(2008)Drton, Sturmfels, and Sullivant]{drton2008lectures}
Mathias Drton, Bernd Sturmfels, and Seth Sullivant.
\newblock \emph{Lectures on algebraic statistics}, volume~39.
\newblock Springer Science \& Business Media, 2008.

\bibitem[Drton et~al.(2009)Drton, Eichler, and Richardson]{drton09a}
Mathias Drton, Michael Eichler, and Thomas~S. Richardson.
\newblock Computing maximum likelihood estimates in recursive linear models
  with correlated errors.
\newblock \emph{Journal of Machine Learning Research}, 10\penalty0
  (81):\penalty0 2329--2348, 2009.

\bibitem[Drton et~al.(2011)Drton, Foygel, and Sullivant]{drton2011global}
Mathias Drton, Rina Foygel, and Seth Sullivant.
\newblock Global identifiability of linear structural equation models.
\newblock \emph{The Annals of Statistics}, 39\penalty0 (2):\penalty0 865--886,
  2011.

\bibitem[Eigenmann et~al.(2017)Eigenmann, Nandy, and Maathuis]{eigenmann17}
Marco Eigenmann, Preetam Nandy, and Marloes~H. Maathuis.
\newblock Structure learning of linear {G}aussian structural equation models
  with weak edges.
\newblock In \emph{Proceedings of the 33rd Annual Conference on Uncertainty in
  Artificial Intelligence (UAI-17)}, 2017.

\bibitem[Fang and He(2020)]{fangida}
Zhuangyan Fang and Yangbo He.
\newblock {IDA} with background knowledge.
\newblock In \emph{Proceedings of the 36th Annual Conference on Uncertainty in
  Artificial Intelligence (UAI-20)}, 2020.

\bibitem[Guo and Perkovic(2021)]{guo2020minimal}
F.~Richard Guo and Emilija Perkovic.
\newblock Minimal enumeration of all possible total effects in a {Markov}
  equivalence class.
\newblock In \emph{Proceedings of The 24th International Conference on
  Artificial Intelligence and Statistics (AISTATS-21)}, volume 130, pages
  2395--2403, 2021.

\bibitem[Gupta et~al.(2020)Gupta, Lipton, and Childers]{gupta2020estimating}
Shantanu Gupta, Zachary~C. Lipton, and David Childers.
\newblock Estimating treatment effects with observed confounders and mediators.
\newblock \emph{arXiv preprint arXiv:2003.11991}, 2020.

\bibitem[Hansen(1982)]{hansen1982large}
Lars~Peter Hansen.
\newblock Large sample properties of generalized method of moments estimators.
\newblock \emph{Econometrica: Journal of the Econometric Society}, pages
  1029--1054, 1982.

\bibitem[Hauser and B{\"u}hlmann(2012)]{hauserBuehlmann12}
Alan Hauser and Peter B{\"u}hlmann.
\newblock Characterization and greedy learning of interventional {M}arkov
  equivalence classes of directed acyclic graphs.
\newblock \emph{Journal of Maching Learning Research}, 13:\penalty0 2409--2464,
  2012.

\bibitem[Henckel et~al.(2022)Henckel, Perkovi{\'c}, and
  Maathuis]{henckel2019graphical}
Leonard Henckel, Emilija Perkovi{\'c}, and Marloes~H. Maathuis.
\newblock Graphical criteria for efficient total effect estimation via
  adjustment in causal linear models.
\newblock \emph{Journal of the Royal Statistical Society: Series B (Statistical
  Methodology)}, to appear, 2022.

\bibitem[Horn and Johnson(2012)]{horn2012matrix}
Roger~A. Horn and Charles~R. Johnson.
\newblock \emph{Matrix Analysis}.
\newblock Cambridge University Press, 2 edition, 2012.

\bibitem[Hoyer et~al.(2008)Hoyer, Hyvarinen, Scheines, Spirtes, Ramsey,
  Lacerda, and Shimizu]{hoyer08}
Patrik~O. Hoyer, Aapo Hyvarinen, Richard Scheines, Peter~L. Spirtes, Joseph
  Ramsey, Gustavo Lacerda, and Shohei Shimizu.
\newblock Causal discovery of linear acyclic models with arbitrary
  distributions.
\newblock In \emph{Proceedings of the 24th Annual Conference on Uncertainty in
  Artificial Intelligence (UAI-08)}, pages 282--289, 2008.

\bibitem[Kalisch et~al.(2012)Kalisch, M\"achler, Colombo, Maathuis, and
  B\"uhlmann]{Kalisch:2012aa}
Markus Kalisch, Martin M\"achler, Diego Colombo, Marloes~H. Maathuis, and Peter
  B\"uhlmann.
\newblock Causal inference using graphical models with the {R} package {pcalg}.
\newblock \emph{Journal of Statistical Software}, 47\penalty0 (11):\penalty0
  1--26, 2012.

\bibitem[Koopmans and Reiers{\o}l(1950)]{koopmans1950identification}
Tjalling~C. Koopmans and Olav Reiers{\o}l.
\newblock The identification of structural characteristics.
\newblock \emph{The Annals of Mathematical Statistics}, 21\penalty0
  (2):\penalty0 165--181, 1950.

\bibitem[Kuroki and Cai(2004)]{kuroki2004selection}
Manabu Kuroki and Zhihong Cai.
\newblock Selection of identifiability criteria for total effects by using path
  diagrams.
\newblock In \emph{Proceedings of the 20th Conference on Uncertainty in
  Artificial Intelligence (UAI-04)}, pages 333--340, 2004.

\bibitem[Kuroki and Miyakawa(2003)]{kuroki2003covariate}
Manabu Kuroki and Masami Miyakawa.
\newblock Covariate selection for estimating the causal effect of control plans
  by using causal diagrams.
\newblock \emph{Journal of the Royal Statistical Society: Series B (Statistical
  Methodology)}, 65\penalty0 (1):\penalty0 209--222, 2003.

\bibitem[Kuroki and Nanmo(2020)]{kuroki2020variance}
Manabu Kuroki and Hisayoshi Nanmo.
\newblock Variance formulas for estimated mean response and predicted response
  with external intervention based on the back-door criterion in linear
  structural equation models.
\newblock \emph{AStA Advances in Statistical Analysis}, pages 1--19, 2020.

\bibitem[Lauritzen(1996)]{lauritzen1996graphical}
Steffen~L. Lauritzen.
\newblock \emph{Graphical Models}.
\newblock Oxford University Press, New York, 1996.

\bibitem[Liu(1999)]{liu1999some}
Jianzhou Liu.
\newblock Some {L{\"o}wner} partial orders of {Schur} complements and
  {Kronecker} products of matrices.
\newblock \emph{Linear algebra and its applications}, 291\penalty0
  (1-3):\penalty0 143--149, 1999.

\bibitem[Maathuis and Colombo(2015)]{maathuis2015generalized}
Marloes~H. Maathuis and Diego Colombo.
\newblock A generalized back-door criterion.
\newblock \emph{The Annals of Statistics}, 43\penalty0 (3):\penalty0
  1060--1088, 2015.

\bibitem[Maathuis et~al.(2009)Maathuis, Kalisch, and
  B{\"u}hlmann]{maathuis2009estimating}
Marloes~H. Maathuis, Markus Kalisch, and Peter B{\"u}hlmann.
\newblock Estimating high-dimensional intervention effects from observational
  data.
\newblock \emph{The Annals of Statistics}, 37\penalty0 (6A):\penalty0
  3133--3164, 2009.

\bibitem[Marbach et~al.(2009{\natexlab{a}})Marbach, Schaffter, Floreano, Prill,
  and Stolovitzky]{marbach2009dream4}
Daniel Marbach, Thomas Schaffter, Dario Floreano, Robert~J. Prill, and Gustavo
  Stolovitzky.
\newblock The {DREAM4} in-silico network challenge.
\newblock Draft, version 0.3
  \url{http://gnw.sourceforge.net/resources/DREAM4%20in%20silico%20challenge.pdf},
  2009{\natexlab{a}}.

\bibitem[Marbach et~al.(2009{\natexlab{b}})Marbach, Schaffter, Mattiussi, and
  Floreano]{marbach2009generating}
Daniel Marbach, Thomas Schaffter, Claudio Mattiussi, and Dario Floreano.
\newblock Generating realistic in silico gene networks for performance
  assessment of reverse engineering methods.
\newblock \emph{Journal of Computational Biology}, 16\penalty0 (2):\penalty0
  229--239, 2009{\natexlab{b}}.

\bibitem[Meek(1995)]{meek1995causal}
Christopher Meek.
\newblock Causal inference and causal explanation with background knowledge.
\newblock In \emph{Proceedings of the 11th Annual Conference on Uncertainty in
  Artificial Intelligence (UAI-95)}, pages 403--410, 1995.

\bibitem[Nandy et~al.(2017)Nandy, Maathuis, and
  Richardson]{nandy2017estimating}
Preetam Nandy, Marloes~H. Maathuis, and Thomas~S. Richardson.
\newblock Estimating the effect of joint interventions from observational data
  in sparse high-dimensional settings.
\newblock \emph{The Annals of Statistics}, 45\penalty0 (2):\penalty0 647--674,
  2017.

\bibitem[Pearl(1993)]{pearl1993bayesian}
Judea Pearl.
\newblock Comment: graphical models, causality and intervention.
\newblock \emph{Statistical Science}, 8\penalty0 (3):\penalty0 266--269, 1993.

\bibitem[Pearl(1995)]{pearl1995causal}
Judea Pearl.
\newblock Causal diagrams for empirical research.
\newblock \emph{Biometrika}, 82\penalty0 (4):\penalty0 669--688, 1995.

\bibitem[Pearl(2009)]{pearl2009causality}
Judea Pearl.
\newblock \emph{Causality}.
\newblock Cambridge University Press, Cambridge, 2nd edition, 2009.

\bibitem[Pearl and Verma(1995)]{pearl1995theory}
Judea Pearl and Thomas~S. Verma.
\newblock A theory of inferred causation.
\newblock In \emph{Studies in Logic and the Foundations of Mathematics}, volume
  134, pages 789--811. Elsevier, 1995.

\bibitem[Perkovi{\'c}(2020)]{perkovic20}
Emilija Perkovi{\'c}.
\newblock Identifying causal effects in maximally oriented partially directed
  acyclic graphs.
\newblock In \emph{Proceedings of the 36th Annual Conference on Uncertainty in
  Artificial Intelligence (UAI-20)}, 2020.

\bibitem[Perkovi{\'c} et~al.(2015)Perkovi{\'c}, Textor, Kalisch, and
  Maathuis]{perkovic2015complete}
Emilija Perkovi{\'c}, Johannes Textor, Markus Kalisch, and Marloes~H. Maathuis.
\newblock A complete generalized adjustment criterion.
\newblock In \emph{Proceedings of the 31st Annual Conference on Uncertainty in
  Artificial Intelligence (UAI-15)}, 2015.

\bibitem[Perkovi{\'c} et~al.(2017)Perkovi{\'c}, Kalisch, and
  Maathuis]{perkovic17}
Emilija Perkovi{\'c}, Markus Kalisch, and Marloes~H. Maathuis.
\newblock Interpreting and using {CPDAG}s with background knowledge.
\newblock In \emph{Proceedings of the 33rd Annual Conference on Uncertainty in
  Artificial Intelligence (UAI-17)}, 2017.

\bibitem[Perkovi\'c et~al.(2018)Perkovi\'c, Textor, Kalisch, and
  Maathuis]{perkovic18}
Emilija Perkovi\'c, Johannes Textor, Markus Kalisch, and Marloes~H. Maathuis.
\newblock Complete graphical characterization and construction of adjustment
  sets in {Markov} equivalence classes of ancestral graphs.
\newblock \emph{Journal of Machine Learning Research}, 18\penalty0
  (220):\penalty0 1--62, 2018.

\bibitem[Peters and B{\"u}hlmann(2014)]{peters2014identifiability}
Jonas Peters and Peter B{\"u}hlmann.
\newblock Identifiability of {Gaussian} structural equation models with equal
  error variances.
\newblock \emph{Biometrika}, 101\penalty0 (1):\penalty0 219--228, 2014.

\bibitem[{R Core Team}(2020)]{R:2020}
{R Core Team}.
\newblock \emph{R: A Language and Environment for Statistical Computing}.
\newblock R Foundation for Statistical Computing, Vienna, Austria, 2020.
\newblock URL \url{https://www.R-project.org/}.

\bibitem[Robins(1986)]{robins1986new}
James~M. Robins.
\newblock A new approach to causal inference in mortality studies with a
  sustained exposure period-application to control of the healthy worker
  survivor effect.
\newblock \emph{Mathematical Modelling}, 7:\penalty0 1393--1512, 1986.

\bibitem[Rothenh{\"a}usler et~al.(2018)Rothenh{\"a}usler, Ernest, and
  B{\"u}hlmann]{ernestroth2016}
Dominik Rothenh{\"a}usler, Jan Ernest, and Peter B{\"u}hlmann.
\newblock Causal inference in partially linear structural equation models.
\newblock \emph{The Annals of Statistics}, 46\penalty0 (6A):\penalty0 2904 --
  2938, 2018.

\bibitem[Rotnitzky and Smucler(2020)]{rotnitzky2019efficient}
Andrea Rotnitzky and Ezequiel Smucler.
\newblock Efficient adjustment sets for population average causal treatment
  effect estimation in graphical models.
\newblock \emph{Journal of Machine Learning Research}, 21\penalty0
  (188):\penalty0 1--86, 2020.

\bibitem[Sargan(1958)]{sargan1958estimation}
John~D. Sargan.
\newblock The estimation of economic relationships using instrumental
  variables.
\newblock \emph{Econometrica: Journal of the Econometric Society}, pages
  393--415, 1958.

\bibitem[Scheines et~al.(1998)Scheines, Spirtes, Glymour, Meek, and
  Richardson]{tetrad1998}
Richard Scheines, Peter Spirtes, Clark Glymour, Christopher Meek, and Thomas
  Richardson.
\newblock The {TETRAD} project: constraint based aids to causal model
  specification.
\newblock \emph{Multivariate Behavioral Research}, 33\penalty0 (1):\penalty0
  65--117, 1998.

\bibitem[Shimizu(2014)]{shimizu2014lingam}
Shohei Shimizu.
\newblock {LiNGAM}: {Non-Gaussian} methods for estimating causal structures.
\newblock \emph{Behaviormetrika}, 41\penalty0 (1):\penalty0 65--98, 2014.

\bibitem[Shimizu et~al.(2006)Shimizu, Hoyer, Hyv{\"a}rinen, and
  Kerminen]{shimizu2006linear}
Shohei Shimizu, Patrik~O. Hoyer, Aapo Hyv{\"a}rinen, and Antti Kerminen.
\newblock A linear {non-Gaussian} acyclic model for causal discovery.
\newblock \emph{Journal of Machine Learning Research}, 7\penalty0
  (72):\penalty0 2003--2030, 2006.

\bibitem[Shimizu et~al.(2011)Shimizu, Inazumi, Sogawa, Hyv{\"a}rinen, Kawahara,
  Washio, Hoyer, and Bollen]{shimizu2011directlingam}
Shohei Shimizu, Takanori Inazumi, Yasuhiro Sogawa, Aapo Hyv{\"a}rinen,
  Yoshinobu Kawahara, Takashi Washio, Patrik~O. Hoyer, and Kenneth Bollen.
\newblock {DirectLiNGAM}: A direct method for learning a linear {non-Gaussian}
  structural equation model.
\newblock \emph{Journal of Machine Learning Research}, 12:\penalty0 1225--1248,
  2011.

\bibitem[Shorack(2000)]{shorack2000probability}
Galen~R. Shorack.
\newblock \emph{Probability for Statisticians}.
\newblock Springer, 2000.

\bibitem[Shpitser et~al.(2010)Shpitser, VanderWeele, and
  Robins]{shpitser2010validity}
Ilya Shpitser, Tyler VanderWeele, and James~M. Robins.
\newblock On the validity of covariate adjustment for estimating causal
  effects.
\newblock In \emph{Proceedings of the 26th Annual Conference on Uncertainty in
  Artificial Intelligence (UAI-10)}, pages 527--536, 2010.

\bibitem[Smucler et~al.(2020)Smucler, Sapienza, and
  Rotnitzky]{smucler2020efficient}
Ezequiel Smucler, Facundo Sapienza, and Andrea Rotnitzky.
\newblock Efficient adjustment sets in causal graphical models with hidden
  variables.
\newblock \emph{arXiv preprint arXiv:2004.10521}, 2020.

\bibitem[Spirtes et~al.(2000)Spirtes, Glymour, and
  Scheines]{spirtes2000causation}
Peter Spirtes, Clark Glymour, and Richard Scheines.
\newblock \emph{Causation, Prediction, and Search}.
\newblock MIT Press, Cambridge, MA, 2nd edition, 2000.

\bibitem[Strotz and Wold(1960)]{strotz1960recursive}
Robert~H. Strotz and H.~O.~A. Wold.
\newblock Recursive vs. nonrecursive systems: An attempt at synthesis (part {I}
  of a triptych on causal chain systems).
\newblock \emph{Econometrica}, 28\penalty0 (2):\penalty0 417--427, 1960.

\bibitem[Sullivant et~al.(2010)Sullivant, Talaska, and
  Draisma]{sullivant2010trek}
Seth Sullivant, Kelli Talaska, and Jan Draisma.
\newblock Trek separation for {Gaussian} graphical models.
\newblock \emph{The Annals of Statistics}, 38\penalty0 (3):\penalty0
  1665--1685, 2010.

\bibitem[Tsiatis(2006)]{tsiatis2007semiparametric}
Anastasios Tsiatis.
\newblock \emph{Semiparametric Theory and Missing Data}.
\newblock Springer, New York, 2006.

\bibitem[van~der Laan and Rose(2011)]{van2011targeted}
Mark~J. van~der Laan and Sherri Rose.
\newblock \emph{Targeted Learning: Causal Inference for Observational and
  Experimental Data}.
\newblock Springer, New York, 2011.

\bibitem[van~der Vaart(2000)]{van2000asymptotic}
Aad~W. van~der Vaart.
\newblock \emph{Asymptotic Statistics}.
\newblock Cambridge University Press, 2000.

\bibitem[Wang et~al.(2017)Wang, Solus, Yang, and Uhler]{wang2017permutation}
Yuhao Wang, Liam Solus, Karren~Dai Yang, and Caroline Uhler.
\newblock Permutation-based causal inference algorithms with interventions.
\newblock In \emph{Advances in Neural Information Processing Systems 30}, pages
  5822--5831. 2017.

\bibitem[Witte et~al.(2020)Witte, Henckel, Maathuis, and
  Didelez]{witte2020efficient}
Janine Witte, Leonard Henckel, Marloes~H Maathuis, and Vanessa Didelez.
\newblock On efficient adjustment in causal graphs.
\newblock \emph{Journal of Machine Learning Research}, 21\penalty0
  (246):\penalty0 1--45, 2020.

\bibitem[Wright(1921)]{wright1921correlation}
Sewall Wright.
\newblock Correlation and causation.
\newblock \emph{Journal of Agricultural Research}, 20:\penalty0 557--585, 1921.

\bibitem[Wright(1934)]{wright1934method}
Sewall Wright.
\newblock The method of path coefficients.
\newblock \emph{The Annals of Mathematical Statistics}, 5\penalty0
  (3):\penalty0 161--215, 1934.

\end{thebibliography}
\end{document}